\documentclass[10pt,a4paper]{article}

\usepackage[T1]{fontenc}
\usepackage[a4paper,margin=2.35cm]{geometry}
\usepackage{microtype}

\usepackage{epsf,epsfig,amsfonts,amsgen,amsmath,amstext,amsbsy,amsopn,amsthm,cases,listings}
\usepackage{ebezier,eepic}
\usepackage{xcolor}
\usepackage{multirow}
\usepackage{epstopdf}
\usepackage{graphicx}
\usepackage{pgf,tikz}
\usepackage{pgfplots}
\usepackage{mathrsfs}
\usepackage[bottom]{footmisc}
\usepackage{enumerate}
\usepackage{enumitem}
\usepackage[titletoc]{appendix}
\usepackage{booktabs}
\usepackage{url}
\usepackage{mathtools}
\usepackage{authblk}
\usepackage{amssymb}
\usepackage{empheq}
\usepackage{dsfont}
\usepackage{longtable}
\usepackage{float}

\usepackage{subfigure}
\pgfplotsset{compat=1.18}
\usetikzlibrary{arrows}
\usepackage{aligned-overset}
\usepackage{bm}
\usepackage{bbm}
\usepackage{titlesec}
\usepackage[
    backref=page,
    hyperfootnotes=false,
    colorlinks=true,
    linkcolor=blue!45!black,
    citecolor=blue!45!black,
    urlcolor=blue!45!black
]{hyperref}

\usepackage{algorithm}
\usepackage{algpseudocode}
\usepackage{multicol}

\usepackage[thinc]{esdiff}

\allowdisplaybreaks[1]
\setlist{itemsep=0.15em,topsep=0.25em,parsep=0pt}
\definecolor{uuuuuu}{rgb}{0.27,0.27,0.27}
\definecolor{sqsqsq}{rgb}{0.1255,0.1255,0.1255}

\theoremstyle{plain}
\newtheorem{definition}{Definition}[section]
\newtheorem{theorem}[definition]{Theorem}
\newtheorem{lemma}[definition]{Lemma}
\newtheorem{proposition}[definition]{Proposition}
\newtheorem{corollary}[definition]{Corollary}
\newtheorem{conjecture}[definition]{Conjecture}
\newtheorem{claim}[definition]{Claim}
\theoremstyle{definition}

\newtheorem{fact}[definition]{Fact}

\newcommand{\norm}[1]{\lVert#1\rVert}

\titleformat{\section}{\normalfont\large\bfseries}{\thesection}{0.75em}{}
\titleformat{\subsection}{\normalfont\normalsize\bfseries}{\thesubsection}{0.75em}{}
\titleformat{\subsubsection}{\normalfont\normalsize\itshape}{\thesubsubsection}{0.75em}{}
\titlespacing*{\section}{0pt}{0.95\baselineskip}{0.45\baselineskip}
\titlespacing*{\subsection}{0pt}{0.75\baselineskip}{0.3\baselineskip}
\titlespacing*{\subsubsection}{0pt}{0.6\baselineskip}{0.25\baselineskip}

\makeatletter
\let\oldthebibliography\thebibliography
\renewcommand{\thebibliography}[1]{%
  \oldthebibliography{#1}%
  \setlength{\itemsep}{0.15em}%
  \setlength{\parskip}{0pt}%
}
\newcommand{\plainfootnotetext}[1]{%
  \begingroup
  \renewcommand\thefootnote{}%
  \long\def\@makefntext##1{\noindent ##1}%
  \footnotetext{#1}%
  \addtocounter{footnote}{-1}%
  \endgroup
}
\makeatother
\begin{document}
\title{\bf\Large Vertex-colored Tur\'{a}n theorems with applications in extremal hypergraph problems}
\author[1]{Wanfang~Chen}
\author[2]{Jinghua~Deng}
\author[3]{Jianfeng~Hou}
\author[1]{Xizhi~Liu}
\author[3]{Yixiao~Zhang}
\affil[1]{School of Mathematical Sciences, University of Science and Technology of China, Hefei, 230026, China}
\affil[2]{School of Data Science and Engineering, South China Normal University, Shanwei, 516600, China}
\affil[3]{Center for Discrete Mathematics, Fuzhou University, Fuzhou, 350108, China}
\date{\today}
\maketitle
\begin{abstract}
Balogh, Clemen, and Lidick\'{y}~\cite{BCL22b} proved that the $\ell_{2}$-norm Tur\'{a}n problem for $K_{5}^{3}$ is asymptotically solved by the balanced bipartite construction, and they further conjectured that this construction is uniquely extremal for all sufficiently large $n$. 
We confirm this conjecture. 
We also determine exactly the maximum number of cliques in an $n$-vertex $K_{5}^{3}$-free $3$-uniform hypergraph for all sufficiently large $n$, thereby verifying the corresponding case of a conjecture of Frankl, Gryaznov, and Talebanfard~\cite{FGT22}. 

The main ingredients are Tur\'{a}n-type theorems for vertex-colored graphs forbidding balanced cliques, including an edge bound, an $\ell_{2}$-norm bound, and a sharp crossing-triangle theorem in the two-colored balanced $K_{4}$-free case. We also use a local modification procedure within the stability method. This reduces the exact hypergraph problems to proving that the relevant objective function increases under suitable local changes near the bipartite construction. 
\end{abstract}

\plainfootnotetext{\textit{Keywords:} vertex-colored graphs, hypergraph Tur\'{a}n problems, stability method \\
\textit{MSC2020:} 05C35, 05C65, 05D05 \\
\textit{Email:} \texttt{a372959313@gmail.com}, \texttt{Jinghua\_deng@163.com}, \texttt{jfhou@fzu.edu.cn}, \texttt{liuxizhi@ustc.edu.cn}, \texttt{fzuzyx@gmail.com}}

\section{Introduction}\label{SEC:Introduction}

Given an integer $r \ge 2$, an $r$-uniform hypergraph (henceforth an $r$-graph) $\mathcal H$ is a collection of $r$-subsets of some set $V$. We call $V$ the vertex set of $\mathcal H$ and denote it by $V(\mathcal H)$. When $V$ is understood, we often identify $\mathcal H$ with its edge set.
Given a family $\mathcal F$ of $r$-graphs, an $r$-graph $\mathcal H$ is $\mathcal F$-free if it does not contain any member of $\mathcal F$ as a subgraph. A central problem in extremal combinatorics is to determine extremal properties of $\mathcal F$-free $r$-graphs on $n$ vertices. In particular, the Tur\'{a}n problem asks for the Tur\'{a}n number
\begin{align*}
    \mathrm{ex}(n,\mathcal F)
    \coloneqq \max\left\{|\mathcal H| : |V(\mathcal H)|=n \text{ and } \mathcal H \text{ is } \mathcal F\text{-free}\right\}.
\end{align*}
For $r=2$, the asymptotic behavior of $\mathrm{ex}(n,\mathcal F)$ is described by the Erd{\H{o}}s--Stone--Simonovits theorem~\cite{ES46} (see also~\cite{ES66}). 
Denote by $K_\ell^r$ the complete $r$-graph on $\ell$ vertices, with the superscript omitted when $r=2$. 
The classical theorem of Tur\'{a}n~\cite{Tur41} determines the exact value of $\mathrm{ex}(n,K_\ell)$ for every $\ell\ge 3$, with the case $\ell=3$ proved earlier by Mantel~\cite{Mantel07}. 
In contrast, for $r\ge 3$, determining $\mathrm{ex}(n,K_\ell^r)$ even asymptotically is notoriously difficult in general. 
Erd{\H{o}}s~\cite{Erd81} offered \$500 for determining the Tur\'{a}n density of $K_\ell^r$ for any single pair $\ell,r$ with $\ell > r \ge 3$. This prize remains unclaimed. We refer the reader to~\cite{BCL22a, Caen94Survey, Fur91, Kee11} for further background on hypergraph Tur\'{a}n theory.

In this paper, we study two exact extremal problems for $K_5^3$-free $3$-graphs: the Tur\'{a}n problem in the $\ell_2$-norm, where the number of edges is replaced by the sum of squared codegrees, and the problem of maximizing the number of cliques. 
Following Balogh, Clemen, and Lidick\'{y}~\cite{BCL22b}, for an $r$-graph $\mathcal{H}$, we write
\begin{align*}
    \norm{\mathcal H}_{2}
    \coloneqq \sum_{e\in \binom{V(\mathcal{H})}{r-1}} d_{\mathcal H}(e)^{2},
\end{align*}
where $d_{\mathcal H}(e)$ denotes the codegree of the $(r-1)$-set $e$. The corresponding extremal quantity and density are
\begin{align*}
    \mathrm{ex}_{\ell_2}(n,\mathcal F)
    \coloneqq \max\left\{\norm{\mathcal H}_{2} : |V(\mathcal H)|=n \text{ and } \mathcal H \text{ is } \mathcal F\text{-free}\right\}, 
\end{align*}
and
\begin{align*}
    \pi_{\ell_2}(\mathcal F)
    \coloneqq \lim_{n\to\infty}\frac{\mathrm{ex}_{\ell_2}(n,\mathcal F)}{\binom{n}{r-1}(n-r+1)^2},
\end{align*}
respectively. 
The existence of this limit follows from the standard averaging argument. 
Recently, Balogh, Clemen, and Lidick\'{y}~\cite{BCL22b} established the $\ell_2$-norm Tur\'{a}n density of the tetrahedron $K_4^3$ to be $1/3$. 
Subsequently, Bodn\'{a}r et al.~\cite{BCD25} determined the exact $\ell_2$-norm Tur\'{a}n number of $K_4^3$ for all sufficiently large $n$. 

For the $\ell_2$-norm Tur\'{a}n problem of $K_{5}^{3}$, the relevant extremal construction is the balanced complete bipartite $3$-graph. 
Given two disjoint sets $V_1$ and $V_2$, let $\mathbb{B}[V_1,V_2]$ denote the $3$-graph on $V_1\cup V_2$ consisting of all triples that intersect both $V_1$ and $V_2$. 
When $V_1\cup V_2=[n]$ is a balanced partition of $[n]$, we write $\mathbb{B}_n$ for $\mathbb{B}[V_1,V_2]$. 
The construction $\mathbb{B}_n$ shows that $\pi_{\ell_2}(K_5^3)\ge 5/8$. 
Using flag algebras, Balogh, Clemen, and Lidick\'{y}~\cite{BCL22b} proved the matching upper bound $\pi_{\ell_2}(K_5^3)\le 5/8$, establishing the following theorem. 

\begin{theorem}[{\cite[Theorem 1.2]{BCL22b}}]\label{THM:K53-L2-density}
It holds that $\pi_{\ell_2}(K_5^3)=5/8$.
\end{theorem}

They further conjectured that the corresponding exact statement holds for all sufficiently large $n$. 

\begin{conjecture}[{\cite[Conjecture 1.4]{BCL22b}}]\label{CONJ:K53_l2_norm}
    There exists $n_0$ such that for every $n \ge n_0$,
    \begin{align*}
        \mathrm{ex}_{\ell_2}(n,K_5^3)=\norm{\mathbb{B}_n}_2, 
    \end{align*}
    and $\mathbb{B}_n$ is the unique $K_5^3$-free $3$-graph on $n$ vertices $\mathcal{H}$ satisfying $\norm{\mathcal{H}}_2 = \mathrm{ex}_{\ell_2}(n,K_5^3)$. 
\end{conjecture}

Our first main result confirms this conjecture.

\begin{theorem}\label{THM:L2-exact-K53}
   There exists an integer $n_0$ such that for all $n \ge n_0$, every $n$-vertex $K_5^3$-free $3$-graph $\mathcal{H}$ satisfies 
   \begin{align*}
       \norm{\mathcal{H}}_2 \le \norm{\mathbb{B}_n}_2,  
   \end{align*}
   and equality holds if and only if $\mathcal{H} \cong \mathbb{B}_n$. 
\end{theorem}

We next turn to the clique-counting problem in hypergraphs. 
Given an $r$-graph $\mathcal{H}$, a clique in $\mathcal{H}$ is a subset $S \subseteq V(\mathcal{H})$ such that either $|S| < r$, or $|S| \ge r$ and every $r$-subset of $S$ belongs to $\mathcal{H}$. 
We write $k(\mathcal{H})$ for the number of cliques in $\mathcal{H}$. 
For an integer $t \ge 0$, define
\begin{align*}
    V(x,t) \coloneqq \sum_{j=0}^{t}\binom{x}{j}.   
\end{align*}
In particular, when $x$ is a positive integer, $V(x,t)$ represents the size of the Hamming ball of radius $t$ in the $x$-dimensional hypercube. 
Frankl, Gryaznov, and Talebanfard~\cite{FGT22} proposed the following conjecture. 

\begin{conjecture}[{\cite[Conjecture 14]{FGT22}}]\label{CONJ:cliques-exact-K53}
    Suppose that $\mathcal{H}$ is an $n$-vertex $r$-graph containing no clique of size $\ell+1$. 
    Then the total number of cliques in $\mathcal{H}$ is bounded by
    \begin{align*}
        k(\mathcal{H})
        \le V\!\left(\frac{(r-1)n}{\ell},r-1\right)^{\frac{\ell}{r-1}}. 
    \end{align*}
    Furthermore, when $r-1 \mid \ell$ and $\ell \mid (r-1)n$, the unique extremal case is the $n$-vertex $\frac{\ell}{r-1}$-partite $r$-graph whose edge set consists of all $r$-sets that intersect each part in at most $r-1$ vertices. 
\end{conjecture}

When $(r,\ell)=(3,4)$, the corresponding construction is the balanced complete bipartite $3$-graph $\mathbb{B}_n$. 
Thus, in our setting, the clique-counting problem again points to the same bipartite construction. 
Bodn\'{a}r~\cite{Bod23} proved an asymptotic version of the corresponding statement, and our second theorem gives the exact result for all sufficiently large $n$. 

\begin{theorem}\label{THM:cliques-exact-K53}
   There exists an integer $n_0$ such that for all $n \ge n_0$, every $n$-vertex $K_5^3$-free $3$-graph $\mathcal{H}$ satisfies 
   \begin{align*}
       k(\mathcal{H}) \le k(\mathbb{B}_n),  
   \end{align*}
   and equality holds if and only if $\mathcal{H} \cong \mathbb{B}_n$. 
\end{theorem}

The main auxiliary input comes from Tur\'{a}n-type problems in vertex-colored graphs. 
The point is that, once a $K_{5}^{3}$-free $3$-graph is close to $\mathbb{B}_n$, the relevant link graphs carry a natural two-part vertex coloring induced by the underlying bipartition. 
Moreover, every balanced copy of $K_4$ in such a link graph must be accounted for by a missing triple of the bipartite template. Otherwise it would extend to a copy of $K_5^3$. 
Thus these link graphs have few balanced copies of $K_4$, and the required local estimates reduce to Tur\'{a}n-type statements for vertex-colored graphs. 
This leads us to study the following general setup. 
Let $G$ be a graph on vertex set $V$, and let $V = V_1 \cup \cdots \cup V_s$ be a fixed partition. 
For integers $1 \le q \le s$ and $t \ge 1$, we call a clique $K \subseteq G$ a $(q,t)$-balanced clique if there exists an index set $I \subseteq [s]$ with $|I|=q$ such that
\begin{align*}
    |V(K) \cap V_i| = t \quad\text{for every } i\in I, 
    \quad\text{and}\quad
    K\cong K_{qt}. 
\end{align*}
When $q=s$, this is simply a balanced $K_{st}$. 

We establish several Tur\'{a}n-type results for vertex-colored graphs that are useful in extremal problems for hypergraphs. 
In particular, we prove an edge bound and an $\ell_2$-norm bound for graphs with no balanced clique, as well as a sharp theorem on crossing triangles in the two-colored balanced $K_4$-free case. 
The precise statements are given in Section~\ref{SEC:Colored-Turan}. 

This paper is organized as follows. 
In Section~\ref{SEC:Colored-Turan}, we prove Tur\'{a}n-type results for vertex-colored graphs. 
In Section~\ref{SEC:Common-Structure}, we collect the common structural properties of $K_5^3$-free $3$-graphs that are close to the bipartite construction $\mathbb{B}_n$. 
These properties will be used in both exact arguments. 
In Section~\ref{SEC:K53-L2norm}, we combine the vertex-colored graph results with the common structural lemmas to prove Theorem~\ref{THM:L2-exact-K53}. 
In Section~\ref{SEC:K53-count-cliques}, we prove Theorem~\ref{THM:cliques-exact-K53}. 
Finally, we give some concluding remarks in Section~\ref{SEC:Concluding-remarks}. 

Both exact proofs follow the same general mechanism. 
A stability theorem first gives a partition with respect to which the hypergraph is close to $\mathbb{B}_n$. 
The common structural results, together with objective-specific local degree estimates, then give the required local control of bad edges and missing edges. 
Finally, we apply a local modification procedure: for a suitable pair, we replace the bad edges containing this pair by the corresponding missing edges from the bipartite template. 
The key point is that this operation strictly increases the objective function, namely the $\ell_2$-norm in Section~\ref{SEC:K53-L2norm} and the clique-counting objective in Section~\ref{SEC:K53-count-cliques}. 
This rules out all bad edges. The remaining conclusion then follows from extremality and from the fact that the complete bipartite construction is maximized by a balanced partition.


\section{Tur\'{a}n-type theorems in vertex-colored graphs}\label{SEC:Colored-Turan}
In this section, we first state and then prove several Tur\'{a}n-type results for vertex-colored graphs. 
Let $G$ be a graph on vertex set $V$, and let $V \coloneqq V_1 \cup \cdots \cup V_s$ be a fixed partition of $V$ (equivalently, an $s$-coloring of the vertices of $G$). 
We recall that, for integers $1 \le q \le s$ and $t \ge 1$, a clique $K \subseteq G$ is called a $(q,t)$-balanced clique if there exists an index set $I\subseteq [s]$ with $|I|=q$ such that 
\begin{align*}
|V(K)\cap V_i| = t \quad\text{for every } i\in I,
    \quad\text{and}\quad
    K \cong K_{qt}. 
\end{align*}
When $q = s$, we simply call $K$ a balanced $K_{st}$.

Throughout this section, we assume that
\begin{align*}
    |V_i|=n \quad\text{for all } i \in [s].
\end{align*}
When cyclic notation is used, the relevant indices are taken modulo the specified range.

For a vertex set $S \subseteq V$, we use $G[S]$ to denote the subgraph of $G$ induced by $S$. 
For pairwise disjoint sets $S_1, \ldots, S_k \subseteq V$, we write $G[S_1, \ldots, S_k]$ for the induced $k$-partite subgraph of $G$ with parts $S_1, \ldots, S_k$. 

\subsection{Extremal results}\label{SUBSEC:Colored-Turan-Extremal}

In this subsection, we state the vertex-colored extremal results used later. 
We first give an edge bound for graphs with no balanced $K_{st}$, then an $\ell_2$-norm bound in the case $t=2$, and finally a sharp bound on crossing triangles in the two-colored balanced $K_4$-free case. 

\subsubsection{Edge Bound}\label{SUBSUBSEC:Colored-Turan-edge}

We first introduce the extremal construction.
Let $V_1 = S_1 \cup \cdots \cup S_{t-1}$ be a partition of the set $V_1$ such that the sizes of $S_i$ differ by at most one.
We define $\Lambda[S_1, \ldots, S_{t-1}; V_{2},\ldots, V_{s}]$ to be the graph on the vertex set $V$ with edge set
\begin{align*}
    \Lambda[S_1, \ldots, S_{t-1}; V_{2},\ldots, V_{s}] \coloneqq \binom{V}{2} \setminus \bigcup_{i\in [t-1]}\binom{S_i}{2}.
\end{align*}
Observe that the graph $\Lambda$ is balanced $K_{st}$-free, as the part $V_1$ does not contain a clique of size $t$.
A straightforward calculation gives
\begin{align*}
    |\Lambda[S_1, \ldots, S_{t-1}; V_{2},\ldots, V_{s}]|
     = \binom{sn}{2} - \binom{n}{2} + t_{t-1}(n) 
     = \frac{(sn)^2}{2} - \frac{n^2}{2(t-1)} + o(n^{2}),
\end{align*}
where $t_{t-1}(n)$ denotes the number of edges in the balanced $(t-1)$-partite Tur\'{a}n graph on $n$ vertices.

\begin{theorem}\label{THM:s-colored-Kst-L1-reduced}
    Let $s,t\ge 2$. If $G$ is balanced $K_{st}$-free, then
    \begin{align*}
        |G| \le  \frac{(sn)^2}{2}-\frac{n^2}{2(t-1)} + s^2 n.
    \end{align*}
\end{theorem}

\begin{corollary}\label{COR:2-colored-K4-L1-reduced}
    Suppose that $t = s = 2$ and $|V_i|= n$ for $i \in [2]$.
    If $G$ is balanced $K_4$-free, then
    \begin{align*}
        |G| \le  \frac{3n^2}{2} + 4n.
    \end{align*}
\end{corollary}

\subsubsection{\texorpdfstring{$\ell_2$}{l2}-Norm Bound}\label{SUBSUBSEC:Colored-Turan-L2}

We now turn to the $\ell_2$-norm. In the case $t=2$, the graph $\Lambda[V_{1}; V_{2},\ldots, V_{s}]$ has $V_1$ independent.
The subgraph induced by \(V_2\cup\cdots\cup V_s\) is complete, and all edges between \(V_1\) and \(V_2\cup\cdots\cup V_s\) are present. Hence
\begin{align*}
    \norm{\Lambda[V_{1}; V_{2},\ldots, V_{s}]}_{2} = (s-1)n \cdot (sn-1)^{2} + n \cdot (sn-n)^{2}.
\end{align*}
This leads to the following upper bound for balanced $K_{2s}$-free graphs.

\begin{theorem}\label{THM:s-colored-K2s-L2-norm-reduced}
    Let $s\ge 2$. If $G$ is balanced $K_{2s}$-free, then 
    \begin{align*}
        \norm{G}_{2} \le (s^{2}(s-1) + (s-1)^{2}) n^{3} + 4s^{3}n^{2}. 
    \end{align*}
\end{theorem}

In the case $s=2$, we obtain the following corollary.

\begin{corollary}\label{COR:2-colored-K4-L2-norm-reduced}
    Suppose that $s = 2$ and $|V_i|= n$ for $i \in [2]$.
    If $G$ is balanced $K_4$-free, then 
    \begin{align*}
        \norm{G}_{2} \le 5n^3 + 32n^2. 
    \end{align*}
\end{corollary}

\subsubsection{Crossing Triangles}\label{SUBSUBSEC:Colored-Turan-Crossing-Triangles}

Finally, we consider crossing triangles in the case $s=2$. Let $\mathcal{N}_{\mathrm{cr}}(K_{3}, G)$ denote the number of crossing triangles in $G$, that is, triangles containing vertices from both $V_1$ and $V_2$. For the graph $\Lambda[V_{1}; V_{2}]$, we have
\begin{align*}
    \mathcal{N}_{\mathrm{cr}}(K_{3}, \Lambda) = n \binom{n}{2}.
\end{align*}
The next theorem shows that this is best possible, and the following proposition gives the corresponding stability statement.

\begin{theorem}\label{THM:2-colored-K4-triangle-reduced}
    Suppose that $s = 2$, $n\ge 3$, and $|V_i|= n$ for $i \in [2]$.
    If $G$ is a balanced $K_{4}$-free graph, then
    \begin{align*}
        \mathcal{N}_{\mathrm{cr}}(K_{3}, G)\le n\binom{n}{2}.
    \end{align*}
\end{theorem}

\begin{proposition}\label{PROP:2-colored-K4-triangle-stability-sym}
    For every $\eta > 0$, there exist $\varepsilon_{\ref{PROP:2-colored-K4-triangle-stability-sym}} \coloneqq \varepsilon_{\ref{PROP:2-colored-K4-triangle-stability-sym}}(\eta) > 0$ and $N_{\ref{PROP:2-colored-K4-triangle-stability-sym}} \coloneqq N_{\ref{PROP:2-colored-K4-triangle-stability-sym}}(\eta)$ such that the following holds for all $n \ge N_{\ref{PROP:2-colored-K4-triangle-stability-sym}}$.
    Suppose that $G$ is a balanced $K_4$-free graph on the partition $V_1 \cup V_2$ with $|V_1|=|V_2|=n$, and
    \begin{align*}
        \mathcal{N}_{\mathrm{cr}}(K_{3}, G)\ge n\binom{n}{2}-\varepsilon_{\ref{PROP:2-colored-K4-triangle-stability-sym}}n^3.
    \end{align*}
    Then $G$ can be transformed into either $\Lambda[V_1;V_2]$ or $\Lambda[V_2;V_1]$ by changing at most $\eta n^2$ edges.
\end{proposition}
\subsection{Local Symmetrization}\label{SUBSEC:Local-Symmetrization}
In this subsection, we record a slight variant of Zykov symmetrization~\cite{Zyk49}. It
will be used in the proofs of the edge bound and the $\ell_2$-norm bound, and it will also
provide a convenient monotonicity statement for crossing triangles.
Given nonadjacent vertices $u,v \in V$, denote by $G_{u \to v}$ the graph obtained from $G$ by
symmetrizing $u$ to $v$, that is, 
\begin{align*}
    G_{u \to v}
    \coloneqq \big( G \setminus \left\{uw \colon w\in N_{G}(u)\right\} \big) \cup \left\{ uw \colon w\in N_{G}(v) \right\}. 
\end{align*}
We say that $u$ and $v$ are equivalent, denoted $u \sim v$, if $u$ and $v$ lie in the same part $V_i$ for some $i\in [s]$ and $N_G(u) = N_G(v)$ (and hence $\{u,v\} \notin G$).
For $i\in[s]$ and $v\in V_i$, write
\begin{align*}
    [v]\coloneqq \{u\in V_i \colon u \sim v\}
\end{align*}
for the equivalence class of $v$ in $V_i$ (with respect to $\sim$ in $G$).

We say that $G$ is locally symmetrized if, for each $i\in[s]$, any two nonadjacent vertices in $V_i$ are equivalent.

Suppose that $u$ and $v$ lie in the same part and $uv\notin G$. Then there are no edges between the classes $[u]$ and $[v]$. Let $G_{[u] \to [v]}$ be the graph obtained by symmetrizing every vertex in $[u]$ to the neighborhood pattern of a vertex in $[v]$.

Each such step reduces the number of equivalence classes by one, so after finitely many steps we obtain a locally symmetrized graph. Moreover, if $G$ is balanced $K_{st}$-free, then so is $G_{[u] \to [v]}$. Indeed, suppose that $G_{[u] \to [v]}$ contains a balanced $K_{st}$, say $K$. If $K$ avoids $[u]$, then $K$ already appears in $G$, a contradiction. Hence $K$ contains a vertex from $[u]$. Since $[u]$ remains independent and has no edges to $[v]$ after the symmetrization, $K$ contains exactly one vertex from $[u]$ and no vertex from $[v]$. Replacing this vertex by $v$ gives a balanced $K_{st}$ in $G$, a contradiction.
The next two propositions show that, under a suitable such symmetrization, neither the number of edges nor the $\ell_2$-norm nor the number of crossing triangles decreases.
\begin{proposition}[{\cite[Proposition~2.1]{CILLP24}}]\label{PROP:Zykov-symmetry-ineq}
    Let $u$ and $v$ be nonadjacent vertices lying in the same part.
    \begin{align*}
        |G| \le \max\left\{|G_{[u] \to [v]}|,~|G_{[v] \to [u]}|\right\}, \quad \norm{G}_{2} \le \max\left\{\norm{G_{[u] \to [v]}}_{2},~\norm{G_{[v] \to [u]}}_2\right\}. 
    \end{align*}
\end{proposition}
\begin{proposition}\label{PROP:Zykov-symmetry-triangle-ineq}
    Let $u$ and $v$ be nonadjacent vertices lying in the same part.
    \begin{align*}
        \mathcal{N}_{\mathrm{cr}}(K_{3}, G) \le \max\left\{\mathcal{N}_{\mathrm{cr}}(K_{3}, G_{[u] \to [v]}),~\mathcal{N}_{\mathrm{cr}}(K_{3}, G_{[v] \to [u]})\right\}.
    \end{align*}
\end{proposition}

\begin{proof}[Proof of Proposition~\ref{PROP:Zykov-symmetry-triangle-ineq}]
    For a vertex $x\in V$, let $\tau(x)$ denote the number of crossing triangles containing $x$.
    Since vertices in the same equivalence class have the same neighborhood, $\tau(x)$ is constant on each of the classes $[u]$ and $[v]$. Let $\tau_u$ and $\tau_v$ denote these values, and assume without loss of generality that $\tau_u \le \tau_v$.
    Since there are no edges between $[u]$ and $[v]$, and since $[u]$ remains independent after the symmetrization, every crossing triangle affected by the passage from $G$ to $G_{[u]\to[v]}$ contains exactly one vertex of $[u]$. For each such vertex, the number of crossing triangles changes from $\tau_u$ to $\tau_v$. Therefore
    \begin{align*}
        \mathcal{N}_{\mathrm{cr}}(K_{3}, G_{[u] \to [v]}) - \mathcal{N}_{\mathrm{cr}}(K_{3}, G)
        = |[u]|(\tau_v-\tau_u)\ge 0.
    \end{align*}
    This proves the proposition.
\end{proof}
Therefore, in the proofs of Theorems~\ref{THM:s-colored-Kst-L1-reduced}, \ref{THM:s-colored-K2s-L2-norm-reduced}, and~\ref{THM:2-colored-K4-triangle-reduced}, we may assume that $G$ is locally symmetrized.
The following fact follows directly from the definition of locally symmetrized.
\begin{fact}\label{FACT:symmetrized-G}
    Suppose that $G$ is locally symmetrized. Then the following statements hold. 
    \begin{enumerate}[label=(\roman*), ref=(\roman*)]
        \item\label{FACT:symmetrized-G-a} Each equivalence class is an independent set in $G$.
        \item\label{FACT:symmetrized-G-b} For each $i \in [s]$ and every pair of distinct equivalence classes $[x], [y] \subseteq V_i$, the bipartite graph between $[x]$ and $[y]$ is complete. 
    \end{enumerate}
\end{fact}
\subsection{Proof of the Edge Bound}\label{SUBSEC:Proof-Edge-Bound}
\begin{proof}[Proof of Theorem~\ref{THM:s-colored-Kst-L1-reduced}]
    Let $G$ be an extremal balanced $K_{st}$-free graph with respect to the partition $(V_1,\ldots,V_s)$, where $|V_i|=n$ for all $i\in[s]$.
    Suppose to the contrary that 
    \begin{align*}
        |G| >  \binom{sn}{2} - \binom{n}{2} + t_{t-1}(n) + s^2 n.
    \end{align*}

    We first show that extremality forces every transversal choice of one vertex from each part to have large total degree.
    \begin{claim}\label{CLAIM:degree-sum-bound}
        For any choice of $s$ vertices $(v_1,\ldots,v_s) \in V_1\times \cdots \times V_s$, we have
        \begin{align*}
            \sum_{i\in[s]} d_G(v_i) > s^2 n - \frac{2n}{t}.
        \end{align*}
    \end{claim}
\begin{proof}[Proof of Claim~\ref{CLAIM:degree-sum-bound}]
    We proceed by contradiction. Suppose there exists a tuple $(v_1,\ldots,v_s)$ such that $\sum_{i=1}^s d_G(v_i) \le s^2 n - \frac{2n}{t}$. 
    Let us label vertices in $V_1, \ldots,V_s$ by $V_i = \{v_i^1, \ldots, v_i^{n}\}$, for all $i\in [s]$.
     Without loss of generality, we may assume the tuple yielding the contradiction is $(v_1,\ldots,v_s) = (v_1^1, \ldots, v_s^1)$. 
    Since 
    \begin{align*}
        \sum_{i\in [n]} \big( d_{G}(v_{1}^i) + \cdots + d_{G}(v_{s}^i) \big)
        = 2|G|
        &> 2\left(\binom{sn}{2} - \binom{n}{2} + t_{t-1}(n) + s^2 n\right) \\
        &= s^2n^2-sn-n^2+n+2t_{t-1}(n)+2s^2n.
    \end{align*}
    Recall that the balanced $(t-1)$-partite Tur\'{a}n graph on $n$ vertices has minimum degree
    \(n-\left\lceil n/(t-1)\right\rceil\). In particular, its average degree satisfies
    \begin{align*}
        \frac{2t_{t-1}(n)}{n} \ge n - \left\lceil \frac{n}{t-1}\right\rceil \ge \frac{t-2}{t-1}n - 1.
    \end{align*}
    
    Hence, by the pigeonhole principle, there exists a tuple $(v_1^{i_0}, \ldots, v_s^{i_0})$ such that 
    \begin{align*}
        d_{G}(v_1^{i_0}) + \cdots + d_{G}(v_s^{i_0})
        \ge \frac{2|G|}{n}
        &> s^2n-s-n+1+\frac{2t_{t-1}(n)}{n}+2s^2 \\
        &\ge s^2n-s-n+1+\left(\frac{t-2}{t-1}n-1\right)+2s^2 \\
        &> s^2n-\frac{2n}{t}+s^2. 
    \end{align*}
    
    If $i_0=1$, then the last displayed inequality contradicts the choice of $(v_1,\ldots,v_s)$. Hence $i_0\ne 1$.
    We now construct a new graph $\tilde{G}$ from $G$ by a sequence of symmetrization operations.
    Set $G^{(0)}\coloneqq G$.
    For $j=1,\ldots,s$, let $G^{(j)}$ be the graph obtained from $G^{(j-1)}$ by deleting all edges incident with $v_j$ and then joining $v_j$ to every vertex in $N_{G^{(j-1)}}(v_j^{i_0})\setminus\{v_j\}$.
    Finally, set $\tilde{G}\coloneqq G^{(s)}$.
    We claim that every graph $G^{(j)}$ is balanced $K_{st}$-free.
    This is clear for $j=0$.
    Suppose that $G^{(j-1)}$ is balanced $K_{st}$-free and that $G^{(j)}$ contains a balanced $K_{st}$, say $K$.
    Since the $j$-th step changes only edges incident with $v_j$, the clique $K$ must contain $v_j$.
    By construction, $v_jv_j^{i_0}\notin G^{(j)}$, so $v_j^{i_0}\notin V(K)$.
    For every $x\in V(K)\setminus\{v_j\}$, the edge $v_jx$ in $G^{(j)}$ implies that $v_j^{i_0}x$ is an edge of $G^{(j-1)}$.
    Replacing $v_j$ by $v_j^{i_0}$ therefore gives a balanced $K_{st}$ in $G^{(j-1)}$, a contradiction.
    Thus $\tilde{G}$ is balanced $K_{st}$-free.
    Moreover, before the $j$-th step only the vertices $v_1,\ldots,v_{j-1}$ have been changed, and hence
    \begin{align*}
        d_{G^{(j-1)}}(v_j^{i_0})\ge d_G(v_j^{i_0})-(j-1)
        \quad\text{and}\quad
        d_{G^{(j-1)}}(v_j)\le d_G(v_j)+(j-1).
    \end{align*}
    In the $j$-th symmetrization, at most one neighbor of $v_j^{i_0}$, namely $v_j$, is not copied to $v_j$.
    Therefore
    \begin{align*}
        |G^{(j)}|-|G^{(j-1)}|
        &\ge d_G(v_j^{i_0})-d_G(v_j)-(2j-1).
    \end{align*}
    Summing this over $j\in[s]$, we obtain
    \begin{align*}
        |\tilde{G}|
        = |G^{(s)}| &\ge |G|+\sum_{j\in[s]}\bigl(d_G(v_j^{i_0})-d_G(v_j)-(2j-1)\bigr)\\
        &= |G| - \sum_{j\in[s]}d_G(v_j) + \sum_{j\in[s]}d_G(v_j^{i_0}) - s^2\\
        &> |G| - \left( s^2n-\frac{2n}{t}\right) + \left( s^2n-\frac{2n}{t} + s^2\right)-s^2
        = |G|, 
    \end{align*}
    contradicting the maximality of $G$. 
\end{proof}
For a set $S \subseteq V(G)$ and an index $i \in [s]$, we define the common neighbors of $S$ in $V_i$ as
\begin{align*}
   N_i(S) \coloneqq \bigcap_{v\in S}\bigl(N_G(v)\cap V_i\bigr).
\end{align*}
Additionally, let $\nu(S)$ denote the number of local equivalence classes contained in $S$.
We may assume throughout the proof that $\nu(V_i)\ge t$ for all $i\in[s]$. 
Indeed, if $\nu(V_i) < t$ for some $i$, then $|G[V_i]|\le t_{t-1}(n)$. 
In this case, the total number of edges in $G$ is bounded by
\begin{align*}
    |G| \le \binom{sn}{2} - \binom{n}{2} + |G[V_i]| 
    \le \binom{sn}{2} - \binom{n}{2} + t_{t-1}(n)
    \le \frac{(sn)^2}{2}-\frac{n^2}{2(t-1)} + s^2 n,
\end{align*}
in which case Theorem~\ref{THM:s-colored-Kst-L1-reduced} follows immediately.
Therefore, we proceed under the assumption that
\begin{align}\label{eq:nu-lower-bound}
        \nu(V_i) \ge t \quad \text{for all } i \in [s].
\end{align}

We next grow a balanced clique as far as possible while keeping enough local classes in every unused part.
Let $\tau$ be the maximum integer for which there exists a set $S\subseteq V(G)$ satisfying:
\begin{enumerate}[label=(\roman*), ref=(\roman*)]
    \item\label{ENM:tua-induced-clique}  
    $S$ induces a $(\tau,t)$-balanced clique in $G$,
    \item\label{ENM:tua-neighbour-nu}  
    $\nu(N_i(S))\ge t$ for every index $i\in[s]$ such that $S\cap V_i=\emptyset$,
\end{enumerate}
where we adopt the convention that $N_i(\emptyset)\coloneqq V_i$.

Let $S$ be a set that realizes this maximum $\tau$.
To simplify notation, we may assume without loss of generality that $S$ intersects exactly the last $\tau$ parts.
Letting $\bar{\tau}\coloneqq s-\tau$, this implies that
\begin{align*}
    S\cap V_i\neq\emptyset \quad \text{if and only if } i\in[\bar{\tau}+1,s].
\end{align*}
Moreover, by defining $S_i\coloneqq S\cap V_i$, we have
\begin{align*}
    |S_i| = t \qquad \text{for every } i\in[\bar{\tau}+1,s].
\end{align*}

\begin{claim}\label{CLAIM:tau-interval}
    $\bar{\tau}\in [2, s]$.
\end{claim}

\begin{proof}[Proof of Claim~\ref{CLAIM:tau-interval}]
    Since $\bar{\tau} \coloneqq s - \tau$, it suffices to show that $\tau\in [0,\,s-2]$.
    We first verify the lower bound, $\tau \ge 0$. The trivial choice of $\tau=0$ and $S=\emptyset$ satisfies condition~\ref{ENM:tua-induced-clique}, while condition~\ref{ENM:tua-neighbour-nu} holds by~\eqref{eq:nu-lower-bound}. Since $\tau$ is defined as the maximum integer admitting such a set, we conclude $\tau \ge 0$.

    For the upper bound, suppose for contradiction that $\tau \ge s-1$. If $\tau=s$, then condition~\ref{ENM:tua-induced-clique} gives a balanced $K_{st}$ in $G$, a contradiction. Thus $\tau=s-1$. Let $i$ be the unique index with $S\cap V_i=\emptyset$. By~\ref{ENM:tua-neighbour-nu}, we have $\nu(N_i(S))\ge t$. Choose $t$ vertices in $N_i(S)$ from distinct local equivalence classes. These vertices are adjacent to every vertex of $S$, and Fact~\ref{FACT:symmetrized-G} implies that they form a clique inside $V_i$. Hence they extend $S$ to a $(s,t)$-balanced clique, contradicting the $(s,t)$-balanced clique freeness of $G$.  
    Thus $\tau \le s-2$.
\end{proof}

By~\ref{ENM:tua-neighbour-nu}, we have $\nu\bigl(N_i(S)\bigr)\ge t$ for each $i\in[\bar{\tau}]$.
For each such $i$, choose one vertex from each of the $t$ largest local equivalence classes in $N_i(S)$, and denote these vertices by
\begin{align*}
    S_i \coloneqq \{v_1^i,\ldots,v_t^i\},
\end{align*}
so that $[v_1^i],\ldots,[v_t^i]$ are the $t$ largest local equivalence classes in $N_i(S)$.
For any $i\in[\bar{\tau}]$, the set $S\cup S_i$ induces a $(\tau+1,t)$-balanced clique in $G$.
Indeed, each vertex of $S_i$ lies in $N_i(S)$, and the vertices of $S_i$ lie in distinct local equivalence classes, so Fact~\ref{FACT:symmetrized-G} implies that they form a clique inside $V_i$.
For each unused part, the maximality of $\tau$ gives an obstruction to extending $S$, and we record one such obstruction by the map $f$.
Thus, for each such $i$, there must exist an index $f(i)\in[\bar{\tau}]\setminus\{i\}$ such that
\begin{align*}
    \nu\bigl(N_{f(i)}(S\cup S_i)\bigr)\le t-1.
\end{align*}

Since $f$ maps the finite set $[\bar{\tau}]$ to itself and has no fixed point, it contains a directed cycle of length at least two.
By relabeling the indices in $[\bar{\tau}]$, we may assume that this cycle is of the form
\begin{align*}
    1\to 2\to\cdots\to r\to 1 \text{\quad for some $r\in[\bar{\tau}]$ with $r\ge2$.}
\end{align*}
In particular, for every $i\in[r]$, we have 
\begin{align*}
    \nu\bigl(N_{i+1}(S\cup S_i)\bigr)\le t-1,
\end{align*}
where the index $i+1$ is taken modulo $r$.
We split the degree sum according to the cyclic obstruction, the remaining unused parts, and the original clique $S$.
We define the following three partial sums of degrees:
\begin{align*}
    e_1 \coloneqq \sum_{i\in [r]}\sum_{v\in S_i}d_G(v),
    \quad
    e_{2} \coloneqq \sum_{i\in [r+1, \bar{\tau}]}\sum_{v\in S_i}d_G(v)
    \quad\text{and}\quad
    e_{3} \coloneqq \sum_{v\in S}d_G(v).
\end{align*}

\begin{claim}\label{CLAIM:degree-sum-part-one}
    $e_1\le srtn- \sum_{i\in [r]}|N_{i}(S)|.$
\end{claim}
\begin{proof}[Proof of Claim~\ref{CLAIM:degree-sum-part-one}]
    Throughout this claim, the index $i+1$ is taken modulo $r$.
    For each $i\in[r]$, we estimate the number of edges between $S_{i}$ and $V_{i+1}$.
    Let $\mathcal{C}$ be the set of local equivalence classes contained in $N_{i+1}(S\cup S_i)$.
    By the definition of $f$ along the cycle, we have $|\mathcal{C}|\le t-1$.
    Every vertex in $N_{i+1}(S)$ outside the classes in $\mathcal{C}$ is adjacent to at most $t-1$ vertices of $S_i$.
    Since $\mathcal{C}$ consists of at most $t-1$ classes contained in $N_{i+1}(S)$, its total size is at most the total size of the $t$ largest classes in $N_{i+1}(S)$, which is $\sum_{j\in[t]}|[v_j^{i+1}]|$ by the choice of $S_{i+1}$.
    Hence
    \begin{align*}
        \sum_{v\in S_i} |N_{i+1}(v)|
        &\le \sum_{v\in S_i} (|V_{i+1}\setminus N_{i+1}(S)| + |N_{i+1}(v)\cap N_{i+1}(S)|)\\ 
        &\le |S_i|(n - |N_{i+1}(S)|) + (|S_i|-1)|N_{i+1}(S)| + |[v_1^{i+1}]| + \cdots + |[v_t^{i+1}]|\\
        &= tn - |N_{i+1}(S)| + \sum_{j\in [t]} |[v_j^{i+1}]|.
    \end{align*}
    Since $G$ is locally symmetrized, the sum of degrees within $V_i$ for vertices in $S_i$ is
    \begin{align*}
        \sum_{v\in S_i} |N_{i}(v)| 
         = (n- |[v_1^{i}]|) + \cdots + (n-|[v_t^{i}]|)
         = tn - \sum_{j\in [t]}|[v_j^{i}]|.
    \end{align*}
    Consequently, the total number of edges between $S_{i}$ and $V_{i}\cup V_{i+1}$ is bounded by
    \begin{align*}
         \sum_{v\in S_i} (|N_{i}(v)| + |N_{i+1}(v)|) \le  2tn - |N_{i+1}(S)| + \sum_{j\in [t]} |[v_j^{i+1}]| - \sum_{j\in [t]}|[v_j^{i}]|.
    \end{align*}
    Including the trivial upper bound $|N_j(v)|\le n$ for the remaining $s-2$ indices $j \notin \{i, i+1\}$, we obtain
    \begin{align*}
        \sum_{v\in S_i} d_G(v) \le stn - |N_{i+1}(S)| + \sum_{j\in [t]} |[v_j^{i+1}]| - \sum_{j\in [t]}|[v_j^{i}]|.
    \end{align*}
    
    Summing this inequality over all $i \in [r]$, the terms involving the sizes of equivalence classes telescope and cancel out, yielding
    \begin{align*}
        \sum_{i\in [r]}\sum_{v\in S_i} d_G(v)\le r\cdot stn - \sum_{i\in [r]}|N_{i+1}(S)| = srtn - \sum_{i\in [r]}|N_{i}(S)|.
    \end{align*}
    This completes the proof.
\end{proof}
Using the trivial upper bound $d_G(v)\le sn$, we estimate the second sum as
\begin{align}\label{eq:degree-sum-part-two}
    e_{2}
    =\sum_{i\in [r+1,\bar{\tau}]}\sum_{v\in S_i} d_G(v)
    \le (\bar{\tau}-r)t \cdot sn.
\end{align}
For the third sum $e_3$, we observe that for $i\in [r]$,
\begin{align*}
    \sum_{v\in S} |N_{i}(v)| \le (|S| - 1)n + |N_{i}(S)|.
\end{align*}
This yields
\begin{align}\label{eq:degree-sum-part-three}
    e_{3}
    =\sum_{v\in S} d_G(v) 
    = \sum_{v\in S}\sum_{j\in [r]} |N_j(v)|+\sum_{v\in S}\sum_{j\in [r+1,s]} |N_j(v)|
    \le s\tau t n + \sum_{i\in [r]}\bigl(|N_{i}(S)| - n\bigr).
\end{align}
Together with the original sets $S_i=S\cap V_i$ for $i\in[\bar{\tau}+1,s]$, the sets $S_i$ are now defined for every $i\in[s]$, and each has size $t$.
Combining Claim~\ref{CLAIM:degree-sum-part-one}, \eqref{eq:degree-sum-part-two}, and \eqref{eq:degree-sum-part-three}, the total degree sum over all $S_i$ is
\begin{align*}
    \sum_{i\in [s]}\sum_{v\in S_i}d_G(v) 
    &= e_1+e_2+e_3\\
    &\le \left(srtn - \sum_{i\in [r]}|N_{i}(S)|\right) + (\bar{\tau}-r)tsn + \left(s\tau t n + \sum_{i\in [r]}\bigl(|N_{i}(S)| - n\bigr)\right)\\
    &\le s^2tn - rn.
\end{align*}
By averaging, there exists a tuple $(v_1,\ldots,v_s)\in S_1\times \cdots\times S_s$ such that
\begin{align*}
    \sum_{i\in [s]}d_G(v_i) 
    &\le \frac{s^2tn - rn}{t} = s^{2}n-\frac{rn}{t}.
\end{align*}
Since $r\ge 2$, this is at most $s^{2}n-\frac{2n}{t}$, contradicting Claim~\ref{CLAIM:degree-sum-bound}.
\end{proof}
\subsection{Proof of the \texorpdfstring{$\ell_2$}{l2}-Norm Bound}\label{SUBSEC:Proof-L2-Norm-Bound}

For a graph $G$, let
\begin{align*}
    \norm{G}_{2}\coloneqq \sum_{v\in V(G)} d_G(v)^2,
    \qquad
    s_G(v)\coloneqq \norm{G}_{2}-\norm{G-v}_{2}.
\end{align*}
Then
\begin{align}\label{equ:def-graph-2norm-degree}
    s_G(v)
    = d_G(v)^2+\sum_{u\in N_G(v)}(2d_G(u)-1).
\end{align}
Summing~\eqref{equ:def-graph-2norm-degree} over \(v\) gives
\begin{align*}
    \sum_{v\in V(G)} s_G(v)
    = 3\norm{G}_{2}-2|G|.
\end{align*}
Indeed, the first term contributes \(\norm{G}_{2}\), while each vertex \(u\) contributes \(d_G(u)(2d_G(u)-1)\) to the second term. Hence
\begin{align*}
    \sum_{v\in V(G)} \sum_{u\in N_G(v)} (2d_G(u)-1)
    = 2\norm{G}_{2}-\sum_{u\in V(G)} d_G(u)
    = 2\norm{G}_{2}-2|G|,
\end{align*}
which gives the desired identity.
If $S=\{v_1,\ldots,v_s\}\subseteq V(G)$ and $G_i\coloneqq G-\{v_1,\ldots,v_i\}$ with $G_0\coloneqq G$, then
\begin{align*}
    \norm{G}_{2}-\norm{G-S}_{2}
    = \sum_{i\in[s]}\bigl(\norm{G_{i-1}}_{2}-\norm{G_i}_{2}\bigr)
    = \sum_{i\in[s]} s_{G_{i-1}}(v_i)
    \le \sum_{v\in S} s_G(v),
\end{align*}
since \(s_{G'}(v)\le s_G(v)\) whenever \(G'\subseteq G\).

We also record the corresponding identity for \(3\)-graphs, which will be used later.
For a $3$-graph \(\mathcal{H}\) and a vertex \(v\in V(\mathcal H)\), let
\begin{align*}
    s_{\mathcal{H}}(v)\coloneqq \norm{\mathcal{H}}_{2} - \norm{\mathcal{H} - v}_{2},
    \qquad
    s(\mathcal{H})\coloneqq \frac{1}{|V(\mathcal{H})|}\sum_{v\in V(\mathcal{H})} s_{\mathcal{H}}(v).
\end{align*}
Then, as in~\cite[Lemma~3.1]{CILLP24},
\begin{align}\label{equ:def-2norm-degree-b}
    s_{\mathcal{H}}(v)
    & = \norm{L_{\mathcal{H}}(v)}_{2} + \sum_{e\in L_{\mathcal{H}}(v)} (2d_{\mathcal{H}}(e) - 1). 
\end{align}
Indeed, let \(\mathcal{P}_v \coloneqq \bigl\{\{v,u\}\colon u\in V(\mathcal{H})\setminus\{v\}\bigr\}\). Then
\begin{align*}
    s_{\mathcal{H}}(v)
    &= \sum_{e\in \mathcal{P}_v} d_{\mathcal{H}}(e)^2
    + \sum_{e\in \binom{V(\mathcal{H})\setminus\{v\}}{2}}
    \left(d_{\mathcal{H}}(e)^2-d_{\mathcal{H}-v}(e)^2\right).
\end{align*}
If \(e\in \binom{V(\mathcal{H})\setminus\{v\}}{2}\), then \(d_{\mathcal{H}-v}(e)=d_{\mathcal{H}}(e)-1\) when \(e\in L_{\mathcal{H}}(v)\), and \(d_{\mathcal{H}-v}(e)=d_{\mathcal{H}}(e)\) otherwise. Hence
\begin{align*}
    s_{\mathcal{H}}(v)
    &= \sum_{u\in V(\mathcal{H})\setminus\{v\}} d_{\mathcal{H}}(vu)^2
    + \sum_{e\in L_{\mathcal{H}}(v)} (2d_{\mathcal{H}}(e)-1) \\
    &= \norm{L_{\mathcal{H}}(v)}_{2}
    + \sum_{e\in L_{\mathcal{H}}(v)} (2d_{\mathcal{H}}(e)-1),
\end{align*}
since \(d_{L_{\mathcal{H}}(v)}(u)=d_{\mathcal{H}}(vu)\) for every \(u\in V(\mathcal{H})\setminus\{v\}\).

Summing~\eqref{equ:def-2norm-degree-b} over \(v\) gives
\begin{align}\label{equ:def-2norm-degree-average}
    \sum_{v\in V(\mathcal H)} s_{\mathcal H}(v)
    = 4\norm{\mathcal H}_{2} - 3|\mathcal H|.
\end{align}
Indeed, \(\sum_{v\in V(\mathcal H)} \norm{L_{\mathcal H}(v)}_{2} = 2\norm{\mathcal H}_{2}\), and each pair \(e\in \binom{V(\mathcal H)}{2}\) contributes \(d_{\mathcal H}(e)\bigl(2d_{\mathcal H}(e)-1\bigr)\) to the second term. Hence
\begin{align*}
    \sum_{v\in V(\mathcal H)} \sum_{e\in L_{\mathcal H}(v)} (2d_{\mathcal H}(e)-1)
    = 2\norm{\mathcal H}_{2} - \sum_{e\in \binom{V(\mathcal H)}{2}} d_{\mathcal H}(e)
    = 2\norm{\mathcal H}_{2} - 3|\mathcal H|,
\end{align*}
which gives~\eqref{equ:def-2norm-degree-average}.
We now prove the $\ell_2$-norm bound for balanced $K_{2s}$-free graphs.
\begin{proof}[Proof of Theorem~\ref{THM:s-colored-K2s-L2-norm-reduced}]
    We proceed by induction on $n$.
    The induction deletes one vertex from each part. If a low-degree vertex exists, the loss of the $\ell_2$-norm is small. Otherwise the edge bound controls the total degree defect.
    For $n=2$,
    \begin{align*}
        \norm{G}_{2}\le |V(G)|^3=(2s)^3=2s^3n^2,
    \end{align*}
    so the statement holds.

    Now let $n\ge 3$, assume that the statement holds for $n-1$, and let $G$ be a balanced $K_{2s}$-free graph with parts $V_1,\ldots,V_s$, each of size $n$.
    By Theorem~\ref{THM:s-colored-Kst-L1-reduced},
    \begin{align*}
        |G|\le \frac{(s^2-1)n^2}{2}+s^2n.
    \end{align*}
    Hence
    \begin{align}\label{eq:degree-weight-sum}
        \sum_{x \in V(G)} d_G(x) = 2|G| \le (s^{2}-1)n^{2} + 2s^{2} n.
    \end{align}
    Suppose first that there is a vertex $v_1\in V(G)$ with $d_G(v_1)<(s-1)n$. By symmetry, we may assume that $v_1\in V_1$, and choose arbitrary vertices \((v_2,\ldots,v_s)\in V_2\times\cdots\times V_s\).
    By~\eqref{equ:def-graph-2norm-degree},
    \begin{align*}
        s_G(v_1)
        \le d_G(v_1)^2+\sum_{u\in N_G(v_1)}2d_G(u) 
        \le d_G(v_1)^2+2sn\,d_G(v_1) 
        \le (s-1)^2n^2+2s(s-1)n^2.
    \end{align*}
    For each $i\in [2,s]$, we similarly have
    \begin{align*}
        s_G(v_i)
        \le d_G(v_i)^2+\sum_{u\in V(G)}2d_G(u) 
        \le s^2n^2+2(s^2-1)n^2+4s^2n 
        = (3s^2-2)n^2+4s^2n,
    \end{align*}
    where we used $d_G(v_i)\le sn$ and~\eqref{eq:degree-weight-sum}. Therefore,
    \begin{align*}
        \norm{G}_{2} - \norm{G -\{v_1,\ldots,v_s\}}_{2}
        &\le \sum_{i\in [s]} s_G(v_i) \\
        &\le (s-1)^2n^2+2s(s-1)n^2+(s-1)\bigl((3s^2-2)n^2+4s^2n\bigr) \\
        &\le (3s^3-6s+3)n^2+4s^3n.
    \end{align*}
    Writing $G'\coloneqq G-\{v_1,\ldots,v_s\}$, we note that $G'$ remains balanced $K_{2s}$-free with parts $V_i\setminus\{v_i\}$, each of size $n-1$.
    Applying the induction hypothesis to $G'$, we obtain
    \begin{align*}
        \norm{G}_{2}
        &\le \norm{G'}_{2}+(3s^3-6s+3)n^2+4s^3n \\
        &\le \bigl(s^2(s-1)+(s-1)^2\bigr)(n-1)^3+4s^3(n-1)^2+(3s^3-6s+3)n^2+4s^3n \\
        &\le \bigl(s^2(s-1)+(s-1)^2\bigr)n^3+4s^3n^2.
    \end{align*}
    The last inequality follows by expanding. The difference between the final right-hand side and the preceding expression is
    \begin{align*}
        (s^3+6s-3)n-(3s^3+2s-1)\ge 0
    \end{align*}
    for $s\ge 2$ and $n\ge 3$.
    Thus the theorem holds in this case.

    We may therefore assume that $(s-1)n\le d_G(v)\le sn$ for every $v\in V(G)$. Since
    \begin{align*}
        \norm{G}_2=\sum_{v\in V(G)} d_G(v)^2,
    \end{align*}
    it remains to bound the degree square sum under these constraints.

    If $n\le 2s^2$, then the trivial bound $d_G(v)\le sn$ gives
    \begin{align*}
        \norm{G}_2=\sum_{v\in V(G)}d_G(v)^2\le sn\cdot (sn)^2=s^3n^3\le (s^2(s-1)+(s-1)^2)n^3+4s^3n^2,
    \end{align*}
    where the last inequality follows from $n\le 2s^2$, since it is equivalent to $(2s-1)n\le 4s^3$.
    Thus we may assume that $n>2s^2$.

    For each vertex $v$, let $\delta_v\coloneqq sn-d_G(v)$. Then $0\le \delta_v\le n$, and by~\eqref{eq:degree-weight-sum},
    \begin{align*}
        \sum_{v}\delta_v
        = s^2n^2-\sum_v d_G(v) 
        \ge s^2n^2-\bigl((s^2-1)n^2+2s^2n\bigr) 
        = n(n-2s^2),
    \end{align*}
    Since $\sum_v\delta_v^2\le n\sum_v\delta_v$, we have
    \begin{align*}
        \sum_v d_G(v)^2
        = \sum_v (sn-\delta_v)^2 
        = s^3n^3-2sn\sum_v\delta_v+\sum_v\delta_v^2 
        \le s^3n^3-(2s-1)n\sum_v\delta_v.
    \end{align*}
    Therefore
    \begin{align*}
        \norm{G}_{2}
        = \sum_v d_G(v)^2
        &\le s^3n^3-(2s-1)n\sum_v\delta_v\\
        &\le s^3n^3-(2s-1)n^2(n-2s^2)\\
        &= (s^2(s-1) + (s-1)^2) n^3 + (4s^3 - 2s^2)n^2.
    \end{align*}
    This completes the proof of Theorem~\ref{THM:s-colored-K2s-L2-norm-reduced}.
\end{proof}
\subsection{Proof of the Crossing-Triangle Bound}\label{SUBSEC:Proof-Crossing-Triangle-Bound}
For a locally symmetrized graph $G$, let
\begin{align*}
    V_1 = A_1 \cup \cdots \cup A_p
    \qquad\text{and}\qquad
    V_2 = B_1 \cup \cdots \cup B_q
\end{align*}
be the decompositions into equivalence classes, and write
\begin{align*}
    a_i \coloneqq |A_i|,
    \qquad
    b_j \coloneqq |B_j|.
\end{align*}
We define an auxiliary bipartite graph $H$ with vertex classes $[p]$ and $[q]$, regarded as disjoint copies, by declaring that $ij\in H$ if and only if $G[A_i,B_j]$ is complete.

\begin{lemma}\label{LEM:crossing-triangle-class-graph}
    The auxiliary graph $H$ is $K_{2,2}$-free.
\end{lemma}

\begin{proof}[Proof of Lemma~\ref{LEM:crossing-triangle-class-graph}]
    Suppose that $i_1,i_2\in [p]$ and $j_1,j_2\in [q]$ span a copy of $K_{2,2}$ in $H$.
    Choosing arbitrary vertices $u_t\in A_{i_t}$ and $v_t\in B_{j_t}$ for $t\in [2]$, we obtain a balanced $K_4$ in $G$, a contradiction.
\end{proof}

\begin{proof}[Proof of Theorem~\ref{THM:2-colored-K4-triangle-reduced}]
    By Proposition~\ref{PROP:Zykov-symmetry-triangle-ineq}, we may assume that $G$ is locally symmetrized.
    The auxiliary graph converts balanced $K_4$-freeness into a $K_{2,2}$-free condition, which lets us count crossing triangles through common neighborhoods in $H$.
    Let $T_{2,1}$ and $T_{1,2}$ denote the numbers of crossing triangles with two vertices in $V_1$ and in $V_2$, respectively.
    Thus
    \begin{align*}
        \mathcal{N}_{\mathrm{cr}}(K_{3}, G)=T_{2,1}+T_{1,2}.
    \end{align*}

    For a fixed class $B_j$, every crossing triangle counted by $T_{2,1}$ that uses a vertex of $B_j$ arises from a pair of distinct classes $A_i,A_k$ with $i,k\in N_H(j)$.
    Conversely, every such choice gives crossing triangles, since $G[A_i,A_k]$ is complete by Fact~\ref{FACT:symmetrized-G}, while $G[A_i,B_j]$ and $G[A_k,B_j]$ are complete by the definition of $H$.
    Hence
    \begin{align*}
        T_{2,1}
        = \sum_{j\in [q]} b_j \sum_{\{i,k\}\subseteq N_H(j)} a_i a_k.
    \end{align*}
    Let $a \coloneqq \max_{i\in [p]} a_i$ and $b \coloneqq \max_{j\in [q]} b_j$.
    Since $H$ is $K_{2,2}$-free by Lemma~\ref{LEM:crossing-triangle-class-graph}, we have
    \begin{align*}
        |N_H(i)\cap N_H(k)|\le 1
        \qquad\text{for all }1\le i<k\le p.
    \end{align*}
    It follows that
    \begin{align*}
        T_{2,1}
        &\le b \sum_{j\in [q]} \sum_{\{i,k\}\subseteq N_H(j)} a_i a_k \\
        &= b \sum_{1\le i<k\le p} a_i a_k \cdot |N_H(i)\cap N_H(k)| \\
        &\le b \sum_{1\le i<k\le p} a_i a_k 
        = b\left(\binom{n}{2}-\sum_{i\in [p]}\binom{a_i}{2}\right)
        \le b\left(\binom{n}{2}-\binom{a}{2}\right).
    \end{align*}
    By symmetry,
    \begin{align*}
        T_{1,2}
        \le a\left(\binom{n}{2}-\binom{b}{2}\right).
    \end{align*}
    Since $n\binom{n}{2}$ is the number of crossing triangles in each of the split graphs $\Lambda[V_1;V_2]$ and $\Lambda[V_2;V_1]$, a direct calculation gives
    \begin{align*}
        n\binom{n}{2}-(T_{1,2}+T_{2,1})
        &\ge n\binom{n}{2}
        - \left(b\left(\binom{n}{2}-\binom{a}{2}\right)+a\left(\binom{n}{2}-\binom{b}{2}\right)\right) \\
        &= \frac{n(a+b-n)(a+b-n-1)}{2} + \frac{(n-a)(n-b)(a+b-2)}{2} \ge 0.
    \end{align*}
    The last inequality holds because \(a+b-n\) is an integer and \(1\le a,b\le n\), so both terms in the final expression are nonnegative.
    Therefore
    \begin{align*}
        \mathcal{N}_{\mathrm{cr}}(K_{3}, G)\le T_{1,2}+T_{2,1}\le n\binom{n}{2}.
    \end{align*}
\end{proof}

\begin{proof}[Proof of Proposition~\ref{PROP:2-colored-K4-triangle-stability-sym}]
    The proof first shows that near-extremality forces one color class to contain a large independent set and the other color class to be almost complete.
    Let $T_{2,1}$ and $T_{1,2}$ denote the numbers of crossing triangles in $G$ with two vertices in $V_1$ and in $V_2$, respectively.
    Let \(\alpha_1\coloneqq \alpha(G[V_1])\) and \(\alpha_2\coloneqq \alpha(G[V_2])\), where \(\alpha(F)\) denotes the independence number of a graph \(F\).
    For every edge $e\in G[V_1]$, the set $N_G(e)\cap V_2$ is independent in $G[V_2]$, since otherwise $e$ together with an edge of $G[V_2]$ inside this common neighborhood would span a balanced copy of $K_4$ in $G$.
    Therefore
    \begin{align*}
        T_{2,1}\le \alpha_2 |G[V_1]|
        \qquad\text{and}\qquad
        T_{1,2}\le \alpha_1 |G[V_2]|.
    \end{align*}
    Also,
    \begin{align*}
        |G[V_i]|\le \binom{n}{2}-\binom{\alpha_i}{2}
        \qquad\text{for }i\in[2].
    \end{align*}
    Set \(x\coloneqq \alpha_1/n\) and \(y\coloneqq \alpha_2/n\). Then the hypothesis gives
    \begin{align*}
        n\binom{n}{2}-\varepsilon_{\ref{PROP:2-colored-K4-triangle-stability-sym}}n^3
        &\le \mathcal{N}_{\mathrm{cr}}(K_3,G) 
        = T_{2,1}+T_{1,2} \\
        &\le \alpha_2\left(\binom{n}{2}-\binom{\alpha_1}{2}\right)
            + \alpha_1\left(\binom{n}{2}-\binom{\alpha_2}{2}\right) \\
        &\le \frac{n^3}{2}\bigl(y(1-x^2)+x(1-y^2)\bigr) 
        = \frac{n^3}{2}(x+y)(1-xy).
    \end{align*}
    We shall choose \(N_{\ref{PROP:2-colored-K4-triangle-stability-sym}}\) at the end so that \(1/n\le \varepsilon_{\ref{PROP:2-colored-K4-triangle-stability-sym}}\). Since
    \begin{align*}
        n\binom{n}{2}=\left(\frac{1}{2}-\frac{1}{2n}\right)n^3,
    \end{align*}
    it follows that \((x+y)(1-xy)\ge 1-3\varepsilon_{\ref{PROP:2-colored-K4-triangle-stability-sym}}\).

    Write \(S\coloneqq x+y\) and \(P\coloneqq xy\). Since
    \begin{align*}
        1-3\varepsilon_{\ref{PROP:2-colored-K4-triangle-stability-sym}}
        \le S(1-P)\le S,
    \end{align*}
    we have \(S\ge 1-3\varepsilon_{\ref{PROP:2-colored-K4-triangle-stability-sym}}\). If \(S\le 1\), then \(|S-1|\le 3\varepsilon_{\ref{PROP:2-colored-K4-triangle-stability-sym}}\). If \(S\ge 1\), then \((1-x)(1-y)\ge 0\) gives \(P\ge S-1\), and hence
    \begin{align*}
        1-3\varepsilon_{\ref{PROP:2-colored-K4-triangle-stability-sym}}
        \le S(1-P)\le S(2-S)=1-(S-1)^2.
    \end{align*}
    Thus \(S\le 1+2\varepsilon_{\ref{PROP:2-colored-K4-triangle-stability-sym}}^{1/2}\). Returning to \(1-3\varepsilon_{\ref{PROP:2-colored-K4-triangle-stability-sym}}\le S(1-P)\), we obtain
    \begin{align*}
        1-P\ge \frac{1-3\varepsilon_{\ref{PROP:2-colored-K4-triangle-stability-sym}}}{S}
        \ge \frac{1-3\varepsilon_{\ref{PROP:2-colored-K4-triangle-stability-sym}}}{1+2\varepsilon_{\ref{PROP:2-colored-K4-triangle-stability-sym}}^{1/2}},
    \end{align*}
    and therefore \(P\le C\varepsilon_{\ref{PROP:2-colored-K4-triangle-stability-sym}}^{1/2}\) for some absolute constant \(C\).
    In the rest of the proof, \(C\) denotes a sufficiently large absolute constant, which may be enlarged finitely many times.

    By symmetry, it suffices to treat the case $x\ge y$. If \(\varepsilon_{\ref{PROP:2-colored-K4-triangle-stability-sym}}\le 1/9\), then \(x\ge S/2\ge (1-3\varepsilon_{\ref{PROP:2-colored-K4-triangle-stability-sym}})/2\ge 1/3\), so
    \begin{align*}
        y=\frac{P}{x}\le 3C\varepsilon_{\ref{PROP:2-colored-K4-triangle-stability-sym}}^{1/2}
        \qquad\text{and}\qquad
        x=S-y\ge 1-C\varepsilon_{\ref{PROP:2-colored-K4-triangle-stability-sym}}^{1/2},
    \end{align*}
    after enlarging $C$ if necessary.

    Let $r\coloneqq n-\alpha_1$. Then \(r\le C\varepsilon_{\ref{PROP:2-colored-K4-triangle-stability-sym}}^{1/2}n\), and hence
    \begin{align*}
        |G[V_1]|\le \binom{n}{2}-\binom{\alpha_1}{2}\le rn\le C\varepsilon_{\ref{PROP:2-colored-K4-triangle-stability-sym}}^{1/2}n^2.
    \end{align*}
    Consequently, \(T_{2,1}\le \alpha_2 |G[V_1]|\le C\varepsilon_{\ref{PROP:2-colored-K4-triangle-stability-sym}}n^3\), and therefore
    \begin{align}
        T_{1,2}
        = \mathcal{N}_{\mathrm{cr}}(K_3,G)-T_{2,1}
        \ge n\binom{n}{2}-C\varepsilon_{\ref{PROP:2-colored-K4-triangle-stability-sym}}n^3. \label{eq:prop26-T12-lower}
    \end{align}
    Since every edge of $G[V_2]$ lies in at most $n$ crossing triangles of type $V_1V_2V_2$, we obtain
    \begin{align*}
        |G[V_2]|
        \ge \frac{T_{1,2}}{n}
        \ge \binom{n}{2}-C\varepsilon_{\ref{PROP:2-colored-K4-triangle-stability-sym}}n^2.
    \end{align*}
    Let \(\overline{G}\) denote the complement of \(G\) on \(V_1\cup V_2\).
    Thus
    \begin{align*}
        |\overline{G}[V_2]|
        = \binom{n}{2}-|G[V_2]|
        \le C\varepsilon_{\ref{PROP:2-colored-K4-triangle-stability-sym}}n^2.
    \end{align*}
    It remains to control the missing cross edges.
    For each \(b\in V_2\), set
    \begin{align*}
        x_b\coloneqq \bigl|\{a\in V_1\colon ab\notin G\}\bigr|,
    \end{align*}
    and thus
    \begin{align*}
        |\overline{G}[V_1,V_2]|=\sum_{b\in V_2}x_b.
    \end{align*}
    An absent triangle of type \(V_1V_2V_2\) is a triple \(\{a,b,b'\}\) with \(a\in V_1\) and \(bb'\in G[V_2]\) that is not a triangle of \(G\), which happens exactly when at least one of \(ab\) and \(ab'\) is missing.
    Thus the number of absent triangles of type $V_1V_2V_2$ is \(n|G[V_2]|-T_{1,2}\).
    Charging each absent triangle to its missing cross edges, each absent triangle is charged at most twice. Hence, by~\eqref{eq:prop26-T12-lower},
    \begin{align}
        \sum_{b\in V_2} x_b d_{G[V_2]}(b)
        \le 2\bigl(n|G[V_2]|-T_{1,2}\bigr)
        \le 2C\varepsilon_{\ref{PROP:2-colored-K4-triangle-stability-sym}}n^3. \label{eq:prop26-weighted-missing}
    \end{align}
    Let
    \begin{align*}
        V_2'\coloneqq \bigl\{b\in V_2\colon d_{G[V_2]}(b)\ge n-1-C\varepsilon_{\ref{PROP:2-colored-K4-triangle-stability-sym}}^{1/2}n\bigr\}.
    \end{align*}
    Then for any $b\in V_2\setminus V_2'$, we have $n-1-d_{G[V_2]}(b)>C\varepsilon_{\ref{PROP:2-colored-K4-triangle-stability-sym}}^{1/2}n$. 
    Since
    \begin{align*}
        \sum_{b\in V_2}(n-1-d_{G[V_2]}(b))
        = 2|\overline{G}[V_2]|
        \le 2C\varepsilon_{\ref{PROP:2-colored-K4-triangle-stability-sym}}n^2,
    \end{align*}
    and every vertex \(b\in V_2\setminus V_2'\) contributes more than \(C\varepsilon_{\ref{PROP:2-colored-K4-triangle-stability-sym}}^{1/2}n\) to this sum, we obtain
    \begin{align*}
        |V_2\setminus V_2'|
        \le C\varepsilon_{\ref{PROP:2-colored-K4-triangle-stability-sym}}^{1/2}n,
    \end{align*}
    after enlarging \(C\) if necessary.
    Hence the number of missing cross edges incident to \(V_2\setminus V_2'\) is at most \(n|V_2\setminus V_2'|\le C\varepsilon_{\ref{PROP:2-colored-K4-triangle-stability-sym}}^{1/2}n^2\).
    For each $b\in V_2'$, if \(C\varepsilon_{\ref{PROP:2-colored-K4-triangle-stability-sym}}^{1/2}\le 1/4\) and \(n\ge 4\), then
    \begin{align*}
        d_{G[V_2]}(b)\ge n-1-\frac{n}{4}\ge \frac{n}{2}.
    \end{align*}
    Hence, by~\eqref{eq:prop26-weighted-missing},
    \begin{align*}
        \frac{n}{2}\sum_{b\in V_2'} x_b
        \le \sum_{b\in V_2'} x_b d_{G[V_2]}(b)
        \le 2C\varepsilon_{\ref{PROP:2-colored-K4-triangle-stability-sym}}n^3.
    \end{align*}
    Therefore
    \begin{align*}
        \sum_{b\in V_2'}x_b\le C\varepsilon_{\ref{PROP:2-colored-K4-triangle-stability-sym}}n^2,
    \end{align*}
    after enlarging $C$ if necessary, and hence
    \begin{align*}
        |\overline{G}[V_1,V_2]|
        = \sum_{b\in V_2}x_b
        \le C\varepsilon_{\ref{PROP:2-colored-K4-triangle-stability-sym}}^{1/2}n^2.
    \end{align*}

    Thus $G$ can be transformed into $\Lambda[V_1;V_2]$ by deleting all edges of $G[V_1]$, adding all missing edges inside $V_2$, and adding all missing cross edges between $V_1$ and $V_2$. The total number of changes is at most
    \begin{align*}
        |G[V_1]| + |\overline{G}[V_2]| + |\overline{G}[V_1,V_2]|
        \le C\varepsilon_{\ref{PROP:2-colored-K4-triangle-stability-sym}}^{1/2}n^2.
    \end{align*}
    If the original graph has \(y>x\), the same argument with the roles of \(V_1\) and \(V_2\) exchanged shows that $G$ can be transformed into $\Lambda[V_2;V_1]$ by changing at most \(C\varepsilon_{\ref{PROP:2-colored-K4-triangle-stability-sym}}^{1/2}n^2\) edges.
    We now choose the parameters so that all smallness assumptions used above and the required edit-distance bound hold.
    At this point \(C\) is fixed. Choose \(\varepsilon_{\ref{PROP:2-colored-K4-triangle-stability-sym}}>0\) such that
    \begin{align*}
        \varepsilon_{\ref{PROP:2-colored-K4-triangle-stability-sym}}
        \le \min\left\{\frac{1}{9},\frac{1}{16C^2},\frac{\eta^2}{C^2}\right\}.
    \end{align*}
    Then \(C\varepsilon_{\ref{PROP:2-colored-K4-triangle-stability-sym}}^{1/2}\le \min\{\eta,1/4\}\), and the condition \(\varepsilon_{\ref{PROP:2-colored-K4-triangle-stability-sym}}\le 1/9\) used above holds.
    Finally, choose \(N_{\ref{PROP:2-colored-K4-triangle-stability-sym}}\) such that
    \begin{align*}
        N_{\ref{PROP:2-colored-K4-triangle-stability-sym}}
        \ge \max\left\{4,\left\lceil \frac{1}{\varepsilon_{\ref{PROP:2-colored-K4-triangle-stability-sym}}}\right\rceil\right\}.
    \end{align*}
    Then \(1/n\le \varepsilon_{\ref{PROP:2-colored-K4-triangle-stability-sym}}\) and \(n\ge 4\) for all \(n\ge N_{\ref{PROP:2-colored-K4-triangle-stability-sym}}\).
    This completes the proof of Proposition~\ref{PROP:2-colored-K4-triangle-stability-sym}.
\end{proof}
\section{Local structure near the bipartite construction}\label{SEC:Common-Structure}
In this section, we collect the common structural facts for $3$-graphs close to the bipartite construction. These facts do not depend on the objective function. They will be used in both the $\ell_2$-norm problem and the clique-counting problem.

We first introduce the notation used in this section.
Let $\mathcal H$ be an $n$-vertex $3$-graph. Given a bipartition $V_1\cup V_2=V(\mathcal H)$, let $\mathbb B[V_1,V_2]$ denote the complete bipartite $3$-graph with parts $V_1,V_2$, whose edges are the triples of type $V_1V_1V_2$ or $V_1V_2V_2$.
We write
\begin{align*}
    \mathcal H[V_1,V_2]\coloneqq \mathcal H\cap \mathbb B[V_1,V_2].
\end{align*}
We say that $V_1\cup V_2$ is a maximum partition of $\mathcal H$ if it maximizes $|\mathcal H[V_1,V_2]|$ over all bipartitions of $V(\mathcal H)$.
With respect to the same partition, define
\begin{align*}
    \mathcal B_{\mathcal H}[V_1,V_2]\coloneqq \mathcal H\setminus \mathbb B[V_1,V_2]
    \qquad\text{and}\qquad
    \mathcal M_{\mathcal H}[V_1,V_2]\coloneqq \mathbb B[V_1,V_2]\setminus \mathcal H.
\end{align*}
We call the edges in $\mathcal B_{\mathcal H}[V_1,V_2]$ bad edges, and the edges in $\mathcal M_{\mathcal H}[V_1,V_2]$ missing edges.
When there is no ambiguity, we write simply $\mathcal B$ and $\mathcal M$.

For pairs, we use
\begin{align*}
    K[V_1,V_2]\coloneqq \bigl\{\{x,y\}\colon x\in V_1,\ y\in V_2\bigr\}
    \qquad\text{and}\qquad
    \overline K[V_1,V_2]\coloneqq \binom{V_1}{2}\cup \binom{V_2}{2},
\end{align*}
so $K[V_1,V_2]$ is the set of cross pairs and $\overline K[V_1,V_2]$ is the set of same-side pairs.

For a family $\mathcal F$ of triples and a set $S\subseteq V(\mathcal H)$, let
\begin{align*}
    \mathcal F(S)\coloneqq \{F\in\mathcal F\colon S\subseteq F\}
    \qquad\text{and}\qquad
    d_{\mathcal F}(S)\coloneqq |\mathcal F(S)|.
\end{align*}
If $e$ is a pair, we also write
\begin{align*}
    N_{\mathcal F}(e)\coloneqq \{v\in V(\mathcal H)\colon e\cup\{v\}\in\mathcal F\},
\end{align*}
so that $d_{\mathcal F}(e)=|N_{\mathcal F}(e)|$. The maximum degree of $\mathcal F$ is
\begin{align*}
    \Delta(\mathcal F)\coloneqq \max_{v\in V(\mathcal H)}d_{\mathcal F}(v).
\end{align*}
Let $\partial\mathcal H$ denote the shadow of $\mathcal H$. For a vertex $v\in V(\mathcal H)$, its link graph is
\begin{align*}
    L_{\mathcal H}(v)\coloneqq \{xy\in \partial\mathcal H\colon vxy\in\mathcal H\}.
\end{align*}
For convenience, we view $L_{\mathcal H}(v)$ as a graph on $V(\mathcal H)$.
If $L$ is a graph and $X,Y$ are disjoint vertex sets, then $L[X]$ denotes the subgraph of $L$ induced by $X$, and $L[X,Y]$ denotes the bipartite subgraph of $L$ induced between $X$ and $Y$.

Let $\mathbb B_n$ denote the balanced complete bipartite $3$-graph on $n$ vertices. A simple calculation gives
\begin{align*}
    |\mathbb B_n|=\frac{n-2}{2}\left\lfloor\frac{n^2}{4}\right\rfloor.
\end{align*}

Let $\delta>0$. We say that $\mathcal H$ is $\delta$-close to $\mathbb B_n$ if there exists a balanced partition $U_1\cup U_2=V(\mathcal H)$, that is, $\bigl||U_1|-|U_2|\bigr|\le 1$, such that
\begin{align*}
    |\mathcal H\triangle \mathbb B[U_1,U_2]|
    \le \delta n^3.
\end{align*}

By comparing the crossing triples containing a vertex before and after moving it to the other part in a maximum partition, we obtain the following fact.

\begin{fact}\label{FACT:common-maximum-partition-link}
    Let $\mathcal H$ be a $3$-graph, and let $V_1\cup V_2=V(\mathcal H)$ be a maximum partition. Fix $i\in[2]$ and $v\in V_i$, and let $L\coloneqq L_{\mathcal H}(v)$. Then
    \begin{align*}
        |L[V_{3-i}]|\ge |L[V_i]|.
    \end{align*}
\end{fact}

The next lemma records the elementary consequences of being close to $\mathbb B_n$ with respect to a maximum partition.

\begin{lemma}\label{LEMMA:common-near-bipartite-setup}
    Let $\delta>0$ be sufficiently small, and let $n$ be sufficiently large.
    Suppose that $\mathcal H$ is an $n$-vertex $3$-graph which is $\delta$-close to $\mathbb B_n$.
    Let $V_1\cup V_2=V(\mathcal H)$ be a maximum partition of $\mathcal H$, and let $\mathcal B$ and $\mathcal M$ be the bad and missing edges with respect to this partition.
    Then we have
    \begin{enumerate}[label=(\roman*), ref=(\roman*)]
        \item\label{LEMMA:common-near-bipartite-setup-bm} $\max\{|\mathcal B|,|\mathcal M|\}\le 2\delta n^3$,
        \item\label{LEMMA:common-near-bipartite-setup-parts} $\bigl||V_i|-n/2\bigr|\le 10\delta^{1/2}n$ for each $i\in[2]$,
        \item\label{LEMMA:common-near-bipartite-setup-size} $\bigl||\mathcal H|-|\mathbb B_n|\bigr|\le \delta n^3$.
    \end{enumerate}
\end{lemma}

\begin{proof}[Proof of Lemma~\ref{LEMMA:common-near-bipartite-setup}]
    Since $\mathcal H$ is $\delta$-close to $\mathbb B_n$, there is a balanced partition $U_1\cup U_2=V(\mathcal H)$ such that
    \begin{align*}
        |\mathcal H\triangle \mathbb B[U_1,U_2]|
        \le \delta n^3.
    \end{align*}
    Let $\mathbb B^\ast\coloneqq \mathbb B[U_1,U_2]$.
    Then $\mathbb B^\ast\cong \mathbb B_n$, and hence $|\mathbb B^\ast|=|\mathbb B_n|$.
    
    The symmetric-difference bound gives
    \begin{align*}
        |\mathcal H\cap \mathbb B^\ast|
        \ge |\mathbb B^\ast|-|\mathcal H\triangle \mathbb B^\ast|
        \ge |\mathbb B_n|-\delta n^3.
    \end{align*}
    Since $V_1\cup V_2$ is a maximum partition, it follows that
    \begin{align}
        |\mathcal H[V_1,V_2]|\ge |\mathcal H\cap \mathbb B^\ast|\ge |\mathbb B_n|-\delta n^3. \label{equ:common-near-bipartite-crossing-lower}
    \end{align}
    
    Since $|V_1|+|V_2|=n$, we have
    \begin{align*}
        |\mathbb B[V_1,V_2]|=\frac{|V_1||V_2|(n-2)}{2}.
    \end{align*}
    This is maximized when the partition is balanced. Hence, we have $|\mathbb B[V_1,V_2]|\le |\mathbb B_n|$. 
    It follows that
    \begin{align*}
        |\mathcal M|
        =
        |\mathbb B[V_1,V_2]|-|\mathcal H[V_1,V_2]|
        \le |\mathbb B_n|-\bigl(|\mathbb B_n|-\delta n^3\bigr)
        =
        \delta n^3.
    \end{align*}
    Also,
    \begin{align*}
        |\mathcal H|
        \le |\mathbb B^\ast|+|\mathcal H\setminus \mathbb B^\ast|
        \le |\mathbb B_n|+|\mathcal H\triangle \mathbb B^\ast|
        \le |\mathbb B_n|+\delta n^3.
    \end{align*}
    It follows that
    \begin{align*}
        |\mathcal B|
        =
        |\mathcal H|-|\mathcal H[V_1,V_2]|
        \le \bigl(|\mathbb B_n|+\delta n^3\bigr)-\bigl(|\mathbb B_n|-\delta n^3\bigr)
        =
        2\delta n^3.
    \end{align*}
    This proves part~\ref{LEMMA:common-near-bipartite-setup-bm}.
    
    It remains to bound the two parts.
    Since $|\mathbb B_n|=\frac{n-2}{2}\lfloor n^2/4\rfloor$ and $\mathcal H[V_1,V_2]\subseteq \mathbb B[V_1,V_2]$, the lower bound in~\eqref{equ:common-near-bipartite-crossing-lower} gives
    \begin{align*}
        \frac{|V_1||V_2|(n-2)}{2}
        \ge \frac{n-2}{2}\left\lfloor\frac{n^2}{4}\right\rfloor-\delta n^3.
    \end{align*}
    Equivalently,
    \begin{align*}
        \left\lfloor\frac{n^2}{4}\right\rfloor-|V_1||V_2|\le \frac{2\delta n^3}{n-2}.
    \end{align*}
    Hence
    \begin{align*}
        \frac{n^2}{4}-|V_1||V_2|
        =\left(\frac{n^2}{4}-\left\lfloor\frac{n^2}{4}\right\rfloor\right)
        +\left(\left\lfloor\frac{n^2}{4}\right\rfloor-|V_1||V_2|\right)
        \le \frac14+\frac{2\delta n^3}{n-2}.
    \end{align*}
    Since $|V_2|=n-|V_1|$, this implies
    \begin{align*}
        \left(|V_1|-\frac n2\right)^2
        =
        \frac{n^2}{4}-|V_1|(n-|V_1|)
        =
        \frac{n^2}{4}-|V_1||V_2|
        \le \frac14+\frac{2\delta n^3}{n-2}
        \le 100\delta n^2.
    \end{align*}
    The last inequality holds because $\delta$ is sufficiently small and $n$ is sufficiently large. Therefore $\bigl||V_1|-n/2\bigr|\le 10\delta^{1/2}n$. The same estimate holds for $V_2$, since $|V_2|=n-|V_1|$. Hence part~\ref{LEMMA:common-near-bipartite-setup-parts} follows.
    
    Finally,
    \begin{align*}
        \bigl||\mathcal H|-|\mathbb B_n|\bigr|
        =
        \bigl||\mathcal H|-|\mathbb B^\ast|\bigr|
        \le |\mathcal H\triangle \mathbb B^\ast|
        \le \delta n^3.
    \end{align*}
    This proves part~\ref{LEMMA:common-near-bipartite-setup-size}, and completes the proof.
\end{proof}

We next record the colored removal lemma in the form needed below.

\begin{lemma}[Colored Graph Removal Lemma~\cite{Fox11}]\label{LEM:common-colored-removal}
    For every $\varepsilon>0$ and every pair of positive integers $f,k$, there exists $\delta_{\ref{LEM:common-colored-removal}}\coloneqq\delta_{\ref{LEM:common-colored-removal}}(\varepsilon,f,k)>0$ such that the following holds.
    Let $\phi:E(F)\to[k]$ be a $k$-edge-coloring of a graph $F$ on $f$ vertices, and let $\psi:E(G)\to[k]$ be a $k$-edge-coloring of a graph $G$ on $n$ vertices.
    If the number of copies of $F$ with coloring $\phi$ in the coloring $\psi$ of $G$ is at most $\delta_{\ref{LEM:common-colored-removal}}n^f$, then all such copies can be removed by deleting at most $\varepsilon n^2$ edges of $G$.
\end{lemma}

Let $G$ be a graph, and let $X\cup Y$ be a bipartition of $V(G)$. A copy of $K_4$ in $G$ is balanced with respect to the partition $X\cup Y$ if it has two vertices in $X$ and two vertices in $Y$.
The following lemma explains why balanced $K_4$'s in the link graphs are controlled by missing edges.

\begin{lemma}\label{LEMMA:common-balanced-k4-missing}
    For every $\varepsilon>0$, there exist $\delta_{\ref{LEMMA:common-balanced-k4-missing}}\coloneqq\delta_{\ref{LEMMA:common-balanced-k4-missing}}(\varepsilon)>0$ and $N_{\ref{LEMMA:common-balanced-k4-missing}}\coloneqq N_{\ref{LEMMA:common-balanced-k4-missing}}(\varepsilon)$ such that the following holds for all $n\ge N_{\ref{LEMMA:common-balanced-k4-missing}}$.
    Suppose that $\mathcal H$ is an $n$-vertex $K_5^3$-free $3$-graph which is $\delta_{\ref{LEMMA:common-balanced-k4-missing}}$-close to $\mathbb B_n$.
    Let $V_1\cup V_2=V(\mathcal H)$ be a maximum partition of $\mathcal H$, and let $\mathcal M$ be the missing edges with respect to this partition.
    For $i\in[2]$ and $v\in V_i$, let $L_v\coloneqq L_{\mathcal H}(v)$, and let $\mathcal T_v$ be the family of all balanced copies of $K_4$ in $L_v$ with respect to the partition $V_1\cup V_2$.
    Then
    \begin{align*}
        |\mathcal T_v|\le |\mathcal M|n\le 2\delta_{\ref{LEMMA:common-balanced-k4-missing}}n^4.
    \end{align*}
    Moreover, by deleting at most $\varepsilon n^2$ edges from $L_v$, we obtain a graph which is balanced $K_4$-free with respect to the partition $V_1 \cup V_2$.
\end{lemma}

\begin{proof}[Proof of Lemma~\ref{LEMMA:common-balanced-k4-missing}]
    Fix $\varepsilon>0$, and set $\eta\coloneqq\delta_{\ref{LEM:common-colored-removal}}(\varepsilon,4,3)$.
    Choose $\delta_{\ref{LEMMA:common-balanced-k4-missing}}>0$ sufficiently small so that Lemma~\ref{LEMMA:common-near-bipartite-setup} applies and $48\delta_{\ref{LEMMA:common-balanced-k4-missing}}\le\eta$.
    Let $N_{\ref{LEMMA:common-balanced-k4-missing}}$ be sufficiently large.
    Write $\delta\coloneqq\delta_{\ref{LEMMA:common-balanced-k4-missing}}$.
    By Lemma~\ref{LEMMA:common-near-bipartite-setup}, we have
    \begin{align*}
        |\mathcal M|\le 2\delta n^3.
    \end{align*}
    
    Fix $v\in V(\mathcal H)$. By symmetry, we may assume that $v\in V_1$, and write $L\coloneqq L_{\mathcal H}(v)$.
    Let $\mathcal T$ be the family of all balanced copies of $K_4$ in $L$ with respect to the partition $V_1\cup V_2$.
    
    We first show that every $Q\in\mathcal T$ contains a missing triple.
    Fix $Q\in\mathcal T$.
    Since $v$ is isolated in $L$, the copy $Q$ does not contain $v$.
    Since $Q$ is a clique in $L$, every pair $xy\in\binom{Q}{2}$ satisfies $vxy\in\mathcal H$.
    Hence all six triples containing $v$ and two vertices of $Q$ are edges of $\mathcal H$.
    Since $Q$ has two vertices in $V_1$ and two vertices in $V_2$, we also have
    \begin{align*}
        \binom{Q}{3}\subseteq \mathbb B[V_1,V_2].
    \end{align*}
    If $\binom{Q}{3}\cap\mathcal M=\emptyset$, then $\binom{Q}{3}\subseteq \mathcal H$. 
    Thus all triples on $Q\cup\{v\}$ belong to $\mathcal H$, which gives a copy of $K_5^3$, a contradiction.
    Therefore
    \begin{align*}
        \binom{Q}{3}\cap\mathcal M\ne\emptyset
        \qquad\text{for every }Q\in\mathcal T.
    \end{align*}
    
    For each $Q\in\mathcal T$, choose one triple $F_Q\in\binom{Q}{3}\cap\mathcal M$, which exists by the previous paragraph.
    Fix $F\in\mathcal M$.
    If $v\in F$, then no such $Q$ is charged to $F$.
    Otherwise, since $F\subseteq\mathbb B[V_1,V_2]$, the triple $F$ has type $V_1V_1V_2$ or $V_1V_2V_2$.
    In either case, $F$ can be extended to a balanced $4$-set in at most $n$ ways.
    Therefore,
    \begin{align}
        |\mathcal T|\le |\mathcal M|n\le 2\delta n^4. \label{equ:common-balanced-k4-count}
    \end{align}
    
    Color the edges of $L$ with three colors according as they lie inside $V_1$, inside $V_2$, or between $V_1$ and $V_2$.
    A balanced copy of $K_4$ is exactly an underlying copy of the colored $K_4$ with one edge inside $V_1$, one edge inside $V_2$, and four cross edges.
    Each balanced copy gives at most $4!$ such colored copies under labelled conventions.
    It follows from~\eqref{equ:common-balanced-k4-count} and the choice of $\delta_{\ref{LEMMA:common-balanced-k4-missing}}$ that the number of copies of the colored $K_4$ described above in $L$ is at most
    \begin{align*}
        24|\mathcal T|\le 48\delta n^4\le \eta n^4.
    \end{align*}
    Applying Lemma~\ref{LEM:common-colored-removal}, all these colored copies can be removed by deleting at most $\varepsilon n^2$ edges of $L$.
    After these deletions, no balanced copy of $K_4$ remains with respect to the partition $V_1\cup V_2$.
\end{proof}

We finish this section with the common local control of bad and missing edges.
This lemma will be applied after the two objective-specific arguments give a local degree bound for the missing and bad edges.

\begin{lemma}\label{LEMMA:common-local-bad-missing-control}
    Let $\eta>0$ be sufficiently small. There exist $\delta_{\ref{LEMMA:common-local-bad-missing-control}}\coloneqq\delta_{\ref{LEMMA:common-local-bad-missing-control}}(\eta)>0$ and $N_{\ref{LEMMA:common-local-bad-missing-control}}\coloneqq N_{\ref{LEMMA:common-local-bad-missing-control}}(\eta)$ such that the following holds for all $n\ge N_{\ref{LEMMA:common-local-bad-missing-control}}$.
    Suppose that $\mathcal H$ is an $n$-vertex $K_5^3$-free $3$-graph which is $\delta$-close to $\mathbb B_n$ for some $0<\delta\le\delta_{\ref{LEMMA:common-local-bad-missing-control}}$.
    Let $V_1\cup V_2=V(\mathcal H)$ be a maximum partition of $\mathcal H$, and let $\mathcal B$ and $\mathcal M$ be the bad and missing edges with respect to this partition.
    Suppose further that
    \begin{align*}
        \max\{\Delta(\mathcal B),\Delta(\mathcal M)\}\le 2\eta n^2.
    \end{align*}
    Then the following statements hold.
    \begin{enumerate}[label=(\roman*), ref=(\roman*)]
        \item\label{LEMMA:common-local-bad-missing-control-heavy} every bad edge $E\in\mathcal B$ contains a pair $e\subseteq E$ such that $d_{\mathcal M}(e)\ge n/10$,
        \item\label{LEMMA:common-local-bad-missing-control-same-side} for every $e\in \overline K[V_1,V_2]$, we have
        \begin{align*}
            d_{\mathcal M}(e)\ge d_{\mathcal B}(e)-7\eta^{1/2}n.
        \end{align*}
    \end{enumerate}
\end{lemma}

\begin{proof}[Proof of Lemma~\ref{LEMMA:common-local-bad-missing-control}]
    Choose $\delta_{\ref{LEMMA:common-local-bad-missing-control}}>0$ sufficiently small so that Lemma~\ref{LEMMA:common-near-bipartite-setup} applies and
    \begin{align*}
        20\delta_{\ref{LEMMA:common-local-bad-missing-control}}^{1/2}\le \eta^{1/2}\qquad\text{and}\qquad 10\delta_{\ref{LEMMA:common-local-bad-missing-control}}^{1/2}\le \frac1{20}.
    \end{align*}
    Since $\eta$ is sufficiently small, we take $N_{\ref{LEMMA:common-local-bad-missing-control}}$ sufficiently large so that, for all $n\ge N_{\ref{LEMMA:common-local-bad-missing-control}}$ and every integer $t$ with $t>n/10$ or $t>5\eta^{1/2}n$, we have
    \begin{align*}
        \binom{t}{2}>6\eta n^2.
    \end{align*}
    Since $\mathcal H$ is $\delta$-close to $\mathbb B_n$ and $\delta\le \delta_{\ref{LEMMA:common-local-bad-missing-control}}$, it is also $\delta_{\ref{LEMMA:common-local-bad-missing-control}}$-close to $\mathbb B_n$.
    By Lemma~\ref{LEMMA:common-near-bipartite-setup} with parameter $\delta_{\ref{LEMMA:common-local-bad-missing-control}}$, for each $i\in[2]$ we have
    \begin{align}
        \left||V_i|-\frac n2\right|\le 10\delta_{\ref{LEMMA:common-local-bad-missing-control}}^{1/2}n\qquad\text{and}\qquad \bigl||V_1|-|V_2|\bigr|
        \le 20\delta_{\ref{LEMMA:common-local-bad-missing-control}}^{1/2}n 
        \le \eta^{1/2}n, \label{equ:common-local-part-balance}
    \end{align}
    where $\bigl||V_1|-|V_2|\bigr|\le 20\delta_{\ref{LEMMA:common-local-bad-missing-control}}^{1/2}n$ follows from the first bound, and the final inequality is by the choice of $\delta_{\ref{LEMMA:common-local-bad-missing-control}}$.

    We first prove~\ref{LEMMA:common-local-bad-missing-control-heavy}.
    Fix a bad edge $E=\{x,y,z\}\in\mathcal B$.
    Since $E\notin \mathbb B[V_1,V_2]$, the edge $E$ is contained in one part.
    By symmetry, we may assume that $E\subseteq V_1$.
    Suppose, for a contradiction, that every pair $e\subseteq E$ satisfies $d_{\mathcal M}(e)<n/10$.
    Let
    \begin{align*}
        W\coloneqq \{w\in V_2\colon xyw,xzw,yzw\notin\mathcal M\}.
    \end{align*}
    For every $w\in V_2$, the triples $xyw,xzw,yzw$ all belong to $\mathbb B[V_1,V_2]$.
    Thus each of these triples that is not in $\mathcal M$ belongs to $\mathcal H$.
    By~\eqref{equ:common-local-part-balance} and the definition of $W$, we obtain
    \begin{align*}
        |W|
        \ge |V_2|-d_{\mathcal M}(xy)-d_{\mathcal M}(xz)-d_{\mathcal M}(yz)
        > \frac n2-10\delta_{\ref{LEMMA:common-local-bad-missing-control}}^{1/2}n-\frac{3n}{10}
        > \frac n{10}.
    \end{align*}
    Fix distinct vertices $w,w'\in W$.
    By the definition of $W$, all triples among $x,y,z,w,w'$ which contain two vertices from $\{x,y,z\}$ and one vertex from $\{w,w'\}$ are edges of $\mathcal H$.
    Also $xyz\in\mathcal H$, since $E\in\mathcal B$.
    If none of $xww',yww',zww'$ lies in $\mathcal M$, then these three triples also belong to $\mathcal H$ because they have type $V_1V_2V_2$.
    Hence $\{x,y,z,w,w'\}$ spans a copy of $K_5^3$, a contradiction.
    Therefore,
    \begin{align*}
        \{xww',yww',zww'\}\cap\mathcal M\ne\emptyset
        \qquad\text{for every }\{w,w'\}\in\binom W2.
    \end{align*}
    Choosing one such missing triple for each pair $\{w,w'\}$ gives an injection into the set of missing triples containing at least one of $x,y,z$, since the chosen triple determines the pair $\{w,w'\}$.
    Thus, by the choice of $N_{\ref{LEMMA:common-local-bad-missing-control}}$,
    \begin{align*}
        6\eta n^2
        <\binom{|W|}{2}
        \le d_{\mathcal M}(x)+d_{\mathcal M}(y)+d_{\mathcal M}(z) 
        \le 3\Delta(\mathcal M)
        \le 6\eta n^2,
    \end{align*}
    a contradiction.
    This proves~\ref{LEMMA:common-local-bad-missing-control-heavy}.

    We now prove~\ref{LEMMA:common-local-bad-missing-control-same-side}.
    Let $e\in\overline K[V_1,V_2]$.
    By symmetry, write $e=\{u_1,u_2\}\subseteq V_1$.
    Set
    \begin{align*}
        N_1\coloneqq N_{\mathcal H}(e)\cap V_1\qquad\text{and}\qquad N_2\coloneqq N_{\mathcal H}(e)\cap V_2.
    \end{align*}
    Then
    \begin{align*}
        d_{\mathcal B}(e)=|N_1|\qquad\text{and}\qquad d_{\mathcal M}(e)=|V_2|-|N_2|.
    \end{align*}
    If $|N_1|\le 6\eta^{1/2}n$, then $d_{\mathcal M}(e)\ge 0\ge |N_1|-7\eta^{1/2}n$, and the desired inequality follows.
    If $|N_2|\le 6\eta^{1/2}n$, then~\eqref{equ:common-local-part-balance} gives
    \begin{align*}
        d_{\mathcal M}(e)
        =|V_2|-|N_2|
        \ge |V_1|-\eta^{1/2}n-6\eta^{1/2}n
        \ge |N_1|-7\eta^{1/2}n,
    \end{align*}
    and again the desired inequality follows.
    Therefore, we may assume that
    \begin{align*}
        |N_1|>6\eta^{1/2}n\qquad\text{and}\qquad |N_2|>6\eta^{1/2}n.
    \end{align*}

    We claim that there exists $y\in N_1$ such that
    \begin{align*}
        d_{\mathcal M}(u_1y)+d_{\mathcal M}(u_2y)\le \eta^{1/2}n.
    \end{align*}
    Indeed, otherwise
    \begin{align*}
        \sum_{y\in N_1}\bigl(d_{\mathcal M}(u_1y)+d_{\mathcal M}(u_2y)\bigr)
        > |N_1|\eta^{1/2}n
        > 6\eta n^2.
    \end{align*}
    On the other hand, for each $j\in[2]$, every missing triple counted by $\sum_{y\in N_1}d_{\mathcal M}(u_jy)$ is of the form $u_jyw$ with $y\in N_1$ and $w\in V_2$. Hence it is counted for at most one vertex $y\in N_1$.
    It follows that
    \begin{align*}
        \sum_{y\in N_1}\bigl(d_{\mathcal M}(u_1y)+d_{\mathcal M}(u_2y)\bigr)
        \le d_{\mathcal M}(u_1)+d_{\mathcal M}(u_2)
        \le 2\Delta(\mathcal M)
        \le 4\eta n^2,
    \end{align*}
    a contradiction.
    This proves the claim.

    Fix such a vertex $y\in N_1$, and define
    \begin{align*}
        W\coloneqq N_2\cap N_{\mathcal H}(u_1y)\cap N_{\mathcal H}(u_2y).
    \end{align*}
    For $w\in N_2$, if $u_1yw\notin\mathcal H$, then $u_1yw\in\mathcal M$ because $u_1,y\in V_1$ and $w\in V_2$.
    The same holds for $u_2yw$.
    It follows that
    \begin{align*}
        |W|
        \ge |N_2|-d_{\mathcal M}(u_1y)-d_{\mathcal M}(u_2y)
        >5\eta^{1/2}n.
    \end{align*}
    For every $w\in W$, we have
    \begin{align*}
        u_1u_2w,\quad u_1yw,\quad u_2yw\in\mathcal H.
    \end{align*}
    Since $y\in N_1$, we also have $u_1u_2y\in\mathcal H$.
    Fix distinct vertices $w,w'\in W$.
    If none of $u_1ww',u_2ww',yww'$ lies in $\mathcal M$, then all these triples belong to $\mathcal H$ because they have type $V_1V_2V_2$.
    Hence all triples on $\{u_1,u_2,y,w,w'\}$ belong to $\mathcal H$, giving a copy of $K_5^3$, a contradiction.
    Thus,
    \begin{align*}
        \{u_1ww',u_2ww',yww'\}\cap\mathcal M\ne\emptyset
        \qquad\text{for every }\{w,w'\}\in\binom W2.
    \end{align*}
    Choosing one such missing triple for each pair $\{w,w'\}$ gives an injection into the set of missing triples containing at least one of $u_1,u_2,y$, since the chosen triple determines the pair $\{w,w'\}$.
    Hence, by the choice of $N_{\ref{LEMMA:common-local-bad-missing-control}}$,
    \begin{align*}
        6\eta n^2
        <\binom{|W|}{2}
        \le d_{\mathcal M}(u_1)+d_{\mathcal M}(u_2)+d_{\mathcal M}(y)
        \le 3\Delta(\mathcal M)
        \le 6\eta n^2,
    \end{align*}
    a contradiction.
    Therefore $d_{\mathcal M}(e)\ge d_{\mathcal B}(e)-7\eta^{1/2}n$ for every $e\in\overline K[V_1,V_2]$.
    This proves~\ref{LEMMA:common-local-bad-missing-control-same-side}, and completes the proof of Lemma~\ref{LEMMA:common-local-bad-missing-control}.
\end{proof}

%
\section{The \texorpdfstring{$\ell_2$}{l2}-norm Tur\'{a}n number of \texorpdfstring{$K_5^3$}{K53}}\label{SEC:K53-L2norm}
In this section, we prove Theorem~\ref{THM:L2-exact-K53}. Starting from the common structure near $\mathbb B_n$ obtained in Section~\ref{SEC:Common-Structure}, we first derive the link estimates needed for the $\ell_2$-argument and then use them to obtain local control on bad and missing edges. Finally, we apply a local modification procedure to show that the $\ell_2$-objective increases under suitable local changes whenever a bad edge remains. This forces every extremal example to be the balanced bipartite construction.

\subsection{Colored graph estimates for links}\label{SUBSEC:K53-L2-link-estimates}

This subsection records the two colored graph estimates which will be applied to link graphs. Recall that, once a bipartition is fixed, balanced copies of $K_4$ are taken with respect to that bipartition, and $\norm{G}_2$ denotes the sum of squared degrees of a graph $G$.

\begin{lemma}\label{LEMMA:K53-L2-link-bound}
    For every $\varepsilon>0$, there exist $\delta_{\ref{LEMMA:K53-L2-link-bound}}\coloneqq \delta_{\ref{LEMMA:K53-L2-link-bound}}(\varepsilon)>0$ and $N_{\ref{LEMMA:K53-L2-link-bound}}\coloneqq N_{\ref{LEMMA:K53-L2-link-bound}}(\varepsilon)$ such that the following holds for all $n\ge N_{\ref{LEMMA:K53-L2-link-bound}}$.
    Let $G$ be a graph with a bipartition $V_1\cup V_2=V(G)$ and $|V(G)|=n$.
    Suppose that
    \begin{align*}
        \max_{i\in[2]}\left||V_i|-\frac n2\right|\le \delta_{\ref{LEMMA:K53-L2-link-bound}} n
    \end{align*}
    and that by deleting at most $\delta_{\ref{LEMMA:K53-L2-link-bound}}n^2$ edges from $G$, we obtain a graph which is balanced $K_4$-free with respect to the partition $V_1\cup V_2$.
    Then
    \begin{align*}
        |G|\le \left(\frac38+\varepsilon\right)n^2\qquad\text{and}\qquad \norm{G}_{2}\le \left(\frac58+\varepsilon\right)n^3.
    \end{align*}
\end{lemma}

\begin{proof}[Proof of Lemma~\ref{LEMMA:K53-L2-link-bound}]
    Fix $\varepsilon>0$.
    Choose $\delta_{\ref{LEMMA:K53-L2-link-bound}}>0$ sufficiently small and $N_{\ref{LEMMA:K53-L2-link-bound}}$ sufficiently large so that, for all $n\ge N_{\ref{LEMMA:K53-L2-link-bound}}$,
    \begin{align*}
        \frac32\left(\frac12+\delta_{\ref{LEMMA:K53-L2-link-bound}}\right)^2+\delta_{\ref{LEMMA:K53-L2-link-bound}}+\frac{4\left(\frac12+\delta_{\ref{LEMMA:K53-L2-link-bound}}\right)}{n}
        &\le \frac38+\varepsilon,\\
        5\left(\frac12+\delta_{\ref{LEMMA:K53-L2-link-bound}}\right)^3+8\delta_{\ref{LEMMA:K53-L2-link-bound}}+\frac{32\left(\frac12+\delta_{\ref{LEMMA:K53-L2-link-bound}}\right)^2}{n}
        &\le \frac58+\varepsilon.
    \end{align*}
    
    Let $D$ be a set of at most $\delta_{\ref{LEMMA:K53-L2-link-bound}}n^2$ edges of $G$ such that $G_0\coloneqq G-D$ is balanced $K_4$-free with respect to the partition $V_1\cup V_2$.
    By symmetry, assume that $|V_1|\ge |V_2|$.
    Adding isolated vertices to $V_2$ in $G_0$ gives a balanced $K_4$-free graph with two parts of size $|V_1|$. This operation changes neither the number of edges nor the $\ell_2$-norm. The assumption gives
    \begin{align*}
        |V_1|\le \left(\frac12+\delta_{\ref{LEMMA:K53-L2-link-bound}}\right)n.
    \end{align*}
    Corollary~\ref{COR:2-colored-K4-L1-reduced} applied to this enlarged graph with part size $|V_1|$ gives
    \begin{align*}
        |G|
        \le |G_0|+\delta_{\ref{LEMMA:K53-L2-link-bound}} n^2
        \le \frac32 |V_1|^2+4|V_1|+\delta_{\ref{LEMMA:K53-L2-link-bound}} n^2
        \le \left(\frac32\left(\frac12+\delta_{\ref{LEMMA:K53-L2-link-bound}}\right)^2+\delta_{\ref{LEMMA:K53-L2-link-bound}}+\frac{4\left(\frac12+\delta_{\ref{LEMMA:K53-L2-link-bound}}\right)}{n}\right)n^2
        \le \left(\frac38+\varepsilon\right)n^2,
    \end{align*}
    where the last inequality follows from the choice of $\delta_{\ref{LEMMA:K53-L2-link-bound}}$ and $N_{\ref{LEMMA:K53-L2-link-bound}}$.
    
    It remains to bound the $\ell_2$-norm. Adding one edge to a graph on $n$ vertices changes the sum of squared degrees by at most $8n$. Hence adding back the deleted edges changes the $\ell_2$-norm by at most $8\delta_{\ref{LEMMA:K53-L2-link-bound}}n^3$. Applying Corollary~\ref{COR:2-colored-K4-L2-norm-reduced} to this enlarged graph with part size $|V_1|$ gives
    \begin{align*}
        \norm{G}_{2}
        &\le \norm{G_0}_{2}+8\delta_{\ref{LEMMA:K53-L2-link-bound}} n^3
        \le 5|V_1|^3+32|V_1|^2+8\delta_{\ref{LEMMA:K53-L2-link-bound}}n^3\\
        &\le \left(5\left(\frac12+\delta_{\ref{LEMMA:K53-L2-link-bound}}\right)^3+8\delta_{\ref{LEMMA:K53-L2-link-bound}}+\frac{32\left(\frac12+\delta_{\ref{LEMMA:K53-L2-link-bound}}\right)^2}{n}\right)n^3
        \le \left(\frac58+\varepsilon\right)n^3,
    \end{align*}
    where the last inequality follows from the choice of $\delta_{\ref{LEMMA:K53-L2-link-bound}}$ and $N_{\ref{LEMMA:K53-L2-link-bound}}$.
    This proves the lemma.
\end{proof}

\begin{lemma}\label{LEMMA:K53-L2-link-stability}
    For every $\varepsilon>0$, there exist $\delta_{\ref{LEMMA:K53-L2-link-stability}}\coloneqq \delta_{\ref{LEMMA:K53-L2-link-stability}}(\varepsilon)>0$ and $N_{\ref{LEMMA:K53-L2-link-stability}}\coloneqq N_{\ref{LEMMA:K53-L2-link-stability}}(\varepsilon)$ such that the following holds for all $n\ge N_{\ref{LEMMA:K53-L2-link-stability}}$.
    Let $G$ be a graph with a bipartition $V_1\cup V_2=V(G)$ and $|V(G)|=n$.
    Suppose that
    \begin{align*}
        \max_{i\in[2]}\left||V_i|-\frac n2\right|\le \delta_{\ref{LEMMA:K53-L2-link-stability}} n
    \end{align*}
    and that by deleting at most $\delta_{\ref{LEMMA:K53-L2-link-stability}}n^2$ edges from $G$, we obtain a graph which is balanced $K_4$-free with respect to the partition $V_1\cup V_2$.
    If
    \begin{align*}
        |G|+|G[V_1,V_2]|\ge \left(\frac58-\delta_{\ref{LEMMA:K53-L2-link-stability}}\right)n^2,
    \end{align*}
    then
    \begin{align*}
        |G[V_1]|\le \varepsilon n^2\qquad\text{or}\qquad |G[V_2]|\le \varepsilon n^2.
    \end{align*}
\end{lemma}

\begin{proof}[Proof of Lemma~\ref{LEMMA:K53-L2-link-stability}]
    Fix $\varepsilon>0$, and choose $\alpha>0$ with $0<\alpha<\varepsilon^2/4$.
    Choose $\delta_{\ref{LEMMA:K53-L2-link-stability}}>0$ sufficiently small so that
    \begin{align*}
        \delta_{\ref{LEMMA:K53-L2-link-stability}}\le \delta_{\ref{LEMMA:K53-L2-link-bound}}(\alpha)
        \qquad\text{and}\qquad
        2\delta_{\ref{LEMMA:K53-L2-link-stability}}+\alpha<\varepsilon^2.
    \end{align*}
    Choose $N_{\ref{LEMMA:K53-L2-link-stability}}\ge N_{\ref{LEMMA:K53-L2-link-bound}}(\alpha)$.
    Let $n\ge N_{\ref{LEMMA:K53-L2-link-stability}}$.
    Since $\delta_{\ref{LEMMA:K53-L2-link-stability}}\le \delta_{\ref{LEMMA:K53-L2-link-bound}}(\alpha)$, Lemma~\ref{LEMMA:K53-L2-link-bound} applied with parameter $\alpha$ gives
    \begin{align*}
        |G|\le \left(\frac38+\alpha\right)n^2.
    \end{align*}
    It follows that
    \begin{align*}
        |G[V_1,V_2]|\ge 
        \left(\frac58-\delta_{\ref{LEMMA:K53-L2-link-stability}}\right)n^2 - \left(\frac38+\alpha\right)n^2
        =\left(\frac14-\delta_{\ref{LEMMA:K53-L2-link-stability}}-\alpha\right)n^2.
    \end{align*}
    Let
    \begin{align*}
        M_{12}\coloneqq \{xy:x\in V_1,\ y\in V_2,\ xy\notin G\}.
    \end{align*}
    Since $|V_1||V_2|\le n^2/4$, it follows that
    \begin{align}
        |M_{12}|\le \bigl(\delta_{\ref{LEMMA:K53-L2-link-stability}}+\alpha\bigr)n^2. \label{equ:K53-L2-link-stability-cross-missing}
    \end{align}
    
    Let $D$ be a set of at most $\delta_{\ref{LEMMA:K53-L2-link-stability}}n^2$ edges of $G$ such that $G-D$ is balanced $K_4$-free with respect to the partition $V_1\cup V_2$.
    Suppose, for a contradiction, that $|G[V_1]|>\varepsilon n^2$ and $|G[V_2]|>\varepsilon n^2$.
    Let
    \begin{align*}
        \mathcal T\coloneqq \{\{x,x',y,y'\}:x,x'\in V_1,\ y,y'\in V_2,\ xx'\in G[V_1],\ yy'\in G[V_2]\}.
    \end{align*}
    Then
    \begin{align*}
        |\mathcal T|=|G[V_1]|\,|G[V_2]|>\varepsilon^2n^4.
    \end{align*}
    Every member of $\mathcal T$ either spans a balanced copy of $K_4$ in $G$, or contains a missing cross pair from $M_{12}$. Since $G-D$ is balanced $K_4$-free, each member of $\mathcal T$ which spans a balanced copy of $K_4$ contains an edge of $D$. Each edge of $D$ and each missing cross pair from $M_{12}$ belongs to at most $n^2$ members of $\mathcal T$. Indeed, after fixing such a pair, there are at most $n^2$ choices for the remaining vertices. It follows from~\eqref{equ:K53-L2-link-stability-cross-missing} and the choice of $\delta_{\ref{LEMMA:K53-L2-link-stability}}$ that
    \begin{align*}
        |\mathcal T|\le \bigl(|D|+|M_{12}|\bigr)n^2\le \bigl(2\delta_{\ref{LEMMA:K53-L2-link-stability}}+\alpha\bigr)n^4<\varepsilon^2n^4,
    \end{align*}
    contradicting $|\mathcal T|>\varepsilon^2n^4$. Hence at least one of $G[V_1]$ and $G[V_2]$ has at most $\varepsilon n^2$ edges, completing the proof of Lemma~\ref{LEMMA:K53-L2-link-stability}. 
\end{proof}

\subsection{Local degree control near \texorpdfstring{$\mathbb B_n$}{B n}}\label{SUBSEC:K53-L2-local-degree-control}

We next record the two $\ell_2$-specific inputs and then prove the local degree control. Since the shadow of $K_5^3$ is complete, duplicating a vertex in a $K_5^3$-free $3$-graph does not create a copy of $K_5^3$. Thus the standard deletion--duplication argument, see for example~\cite[Lemma~2.6]{HLZ25Fano}, gives the following near-regularity lemma for the $\ell_2$-degree.

\begin{lemma}\label{LEMMA:K53-L2-degree-regularity}
    Suppose that $\mathcal H$ is a $K_5^3$-free $n$-vertex $3$-graph with $\norm{\mathcal H}_{2}=\mathrm{ex}_{\ell_2}(n,K_5^3)$.
    Then, for every pair of vertices $u,v\in V(\mathcal H)$, we have $\left|s_{\mathcal H}(u)-s_{\mathcal H}(v)\right|\le 60n^2$.
    In particular, for every vertex $v\in V(\mathcal H)$, we have $\left|s_{\mathcal H}(v)-s(\mathcal H)\right|\le 60n^2$.
\end{lemma}

We will also use the following stability result of Balogh, Clemen, and Lidick\'{y}~\cite{BCL22b}.

\begin{theorem}[{\cite[Theorem~1.6]{BCL22b}}]\label{THM:K53-L2-BCL-stability}
    For every $\varepsilon>0$, there exist $\delta_{\ref{THM:K53-L2-BCL-stability}}\coloneqq\delta_{\ref{THM:K53-L2-BCL-stability}}(\varepsilon)>0$ and $N_{\ref{THM:K53-L2-BCL-stability}}\coloneqq N_{\ref{THM:K53-L2-BCL-stability}}(\varepsilon)$ such that the following holds for all $n\ge N_{\ref{THM:K53-L2-BCL-stability}}$.
    If $\mathcal H$ is an $n$-vertex $K_5^3$-free $3$-graph with
    \begin{align*}
        \norm{\mathcal H}_{2}\ge \left(\frac58-\delta_{\ref{THM:K53-L2-BCL-stability}}\right)\frac{n^4}{2}, 
    \end{align*}
    then $\mathcal H$ is $\varepsilon$-close to $\mathbb B_n$.
\end{theorem}

The next lemma is the $\ell_2$-specific step which will be combined with the common local control in Lemma~\ref{LEMMA:common-local-bad-missing-control}.

\begin{lemma}\label{LEMMA:K53-L2-local-degree-control}
    Let $\eta>0$ be sufficiently small. There exist $\delta_{\ref{LEMMA:K53-L2-local-degree-control}}\coloneqq\delta_{\ref{LEMMA:K53-L2-local-degree-control}}(\eta)>0$ and $N_{\ref{LEMMA:K53-L2-local-degree-control}}\coloneqq N_{\ref{LEMMA:K53-L2-local-degree-control}}(\eta)$ such that the following holds for all $n\ge N_{\ref{LEMMA:K53-L2-local-degree-control}}$.
    Suppose that $\mathcal H$ is an $n$-vertex $K_5^3$-free $3$-graph with $\norm{\mathcal H}_{2}=\mathrm{ex}_{\ell_2}(n,K_5^3)$, and that $\mathcal H$ is $\delta$-close to $\mathbb B_n$ for some $0<\delta\le \delta_{\ref{LEMMA:K53-L2-local-degree-control}}$.
    Let $V_1\cup V_2=V(\mathcal H)$ be a maximum partition of $\mathcal H$, and let $\mathcal B$ and $\mathcal M$ be the bad and missing edges with respect to this partition.
    Then $\max\{\Delta(\mathcal B),\Delta(\mathcal M)\}\le 2\eta n^2$.
\end{lemma}

\begin{proof}[Proof of Lemma~\ref{LEMMA:K53-L2-local-degree-control}]
    Choose $\alpha>0$ sufficiently small so that $3\alpha\le \delta_{\ref{LEMMA:K53-L2-link-stability}}(\eta)/2$ and $4\alpha\le \eta$.
    Set $\rho\coloneqq \min\{\delta_{\ref{LEMMA:K53-L2-link-bound}}(\alpha),\delta_{\ref{LEMMA:K53-L2-link-stability}}(\eta)\}$. This will be the deletion tolerance supplied by Lemma~\ref{LEMMA:common-balanced-k4-missing}, so that the link graph satisfies the deletion hypotheses of both Lemma~\ref{LEMMA:K53-L2-link-bound} and Lemma~\ref{LEMMA:K53-L2-link-stability}.
    Choose $0<\delta_{\ref{LEMMA:K53-L2-local-degree-control}}<1$ sufficiently small so that Lemma~\ref{LEMMA:common-near-bipartite-setup} is applicable with parameter $\delta_{\ref{LEMMA:K53-L2-local-degree-control}}$ and
    \begin{align}
        \delta_{\ref{LEMMA:K53-L2-local-degree-control}}\le \delta_{\ref{LEMMA:common-balanced-k4-missing}}(\rho),\quad 10\delta_{\ref{LEMMA:K53-L2-local-degree-control}}^{1/2}\le \rho,\quad 42\delta_{\ref{LEMMA:K53-L2-local-degree-control}}^{1/2}\le \alpha,\quad 15\delta_{\ref{LEMMA:K53-L2-local-degree-control}}^{1/2}+150\delta_{\ref{LEMMA:K53-L2-local-degree-control}}\le \eta-3\alpha. \label{equ:K53-L2-local-degree-delta-choice}
    \end{align}
    Finally, choose $N_{\ref{LEMMA:K53-L2-local-degree-control}}$ sufficiently large so that all lemmas used below are applicable with the parameters fixed above.
    
    Since $\mathcal H$ is $\delta$-close to $\mathbb B_n$ and $\delta\le\delta_{\ref{LEMMA:K53-L2-local-degree-control}}$, it is also $\delta_{\ref{LEMMA:K53-L2-local-degree-control}}$-close to $\mathbb B_n$. Lemma~\ref{LEMMA:common-near-bipartite-setup} applied with parameter $\delta_{\ref{LEMMA:K53-L2-local-degree-control}}$ gives
    \begin{align}
        \max\{|\mathcal B|,|\mathcal M|\}\le 2\delta_{\ref{LEMMA:K53-L2-local-degree-control}} n^3,\qquad \left||V_i|-n/2\right|\le 10\delta_{\ref{LEMMA:K53-L2-local-degree-control}}^{1/2}n\text{ for each }i\in[2],\qquad |\mathcal H|\le |\mathbb B_n|+\delta_{\ref{LEMMA:K53-L2-local-degree-control}} n^3. \label{equ:K53-L2-local-degree-setup}
    \end{align}
    
    Fix a vertex $v\in V(\mathcal H)$. By symmetry, assume that $v\in V_1$, and write $L\coloneqq L_{\mathcal H}(v)$. Also set $L^{\mathrm{cr}}\coloneqq L[V_1,V_2]$.
    By Lemma~\ref{LEMMA:K53-L2-degree-regularity}, \eqref{equ:def-2norm-degree-average}, the extremality of $\mathcal H$, the upper bound on $|\mathcal H|$ in~\eqref{equ:K53-L2-local-degree-setup}, and the choice of $N_{\ref{LEMMA:K53-L2-local-degree-control}}$, we have
    \begin{align}
        s_{\mathcal H}(v)\ge s(\mathcal H)-60n^2=\frac{4\norm{\mathcal H}_{2}-3|\mathcal H|}{n}-60n^2\ge \frac{4\norm{\mathbb B_n}_{2}-3\bigl(|\mathbb B_n|+\delta_{\ref{LEMMA:K53-L2-local-degree-control}} n^3\bigr)}{n}-60n^2\ge \frac54n^3-\alpha n^3. \label{equ:K53-L2-local-degree-s-lower}
    \end{align}
    The last inequality follows from the exact formulae for $|\mathbb B_n|$ and $\norm{\mathbb B_n}_{2}$, by taking $N_{\ref{LEMMA:K53-L2-local-degree-control}}$ sufficiently large.
    For every pair $e$, we have $d_{\mathcal H}(e)\le d_{\mathbb B[V_1,V_2]}(e)+d_{\mathcal B}(e)$. Using~\eqref{equ:K53-L2-local-degree-setup} and also $|L|+|L^{\mathrm{cr}}|\le n^2$, we obtain
    \begin{align}
        \sum_{e\in L}d_{\mathcal H}(e)
        \le \left(\frac n2+10\delta_{\ref{LEMMA:K53-L2-local-degree-control}}^{1/2}n\right)|L\setminus L^{\mathrm{cr}}|+(n-2)|L^{\mathrm{cr}}|+3|\mathcal B| 
        \le \frac n2\bigl(|L|+|L^{\mathrm{cr}}|\bigr)+16\delta_{\ref{LEMMA:K53-L2-local-degree-control}}^{1/2}n^3. \label{equ:K53-L2-local-degree-codegree-upper}
    \end{align}
    By~\eqref{equ:def-2norm-degree-b} and~\eqref{equ:K53-L2-local-degree-codegree-upper},
    \begin{align}
        s_{\mathcal H}(v)\le \norm{L}_{2}+n\bigl(|L|+|L^{\mathrm{cr}}|\bigr)+32\delta_{\ref{LEMMA:K53-L2-local-degree-control}}^{1/2}n^3. \label{equ:K53-L2-local-degree-s-upper}
    \end{align}
    Combining~\eqref{equ:K53-L2-local-degree-s-lower} and~\eqref{equ:K53-L2-local-degree-s-upper}, we get
    \begin{align}
        \norm{L}_{2}+n\bigl(|L|+|L^{\mathrm{cr}}|\bigr)\ge \frac54n^3-\alpha n^3-32\delta_{\ref{LEMMA:K53-L2-local-degree-control}}^{1/2}n^3\ge \frac54n^3-2\alpha n^3, \label{equ:K53-L2-local-degree-link-energy-lower}
    \end{align}
    where the last inequality follows from~\eqref{equ:K53-L2-local-degree-delta-choice}.
    
    Since $\delta\le\delta_{\ref{LEMMA:K53-L2-local-degree-control}}\le\delta_{\ref{LEMMA:common-balanced-k4-missing}}(\rho)$, the graph $\mathcal H$ is also $\delta_{\ref{LEMMA:common-balanced-k4-missing}}(\rho)$-close to $\mathbb B_n$. Thus Lemma~\ref{LEMMA:common-balanced-k4-missing} applies with parameter $\rho$. By its moreover part, after deleting at most $\rho n^2$ edges from $L$, we obtain a graph which is balanced $K_4$-free with respect to the partition $V_1\cup V_2$.
    Since $\rho\le\delta_{\ref{LEMMA:K53-L2-link-bound}}(\alpha)$, the deletion set has size at most $\delta_{\ref{LEMMA:K53-L2-link-bound}}(\alpha)n^2$. By~\eqref{equ:K53-L2-local-degree-delta-choice} and~\eqref{equ:K53-L2-local-degree-setup}, the partition satisfies $\bigl||V_i|-n/2\bigr|\le \rho n$ for each $i\in[2]$. Lemma~\ref{LEMMA:K53-L2-link-bound} gives $\norm{L}_{2}\le \left(\frac58+\alpha\right)n^3$.
    Together with~\eqref{equ:K53-L2-local-degree-link-energy-lower}, this gives
    \begin{align}
        |L|+|L^{\mathrm{cr}}|\ge \left(\frac58-3\alpha\right)n^2\ge \left(\frac58-\delta_{\ref{LEMMA:K53-L2-link-stability}}(\eta)\right)n^2. \label{equ:K53-L2-local-degree-link-size-lower}
    \end{align}
    Since $\rho\le\delta_{\ref{LEMMA:K53-L2-link-stability}}(\eta)$, the partition satisfies $\bigl||V_i|-n/2\bigr|\le \delta_{\ref{LEMMA:K53-L2-link-stability}}(\eta)n$ for each $i\in[2]$, and the same deletion set has size at most $\delta_{\ref{LEMMA:K53-L2-link-stability}}(\eta)n^2$. Thus Lemma~\ref{LEMMA:K53-L2-link-stability} applied to $L$ gives $|L[V_1]|\le \eta n^2$ or $|L[V_2]|\le \eta n^2$. Since $V_1\cup V_2$ is a maximum partition, Fact~\ref{FACT:common-maximum-partition-link} gives $|L[V_2]|\ge |L[V_1]|$. It follows that $d_{\mathcal B}(v)=|L[V_1]|\le \eta n^2$. As $v$ was arbitrary, $\Delta(\mathcal B)\le \eta n^2$.
    
    It remains to bound the missing degree. By~\eqref{equ:K53-L2-local-degree-link-size-lower} and $|L^{\mathrm{cr}}|\le |V_1||V_2|\le n^2/4$, we have $|L|\ge \left(\frac38-3\alpha\right)n^2$.
    For $v\in V_1$, the number of triples of $\mathbb B[V_1,V_2]$ containing $v$ is
    \begin{align}
        d_{\mathbb B[V_1,V_2]}(v)=\binom{|V_2|}{2}+(|V_1|-1)|V_2|\le \frac32\left(\frac n2+10\delta_{\ref{LEMMA:K53-L2-local-degree-control}}^{1/2}n\right)^2. \label{equ:K53-L2-local-degree-template-degree}
    \end{align}
    Therefore, using $d_{\mathcal B}(v)\le \eta n^2$, the lower bound on $|L|$, and~\eqref{equ:K53-L2-local-degree-template-degree}, we obtain
    \begin{align*}
        d_{\mathcal M}(v)
        &=d_{\mathbb B[V_1,V_2]}(v)-\bigl(|L^{\mathrm{cr}}|+|L[V_2]|\bigr)\\
        &\le d_{\mathbb B[V_1,V_2]}(v)-\bigl(|L|-|L[V_1]|\bigr)\\
        &\le \frac32\left(\frac n2+10\delta_{\ref{LEMMA:K53-L2-local-degree-control}}^{1/2}n\right)^2-\left(\frac38-3\alpha\right)n^2+\eta n^2\\
        &\le \bigl(15\delta_{\ref{LEMMA:K53-L2-local-degree-control}}^{1/2}+150\delta_{\ref{LEMMA:K53-L2-local-degree-control}}+3\alpha+\eta\bigr)n^2
        \le 2\eta n^2,
    \end{align*}
    where the last inequality follows from the choice of $\delta_{\ref{LEMMA:K53-L2-local-degree-control}}$.
    The same argument applies when $v\in V_2$, and hence $\Delta(\mathcal M)\le 2\eta n^2$.
    Together with the bound on $\Delta(\mathcal B)$ above, this proves the lemma.
\end{proof}

\subsection{The local replacement step}\label{SUBSEC:K53-L2-local-modification}

We first record the bipartite calculation needed at the end of the exact argument. It says that among all $n$-vertex bipartite $3$-graphs, the complete balanced bipartite construction has maximum $\ell_2$-norm.

\begin{lemma}\label{LEMMA:K53-L2-bipartite-max}
    Let $n\ge4$, and let $\mathcal G$ be an $n$-vertex bipartite $3$-graph. Then
    \begin{align*}
        \norm{\mathcal G}_{2}\le \norm{\mathbb B_n}_{2},
    \end{align*}
    and equality holds if and only if $\mathcal G\cong\mathbb B_n$.
\end{lemma}

\begin{proof}[Proof of Lemma~\ref{LEMMA:K53-L2-bipartite-max}]
    Suppose that $\mathcal G$ is an $n$-vertex bipartite $3$-graph with bipartition $V_1\cup V_2=V(\mathcal G)$. Since adding edges cannot decrease the $\ell_2$-norm, we have
    \begin{align*}
        \norm{\mathcal G}_{2}\le \norm{\mathbb B[V_1,V_2]}_{2},
    \end{align*}
    and equality in this inequality is possible only when $\mathcal G$ is the complete bipartite $3$-graph with this bipartition, namely $\mathcal G=\mathbb B[V_1,V_2]$.

    Put $m\coloneqq |V_1||V_2|$. In $\mathbb B[V_1,V_2]$, pairs inside $V_1$ have codegree $|V_2|$, pairs inside $V_2$ have codegree $|V_1|$, and cross pairs have codegree $n-2$. Hence
    \begin{align*}
        \norm{\mathbb B[V_1,V_2]}_{2}=\binom{|V_1|}{2}|V_2|^2+\binom{|V_2|}{2}|V_1|^2+|V_1||V_2|(n-2)^2=m\left(m-\frac n2+(n-2)^2\right).
    \end{align*}
    For $n\ge4$, the expression on the right is strictly increasing in $m$. Since $m\le \lfloor n^2/4\rfloor$, with equality if and only if the bipartition is balanced, we have $\norm{\mathbb B[V_1,V_2]}_{2}\le\norm{\mathbb B_n}_{2}$. Equality holds if and only if $\mathcal G=\mathbb B[V_1,V_2]$ and the bipartition is balanced. This proves Lemma~\ref{LEMMA:K53-L2-bipartite-max}.
\end{proof}

We also need the local replacement estimate. It will be applied only to $3$-graphs obtained during the replacement procedure, so no forbidden-subgraph condition is needed in its statement.

\begin{lemma}\label{LEMMA:K53-L2-switching}
    There exists $\xi_{\ref{LEMMA:K53-L2-switching}}>0$ such that for every $0<\xi\le \xi_{\ref{LEMMA:K53-L2-switching}}$ there exists $N_{\ref{LEMMA:K53-L2-switching}}\coloneqq N_{\ref{LEMMA:K53-L2-switching}}(\xi)$ such that the following holds for all $n\ge N_{\ref{LEMMA:K53-L2-switching}}$.
    Let $\mathcal G$ be an $n$-vertex $3$-graph with a bipartition $V_1\cup V_2=V(\mathcal G)$, and let $\mathcal B$ and $\mathcal M$ be the bad and missing edges with respect to this partition.
    Let $e_{\ast}\in\partial\mathcal G$ be a pair in the shadow such that $e_{\ast}\in\overline K[V_1,V_2]$. Suppose that $e_\ast$ satisfies
    \begin{enumerate}[label=(\roman*), ref=(\roman*)]
        \item\label{LEMMA:K53-L2-switching-balance} $\bigl||V_i|-n/2\bigr|\le \xi n$ for $i\in[2]$,
        \item\label{LEMMA:K53-L2-switching-degree} $\max\{\Delta(\mathcal B),\Delta(\mathcal M)\}\le \xi n^2$,
        \item\label{LEMMA:K53-L2-switching-heavy} $d_{\mathcal M}(e_\ast)\ge 47\xi^{1/2}n$,
        \item\label{LEMMA:K53-L2-switching-comparison} $d_{\mathcal M}(e_\ast)\ge d_{\mathcal B}(e_\ast)-\xi n$.
    \end{enumerate}
    Let $\mathcal G^\ast\coloneqq(\mathcal G\setminus\mathcal B(e_\ast))\cup\mathcal M(e_\ast)$. Then
    \begin{align*}
        \norm{\mathcal G^\ast}_{2}>\norm{\mathcal G}_{2}.
    \end{align*}
\end{lemma}

\begin{proof}[Proof of Lemma~\ref{LEMMA:K53-L2-switching}]
    Choose $\xi_{\ref{LEMMA:K53-L2-switching}}>0$ sufficiently small. Fix $0<\xi\le \xi_{\ref{LEMMA:K53-L2-switching}}$, and choose $N_{\ref{LEMMA:K53-L2-switching}}$ sufficiently large so that $2\le 3\xi^{1/2}n$ for all $n\ge N_{\ref{LEMMA:K53-L2-switching}}$. By symmetry, assume that $e_\ast=\{u_1,u_2\}\subseteq V_1$. Let
    \begin{align*}
        S_1&\coloneqq \{u_1w:w\in N_{\mathcal M}(e_\ast)\}\cup\{u_2w:w\in N_{\mathcal M}(e_\ast)\},\\
        S_2&\coloneqq \{u_1w:w\in N_{\mathcal B}(e_\ast)\}\cup\{u_2w:w\in N_{\mathcal B}(e_\ast)\}.
    \end{align*}
    Then $|S_1|=2d_{\mathcal M}(e_\ast)$ and $|S_2|=2d_{\mathcal B}(e_\ast)$. Apart from $e_\ast$, the only affected pairs are those in $S_1\cup S_2$. Each pair in $S_1$ has its codegree increased by one, and each pair in $S_2$ has its codegree decreased by one. Hence
    \begin{align}
        \norm{\mathcal G^\ast}_{2}-\norm{\mathcal G}_{2}
        &\ge \bigl(d_{\mathcal M}(e_\ast)-d_{\mathcal B}(e_\ast)\bigr)\bigl(d_{\mathcal G^\ast}(e_\ast)+d_{\mathcal G}(e_\ast)\bigr) +2\sum_{e\in S_1}d_{\mathcal G}(e)-2\sum_{e\in S_2}d_{\mathcal G}(e) \notag\\
        &\ge -2\xi n^2+2\sum_{e\in S_1}d_{\mathcal G}(e)-2\sum_{e\in S_2}d_{\mathcal G}(e), \label{equ:K53-L2-switching-difference-start}
    \end{align}
    where the last inequality follows from assumption~\ref{LEMMA:K53-L2-switching-comparison} and $d_{\mathcal G^\ast}(e_\ast)+d_{\mathcal G}(e_\ast)\le2n$.

    Let $S_1^h\coloneqq\{e\in S_1:d_{\mathcal M}(e)\ge \xi^{1/2}n\}$. Every missing edge counted by $\sum_{e\in S_1^h}d_{\mathcal M}(e)$ contains $u_1$ or $u_2$, and is counted at most twice. Thus
    \begin{align*}
        d_{\mathcal M}(u_1)+d_{\mathcal M}(u_2)\ge \frac12\sum_{e\in S_1^h}d_{\mathcal M}(e)\ge \frac12|S_1^h|\xi^{1/2}n.
    \end{align*}
    By assumption~\ref{LEMMA:K53-L2-switching-degree}, this gives $|S_1^h|\le4\xi^{1/2}n$. For every $e\in S_1\setminus S_1^h$, the pair $e$ is a cross pair, and hence $d_{\mathcal G}(e)=d_{\mathbb B[V_1,V_2]}(e)-d_{\mathcal M}(e)\ge n-2-\xi^{1/2}n\ge n-4\xi^{1/2}n$. Since $|S_1|\le2n$, it follows that
    \begin{align}
        \sum_{e\in S_1}d_{\mathcal G}(e)\ge n|S_1|-12\xi^{1/2}n^2. \label{equ:K53-L2-switching-S1-lower}
    \end{align}

    Similarly, let $S_2^h\coloneqq\{e\in S_2:d_{\mathcal B}(e)\ge \xi^{1/2}n\}$. Assumption~\ref{LEMMA:K53-L2-switching-degree} gives $|S_2^h|\le4\xi^{1/2}n$. For every $e\in S_2\setminus S_2^h$, the pair $e$ is a same-side pair, and by assumption~\ref{LEMMA:K53-L2-switching-balance},
    \begin{align*}
        d_{\mathcal G}(e)\le d_{\mathbb B[V_1,V_2]}(e)+d_{\mathcal B}(e)\le \frac n2+\xi n+\xi^{1/2}n.
    \end{align*}
    Since $|S_2|\le2n$ and $d_{\mathcal G}(e)\le n$ for all pairs $e$, the choice of $\xi_{\ref{LEMMA:K53-L2-switching}}$ and the inequality $\xi\le\xi_{\ref{LEMMA:K53-L2-switching}}$ give
    \begin{align}
        \sum_{e\in S_2}d_{\mathcal G}(e)\le |S_2|\left(\frac n2+\xi n+\xi^{1/2}n\right)+|S_2^h|n<\frac n2|S_2|+10\xi^{1/2}n^2. \label{equ:K53-L2-switching-S2-upper}
    \end{align}
    Combining~\eqref{equ:K53-L2-switching-difference-start},~\eqref{equ:K53-L2-switching-S1-lower}, and~\eqref{equ:K53-L2-switching-S2-upper}, we obtain
    \begin{align*}
        \norm{\mathcal G^\ast}_{2}-\norm{\mathcal G}_{2}
        &> -2\xi n^2+2\bigl(n|S_1|-12\xi^{1/2}n^2\bigr)-2\left(\frac n2|S_2|+10\xi^{1/2}n^2\right)\\
        &=2n\bigl(2d_{\mathcal M}(e_\ast)-d_{\mathcal B}(e_\ast)\bigr)-(44\xi^{1/2}+2\xi)n^2\\
        &\ge 2n\bigl(d_{\mathcal M}(e_\ast)-\xi n\bigr)-(44\xi^{1/2}+2\xi)n^2\\
        &\ge (50\xi^{1/2}-4\xi)n^2>0,
    \end{align*}
    where the last two inequalities use assumptions~\ref{LEMMA:K53-L2-switching-heavy} and~\ref{LEMMA:K53-L2-switching-comparison}, and the final inequality follows from the choice of $\xi_{\ref{LEMMA:K53-L2-switching}}$.
    This completes the proof of Lemma~\ref{LEMMA:K53-L2-switching}.
\end{proof}

\subsection{Proof of Theorem~\ref{THM:L2-exact-K53}}\label{SUBSEC:K53-L2-proof-exact}

We now complete the proof of the exact $\ell_2$-norm Tur\'{a}n theorem for $K_5^3$.

\begin{proof}[Proof of Theorem~\ref{THM:L2-exact-K53}]
    Let $\xi_{\ref{LEMMA:K53-L2-switching}}$ be the constant from Lemma~\ref{LEMMA:K53-L2-switching}. Choose $\eta>0$ sufficiently small so that Lemma~\ref{LEMMA:K53-L2-local-degree-control} and Lemma~\ref{LEMMA:common-local-bad-missing-control} apply with parameter $\eta$. Set $\xi\coloneqq7\eta^{1/2}$, and assume, by decreasing $\eta$ if necessary, that
    \begin{align*}
        \xi<\xi_{\ref{LEMMA:K53-L2-switching}},\qquad 47\xi^{1/2}\le \frac1{10},\qquad 2\eta\le \xi.
    \end{align*}
    Choose $0<\delta_0<1$ sufficiently small so that Lemma~\ref{LEMMA:common-near-bipartite-setup} applies with parameter $\delta_0$, and
    \begin{align*}
        \delta_0\le \delta_{\ref{LEMMA:K53-L2-local-degree-control}}(\eta),\qquad \delta_0\le \delta_{\ref{LEMMA:common-local-bad-missing-control}}(\eta),\qquad 10\delta_0^{1/2}\le \xi.
    \end{align*}
    Choose $n_0$ sufficiently large so that all lemmas used below apply with the parameters fixed above.

    Let $n\ge n_0$, and let $\mathcal H$ be an $n$-vertex $K_5^3$-free $3$-graph with $\norm{\mathcal H}_{2}=\mathrm{ex}_{\ell_2}(n,K_5^3)$. Since $\mathbb B_n$ is $K_5^3$-free, and by the exact formula for $\norm{\mathbb B_n}_{2}$ and the choice of $n_0$, we have
    \begin{align*}
        \norm{\mathcal H}_{2}\ge\norm{\mathbb B_n}_{2}\ge \left(\frac58-\delta_{\ref{THM:K53-L2-BCL-stability}}(\delta_0)\right)\frac{n^4}{2}.
    \end{align*}
    By Theorem~\ref{THM:K53-L2-BCL-stability}, the graph $\mathcal H$ is $\delta_0$-close to $\mathbb B_n$.

    Let $V_1\cup V_2=V(\mathcal H)$ be a maximum partition of $\mathcal H$, and let $\mathcal B$ and $\mathcal M$ be the corresponding sets of bad and missing edges. Lemma~\ref{LEMMA:common-near-bipartite-setup}, together with the choice of $\delta_0$, gives
    \begin{align}
        \max_{i\in[2]}\left||V_i|-\frac n2\right|\le 10\delta_0^{1/2}n\le \xi n. \label{equ:K53-L2-proof-balance}
    \end{align}
    Lemma~\ref{LEMMA:K53-L2-local-degree-control} and the choice of $\eta$ give
    \begin{align}
        \max\{\Delta(\mathcal B),\Delta(\mathcal M)\}\le 2\eta n^2\le \xi n^2. \label{equ:K53-L2-proof-degree}
    \end{align}
    By Lemma~\ref{LEMMA:common-local-bad-missing-control}, every bad edge $E\in\mathcal B$ contains a pair $e\subseteq E$ with $d_{\mathcal M}(e)\ge n/10$. Moreover, for every $e\in\overline K[V_1,V_2]$,
    \begin{align}
        d_{\mathcal M}(e)\ge d_{\mathcal B}(e)-7\eta^{1/2}n=d_{\mathcal B}(e)-\xi n. \label{equ:K53-L2-proof-pair-comparison}
    \end{align}

    Suppose, for a contradiction, that $\mathcal B\ne\emptyset$. Define
    \begin{align*}
        \mathcal I\coloneqq\left\{e\in\overline K[V_1,V_2]:d_{\mathcal M}(e)\ge \frac n{10}\text{ and }\mathcal B(e)\ne\emptyset\right\}=\{e_1,\ldots,e_k\}.
    \end{align*}
    Since every bad edge lies inside one part, the preceding paragraph shows that every bad edge contains some pair from $\mathcal I$.

    Let $\mathcal H_0\coloneqq\mathcal H$, $\mathcal B_0\coloneqq\mathcal B$, and $\mathcal M_0\coloneqq\mathcal M$. For $j\in[k]$, once $\mathcal H_{j-1}$, $\mathcal B_{j-1}$, and $\mathcal M_{j-1}$ have been defined, we proceed as follows. If $\mathcal B_{j-1}(e_j)=\emptyset$, then set $\mathcal H_j\coloneqq\mathcal H_{j-1}$, $\mathcal B_j\coloneqq\mathcal B_{j-1}$, and $\mathcal M_j\coloneqq\mathcal M_{j-1}$. Otherwise set
    \begin{align*}
        \mathcal H_j\coloneqq(\mathcal H_{j-1}\setminus\mathcal B_{j-1}(e_j))\cup\mathcal M_{j-1}(e_j),\qquad \mathcal B_j\coloneqq\mathcal H_j\setminus\mathbb B[V_1,V_2],\qquad \mathcal M_j\coloneqq\mathbb B[V_1,V_2]\setminus\mathcal H_j.
    \end{align*}
    This procedure only deletes bad edges and only adds edges of $\mathbb B[V_1,V_2]$, so
    \begin{align*}
        \mathcal B=\mathcal B_0\supseteq\mathcal B_1\supseteq\cdots\supseteq\mathcal B_k\qquad\text{and}\qquad \mathcal M=\mathcal M_0\supseteq\mathcal M_1\supseteq\cdots\supseteq\mathcal M_k.
    \end{align*}
    Every edge of $\mathbb B[V_1,V_2]$ contains exactly one same-side pair. Hence, before step $j$, the additions made at pairs different from $e_j$ do not affect the missing edges containing $e_j$. Consider a step $j$ with $\mathcal B_{j-1}(e_j)\ne\emptyset$. Then $e_j\in\partial\mathcal H_{j-1}$, and
    \begin{align*}
        d_{\mathcal M_{j-1}}(e_j)=d_{\mathcal M}(e_j)\ge \frac n{10}\ge47\xi^{1/2}n,\qquad d_{\mathcal M_{j-1}}(e_j)\ge d_{\mathcal B_{j-1}}(e_j)-\xi n.
    \end{align*}
    The second inequality follows from the equality above,~\eqref{equ:K53-L2-proof-pair-comparison}, and the inclusion $\mathcal B_{j-1}\subseteq\mathcal B$. Together with~\eqref{equ:K53-L2-proof-balance} and~\eqref{equ:K53-L2-proof-degree}, the inclusions $\mathcal B_{j-1}\subseteq\mathcal B$ and $\mathcal M_{j-1}\subseteq\mathcal M$ show that Lemma~\ref{LEMMA:K53-L2-switching} applies to $\mathcal H_{j-1}$ with the fixed partition $V_1\cup V_2$ and the pair $e_j$. Hence $\norm{\mathcal H_j}_{2}>\norm{\mathcal H_{j-1}}_{2}$ for every step with $\mathcal B_{j-1}(e_j)\ne\emptyset$, while the other steps leave the graph unchanged.
    
    At least one step is nontrivial. Indeed, choose an original bad edge $E\in\mathcal B$ and a pair $e_j\subseteq E$ with $e_j\in\mathcal I$. If $E$ is still present before step $j$, then $\mathcal B_{j-1}(e_j)\ne\emptyset$. Otherwise, $E$ was removed in an earlier nontrivial step. Therefore
    \begin{align*}
        \norm{\mathcal H_k}_{2}>\norm{\mathcal H}_{2}.
    \end{align*}
    Since each original bad edge contains a pair from $\mathcal I$, it is removed when one such pair is processed, unless it was removed earlier. Therefore $\mathcal H_k$ has no bad edge, and so $\mathcal H_k\subseteq\mathbb B[V_1,V_2]$. By Lemma~\ref{LEMMA:K53-L2-bipartite-max},
    \begin{align*}
        \norm{\mathcal H}_{2}<\norm{\mathcal H_k}_{2}\le\norm{\mathbb B_n}_{2}\le\norm{\mathcal H}_{2},
    \end{align*}
    a contradiction. Hence $\mathcal B=\emptyset$.

    If $\mathcal M\ne\emptyset$, then adding any edge of $\mathcal M$ keeps the graph bipartite and hence $K_5^3$-free. It also strictly increases the $\ell_2$-norm, since the codegrees of the three pairs contained in this edge all increase by one. This contradicts the extremality of $\mathcal H$. Therefore $\mathcal M=\emptyset$, and so $\mathcal H=\mathbb B[V_1,V_2]$. Finally, Lemma~\ref{LEMMA:K53-L2-bipartite-max} gives $\norm{\mathcal H}_{2}\le\norm{\mathbb B_n}_{2}$, while extremality gives the reverse inequality. Hence equality holds in Lemma~\ref{LEMMA:K53-L2-bipartite-max}, and so $\mathcal H\cong\mathbb B_n$. This proves that every extremal graph is isomorphic to $\mathbb B_n$, and completes the proof of Theorem~\ref{THM:L2-exact-K53}.
\end{proof}

\section{Counting cliques in \texorpdfstring{$K_{5}^{3}$}{K53}-free \texorpdfstring{$3$}{3}-graphs}\label{SEC:K53-count-cliques}
For a fixed $3$-graph $F$, the number of induced copies of $F$ in $\mathcal H$ is
\begin{align*}
    \mathrm N_{\mathrm{ind}}(F,\mathcal H)
    \coloneqq \left|\left\{S\in\binom{V(\mathcal H)}{|V(F)|}\colon \mathcal H[S]\cong F\right\}\right|.
\end{align*}
For complete $3$-graphs, induced and non-induced copies coincide, and we write $\mathrm N(K_t^3,\mathcal H)\coloneqq \mathrm N_{\mathrm{ind}}(K_t^3,\mathcal H)$.
Since a $K_{5}^{3}$-free $3$-graph has no cliques of size at least $5$, we have
\begin{align}
    k(\mathcal H)
    = 1+n+\binom{n}{2}+|\mathcal H|+\mathrm N(K_4^3,\mathcal H). \label{equ:K53-count-clique-decomposition}
\end{align}
Thus, by~\eqref{equ:K53-count-clique-decomposition}, to prove Theorem~\ref{THM:cliques-exact-K53}, it is enough to maximize
\begin{align*}
    \Psi(\mathcal H)
    \coloneqq |\mathcal H|+\mathrm N(K_4^3,\mathcal H).
\end{align*}
In this section, we first establish a stability result for $K_5^3$-free $3$-graphs with many copies of $K_4^3$. We then apply a local modification argument as in Section~\ref{SEC:K53-L2norm} and finally prove Theorem~\ref{THM:cliques-exact-K53}.

\subsection{The edge-stability result}\label{SUBSEC:K53-count-edge-stability}

We first prove the stability statement that will be used in the exact argument.
Let $E_t$ denote the empty $3$-graph on $t$ vertices. Let $K_4^{3-}$ and $K_5^{3-}$ be obtained from the complete $3$-graphs $K_4^3$ and $K_5^3$, respectively, by deleting one edge. Let $K_5^{3=}$ be the $5$-vertex $3$-graph obtained from $K_5^3$ by deleting two edges whose intersection has size one. Finally, let $J_4$ be the $5$-vertex $3$-graph on $\{z\}\cup W$, where $|W|=4$, whose edges are exactly the six sets $\{z\}\cup e$ with $e\in\binom W2$.

The certificate in Bodn\'{a}r~\cite[Eq.~(3)]{Bod23} expresses the extremal deficit as a sum of positive semidefinite flag terms, with finite-order terms absorbed by taking $n$ sufficiently large. Hence, for every near-extremal sequence, the corresponding squared rooted statistics have vanishing mean. Translating the first, third, fourth, and fifth square terms of that certificate back to the original $3$-graph, and applying the standard extraction of stability information from flag algebra certificates as in~\cite{PST19}, gives the following proposition.

\begin{proposition}[see Bodn\'{a}r~\cite{Bod23}, Eq.~(3)]\label{PROP:K53-count-flag-local}
    For every $\varepsilon>0$, there exist $\delta_{\ref{PROP:K53-count-flag-local}}\coloneqq\delta_{\ref{PROP:K53-count-flag-local}}(\varepsilon)>0$ and $N_{\ref{PROP:K53-count-flag-local}}\coloneqq N_{\ref{PROP:K53-count-flag-local}}(\varepsilon)$ such that the following holds for all $n\ge N_{\ref{PROP:K53-count-flag-local}}$.
    Suppose that $\mathcal H$ is an $n$-vertex $K_5^3$-free $3$-graph with
    \begin{align*}
        \mathrm N(K_4^3,\mathcal H)\ge \left(\frac1{64}-\delta_{\ref{PROP:K53-count-flag-local}}\right)n^4.
    \end{align*}
    Then the following statements hold.
    \begin{enumerate}[label=(\roman*), ref=(\roman*)]
        \item\label{PROP:K53-count-flag-local-degree} All but at most $\varepsilon n$ vertices $v\in V(\mathcal H)$ satisfy
        \begin{align*}
            \left|d_{\mathcal H}(v)-\frac38 n^2\right|\le \varepsilon n^2.
        \end{align*}
        \item\label{PROP:K53-count-flag-local-K4} All but at most $\varepsilon n^3$ edges $E\in\mathcal H$ satisfy
        \begin{align*}
            \left|\left|\left\{x\in V(\mathcal H)\setminus E\colon \mathcal H[E\cup\{x\}]\cong K_4^3\right\}\right|-\frac12 n\right|\le \varepsilon n.
        \end{align*}
        \item\label{PROP:K53-count-flag-local-K4minus} All but at most $\varepsilon n^3$ edges $E\in\mathcal H$ satisfy
        \begin{align*}
            \left|\left|\left\{x\in V(\mathcal H)\setminus E\colon \mathcal H[E\cup\{x\}]\cong K_4^{3-}\right\}\right|-\frac12 n\right|\le \varepsilon n.
        \end{align*}
        \item\label{PROP:K53-count-flag-local-K5minus} For all but at most $\varepsilon n^4$ four-vertex sets $Q$ inducing $K_4^{3-}$,
        \begin{align*}
            \left|\left|\left\{x\in V(\mathcal H)\setminus Q\colon \mathcal H[Q\cup\{x\}]\cong K_5^{3-}\right\}\right|-\frac12 n\right|\le \varepsilon n.
        \end{align*}
    \end{enumerate}
\end{proposition}

The next lemma converts these local consequences into a statement about five-vertex induced subgraphs.\begingroup\renewcommand{\thefootnote}{\fnsymbol{footnote}}\footnote[1]{Lemma~\ref{LEMMA:K53-count-five-types} can also be verified directly by the flag algebra computation. The corresponding code and certificate are available at \href{https://github.com/xliu2022/xliu2022.github.io/tree/main/FlagAlgebra_Certificate/ForbidK53CountK43}{\url{https://github.com/xliu2022/xliu2022.github.io/tree/main/FlagAlgebra_Certificate/ForbidK53CountK43}}.}\endgroup

\begin{lemma}\label{LEMMA:K53-count-five-types}
    Let $\varepsilon>0$ be sufficiently small, and let $n$ be sufficiently large.
    Suppose that $\mathcal H$ is an $n$-vertex $K_5^3$-free $3$-graph satisfying the conclusions of Proposition~\ref{PROP:K53-count-flag-local} with parameter $\varepsilon$.
    Then all but at most $8\varepsilon n^5$ five-vertex sets induce one of $E_5,J_4,K_5^{3-}$.
\end{lemma}

\begin{proof}[Proof of Lemma~\ref{LEMMA:K53-count-five-types}]
    Let $V_0$ be the set of vertices satisfying the estimate in Proposition~\ref{PROP:K53-count-flag-local}~\ref{PROP:K53-count-flag-local-degree}. Then $|V(\mathcal H)\setminus V_0|\le\varepsilon n$. For every $v\in V(\mathcal H)$, we have $0\le d_{\mathcal H}(v)\le n^2/2$, and hence, for $n$ sufficiently large,
    \begin{align*}
        \left|3|\mathcal H|-\frac38n^3\right|
        \le \sum_{v\in V_0}\left|d_{\mathcal H}(v)-\frac38n^2\right|
        +\sum_{v\notin V_0}\left|d_{\mathcal H}(v)-\frac38n^2\right|
        \le \varepsilon n^3+\frac38\varepsilon n^3<\frac32\varepsilon n^3.
    \end{align*}
    Therefore, we have 
    \begin{align}
        \left||\mathcal H|-\frac18 n^3\right|\le \frac12\varepsilon n^3. \label{equ:K53-count-edge-estimate}
    \end{align}

    For an edge $E\in\mathcal H$, let
    \begin{align*}
        A_4(E)\coloneqq \left|\left\{x\in V(\mathcal H)\setminus E\colon \mathcal H[E\cup\{x\}]\cong K_4^3\right\}\right|.
    \end{align*}
    By Proposition~\ref{PROP:K53-count-flag-local}~\ref{PROP:K53-count-flag-local-K4}, all but at most $\varepsilon n^3$ edges $E$ satisfy $|A_4(E)-n/2|\le\varepsilon n$. Since $0\le A_4(E)\le n$ and $|\mathcal H|\le n^3/6$, we get
    \begin{align*}
        \left|\sum_{E\in\mathcal H}A_4(E)-\frac n2|\mathcal H|\right|
        \le \varepsilon n|\mathcal H|+\frac12\varepsilon n^4\le \frac23\varepsilon n^4.
    \end{align*}
    Each copy of $K_4^3$ is counted once for each of its four edges, so $\sum_{E\in\mathcal H}A_4(E)=4\mathrm N(K_4^3,\mathcal H)$. Combining this with~\eqref{equ:K53-count-edge-estimate}, we obtain
    \begin{align*}
        \left|4\mathrm N(K_4^3,\mathcal H)-\frac1{16}n^4\right|
        \le \left|\sum_{E\in\mathcal H}A_4(E)-\frac n2|\mathcal H|\right|+ \left|\frac n2|\mathcal H|-\frac1{16}n^4\right|
        \le \frac23\varepsilon n^4 + \frac14\varepsilon n^4
        = \frac{11}{12}\varepsilon n^4,
    \end{align*}
    and hence
    \begin{align}
        \left|\mathrm N(K_4^3,\mathcal H)-\frac1{64}n^4\right|
        \le \frac{11}{48}\varepsilon n^4
        \le \frac12\varepsilon n^4. \label{equ:K53-count-K4-estimate}
    \end{align}

    Similarly, by Proposition~\ref{PROP:K53-count-flag-local}~\ref{PROP:K53-count-flag-local-K4minus}, we have 
    \begin{align}
        \left|\mathrm N_{\mathrm{ind}}(K_4^{3-},\mathcal H)-\frac1{48}n^4\right|\le \frac12\varepsilon n^4. \label{equ:K53-count-K4minus-estimate}
    \end{align}
    
    Let $\mathcal X_4$ be the family of four-vertex sets inducing none of $E_4,K_4^{3-},K_4^3$.
    Such a four-set contains either one or two edges, so $|\mathcal X_4|\le\sum_{Q\in\mathcal X_4}|\mathcal H[Q]|$. Counting pairs $(E,x)$ with $E\in\mathcal H$ and $x\in V(\mathcal H)$ gives
    \begin{align*}
        n|\mathcal H|
        =4\mathrm N(K_4^3,\mathcal H)+3\mathrm N_{\mathrm{ind}}(K_4^{3-},\mathcal H)
        +\sum_{Q\in\mathcal X_4}|\mathcal H[Q]|+3|\mathcal H|.
    \end{align*}
    Therefore~\eqref{equ:K53-count-edge-estimate}, \eqref{equ:K53-count-K4-estimate}, and~\eqref{equ:K53-count-K4minus-estimate} imply
    \begin{align*}
        |\mathcal X_4|
        &\le n|\mathcal H|-4\mathrm N(K_4^3,\mathcal H)-3\mathrm N_{\mathrm{ind}}(K_4^{3-},\mathcal H)\\
        &\le \left|n|\mathcal H|-\frac18 n^4\right|
        +4\left|\mathrm N(K_4^3,\mathcal H)-\frac1{64}n^4\right|
        +3\left|\mathrm N_{\mathrm{ind}}(K_4^{3-},\mathcal H)-\frac1{48}n^4\right|
        \le 4\varepsilon n^4.
    \end{align*}
    Next, every induced copy of $K_5^{3-}$ contains exactly two induced copies of $K_4^{3-}$.
    For a four-set $Q$ inducing $K_4^{3-}$, let
    \begin{align*}
        A_{5^-}(Q)\coloneqq \left|\left\{x\in V(\mathcal H)\setminus Q\colon \mathcal H[Q\cup\{x\}]\cong K_5^{3-}\right\}\right|.
    \end{align*}
    By Proposition~\ref{PROP:K53-count-flag-local}~\ref{PROP:K53-count-flag-local-K5minus}, all but at most $\varepsilon n^4$ such four-sets $Q$ satisfy $|A_{5^-}(Q)-n/2|\le\varepsilon n$. Since $0\le A_{5^-}(Q)\le n$ and the number of such $Q$ is at most $n^4$, we have
    \begin{align*}
        \left|\sum_{Q}A_{5^-}(Q)-\frac n2\mathrm N_{\mathrm{ind}}(K_4^{3-},\mathcal H)\right|
        \le \frac32\varepsilon n^5,
    \end{align*}
    where the sum is over all four-sets $Q$ inducing $K_4^{3-}$. Since each induced copy of $K_5^{3-}$ contains exactly two induced copies of $K_4^{3-}$, we have $\sum_QA_{5^-}(Q)=2\mathrm N_{\mathrm{ind}}(K_5^{3-},\mathcal H)$. Together with~\eqref{equ:K53-count-K4minus-estimate}, this gives
    \begin{align*}
        \left|\mathrm N_{\mathrm{ind}}(K_5^{3-},\mathcal H)-\frac1{192}n^5\right|\le \varepsilon n^5.
    \end{align*}
    
    A direct check of the possible five-vertex $3$-graphs shows that if a five-set contains no member of $\mathcal X_4$, then it induces one of $E_5,J_4,K_5^{3-},K_5^{3=}$, or $K_5^3$. The last type is excluded because $\mathcal H$ is $K_5^3$-free. The number of five-sets containing a member of $\mathcal X_4$ is at most $4\varepsilon n^5$.
    It remains to bound the number of induced copies of $K_5^{3=}$.
    Counting pairs $(Q,S)$, where $S$ is a five-set and $Q\in\binom S4$ induces $K_4^3$, gives
    \begin{align*}
        n\mathrm N(K_4^3,\mathcal H)
        \ge 3\mathrm N_{\mathrm{ind}}(K_5^{3-},\mathcal H)+\mathrm N_{\mathrm{ind}}(K_5^{3=},\mathcal H).
    \end{align*}
    Indeed, $K_5^{3-}$ and $K_5^{3=}$ contain respectively three and one four-sets inducing $K_4^3$, and all remaining five-sets contribute nonnegatively to the left-hand side. It follows that
    \begin{align*}
        \mathrm N_{\mathrm{ind}}(K_5^{3=},\mathcal H)
        &\le n\left|\mathrm N(K_4^3,\mathcal H)-\frac1{64}n^4\right|
        +3\left|\mathrm N_{\mathrm{ind}}(K_5^{3-},\mathcal H)-\frac1{192}n^5\right|
        \le  \frac12\varepsilon n^5 + 3\varepsilon n^5
        = \frac72\varepsilon n^5.
    \end{align*}
    Hence the number of five-sets inducing none of $E_5,J_4,K_5^{3-}$ is at most
    \begin{align*}
        4\varepsilon n^5+\frac72\varepsilon n^5<8\varepsilon n^5,
    \end{align*}
    completing the proof.
\end{proof}

We also need the exact structure forced by these three five-vertex types.

\begin{lemma}\label{LEMMA:K53-count-exact-five-type-structure}
    Let $\mathcal H$ be a $3$-graph on $n\ge5$ vertices such that every five-vertex set induces one of $E_5,J_4,K_5^{3-}$. Then there is a partition $V(\mathcal H)=A\cup B$ such that $\mathcal H=\mathbb B[A,B]$.
\end{lemma}

\begin{proof}[Proof of Lemma~\ref{LEMMA:K53-count-exact-five-type-structure}]
    Extending a four-vertex set to a five-vertex set shows that every four-vertex set induces one of $E_4,K_4^{3-},K_4^3$. Also, $\mathcal H$ has a nonedge, since otherwise some five-set would induce $K_5^3$. Fix $abc\notin\mathcal H$. For each $v\notin\{a,b,c\}$, the four-set $\{a,b,c,v\}$ induces either $E_4$ or $K_4^{3-}$. Let $A_0$ be the set of vertices $v$ for which $\{a,b,c,v\}$ induces $E_4$, let $B$ be the set of the remaining vertices, and set $A\coloneqq\{a,b,c\}\cup A_0$.
    
    We claim that $A$ and $B$ are independent, and that every edge meeting both $A$ and $B$ is present. First, the definition of $A_0$ gives all nonedges with one vertex in $A_0$ and two vertices in $\{a,b,c\}$. If $u,v\in A_0$, then the five-set $abcuv$ has at least seven nonedges, so it must induce $E_5$. Hence all edges in $\{a,b,c,u,v\}$ are absent. If $u,v,w\in A_0$, then $abuvw$ has all edges absent except possibly $uvw$, and this possible one-edge five-set is not one of the allowed types. Thus $uvw\notin\mathcal H$, and $A$ is independent.
    
    Next let $x,y\in B$. The five-set $abcxy$ contains the six edges obtained by adding $x$ or $y$ to a pair from $\{a,b,c\}$, and it misses $abc$. It cannot induce $E_5$ or $J_4$, and hence it induces $K_5^{3-}$. In particular, all edges with one vertex in $\{a,b,c\}$ and two vertices in $B$ are present. If $x,y,z\in B$, then $abxyz$ already contains all edges except possibly $xyz$. Since the allowed types do not include $K_5^3$, we have $xyz\notin\mathcal H$. Thus $B$ is independent.
    
    It remains to check the crossing edges involving vertices of $A_0$. Let $u\in A_0$ and $x\in B$. The five-set $abcux$ has the four forced nonedges $abc,abu,acu,bcu$ and the three forced edges $abx,acx,bcx$, so it must induce $J_4$, and therefore $aux,bux,cux\in\mathcal H$. If $u,v\in A_0$ and $x\in B$, then $abuvx$ has five forced edges and four forced nonedges, so the only allowed completion gives $uvx\in\mathcal H$. Finally, if $u\in A_0$ and $x,y\in B$, then $abuxy$ has the forced nonedge $abu$ and the eight forced edges other than $uxy$, so it must induce $K_5^{3-}$, giving $uxy\in\mathcal H$. Hence every crossing edge is present, and $\mathcal H=\mathbb B[A,B]$, completing the proof of Lemma~\ref{LEMMA:K53-count-exact-five-type-structure}. 
\end{proof}

\begin{theorem}\label{THM:K53-count-edge-stability}
    For every $\varepsilon>0$, there exist $\delta_{\ref{THM:K53-count-edge-stability}}\coloneqq\delta_{\ref{THM:K53-count-edge-stability}}(\varepsilon)>0$ and $N_{\ref{THM:K53-count-edge-stability}}\coloneqq N_{\ref{THM:K53-count-edge-stability}}(\varepsilon)$ such that the following holds for all $n\ge N_{\ref{THM:K53-count-edge-stability}}$.
    If $\mathcal H$ is an $n$-vertex $K_5^3$-free $3$-graph with
    \begin{align*}
        \mathrm N(K_4^3,\mathcal H)\ge \left(\frac1{64}-\delta_{\ref{THM:K53-count-edge-stability}}\right)n^4,
    \end{align*}
    then $\mathcal H$ is $\varepsilon$-close to $\mathbb B_n$.
\end{theorem}

\begin{proof}[Proof of Theorem~\ref{THM:K53-count-edge-stability}]
    It is enough to consider $0<\varepsilon\le1$.
    Apply the Induced Removal Lemma, see, e.g.,~\cite{RS09}, to the finite family
    \begin{align*}
        \mathcal F_5\coloneqq \{F\colon |V(F)|=5,\ F\not\cong E_5,J_4,K_5^{3-}\}
    \end{align*}
    with edit parameter $\varepsilon^2/128$, and let $\gamma>0$ be the corresponding counting threshold. Choose $0<\eta\le \gamma/8$ sufficiently small so that Lemma~\ref{LEMMA:K53-count-five-types} applies with parameter $\eta$. Then every $3$-graph with at most $8\eta n^5$ induced copies of members of $\mathcal F_5$ can be made induced-$\mathcal F_5$-free by changing at most $\frac{\varepsilon^2}{128}n^3$ edges.
    Choose
    \begin{align*}
        \delta_{\ref{THM:K53-count-edge-stability}}
        \le \min\left\{\delta_{\ref{PROP:K53-count-flag-local}}(\eta),\frac{\varepsilon^2}{128}\right\},
    \end{align*}
    and choose $N_{\ref{THM:K53-count-edge-stability}}$ sufficiently large so that all results used below apply and $N_{\ref{THM:K53-count-edge-stability}}\ge 2/\varepsilon$.
    
    Let $n\ge N_{\ref{THM:K53-count-edge-stability}}$, and let $\mathcal H$ satisfy the assumptions. Proposition~\ref{PROP:K53-count-flag-local} and Lemma~\ref{LEMMA:K53-count-five-types} give at most $8\eta n^5$ induced copies of members of $\mathcal F_5$. Hence, by the choice of $\eta$, there is a $3$-graph $\mathcal H'$ on the same vertex set such that
    \begin{align*}
        |\mathcal H\triangle \mathcal H'|\le \frac{\varepsilon^2}{128}n^3,
    \end{align*}
    and every five-vertex set of $\mathcal H'$ induces one of $E_5,J_4,K_5^{3-}$. By Lemma~\ref{LEMMA:K53-count-exact-five-type-structure}, there is a partition $V(\mathcal H')=A\cup B$ such that $\mathcal H'=\mathbb B[A,B]$.
    
    Changing one edge can affect at most $n$ four-vertex sets. Thus
    \begin{align*}
        \mathrm N(K_4^3,\mathcal H')\ge \left(\frac1{64}-\delta_{\ref{THM:K53-count-edge-stability}}-\frac{\varepsilon^2}{128}\right)n^4.
    \end{align*}
    In $\mathbb B[A,B]$, the copies of $K_4^3$ are exactly the four-sets with two vertices in each part. Put $x\coloneqq |A|/n$ and $t\coloneqq |x-1/2|$. Then
    \begin{align*}
        \frac1{64}-\delta_{\ref{THM:K53-count-edge-stability}}-\frac{\varepsilon^2}{128}
        \le \frac1{n^4}\binom{|A|}{2}\binom{|B|}{2}
        \le \frac{(|A||B|)^2}{4n^4}
        \le \frac{x^2(1-x)^2}{4}
        =\frac1{64}-\frac{t^2}{8}+\frac{t^4}{4}
        \le \frac1{64}-\frac{t^2}{16}.
    \end{align*}
    Therefore
    \begin{align*}
        \left||A|-\frac n2\right|\le 4\left(\delta_{\ref{THM:K53-count-edge-stability}}+\frac{\varepsilon^2}{128}\right)^{1/2}n.
    \end{align*}
    Move at most $4(\delta_{\ref{THM:K53-count-edge-stability}}+\varepsilon^2/128)^{1/2}n+1$ vertices from the larger part to the smaller one, and let $U_1\cup U_2=V(\mathcal H)$ be the resulting balanced partition. Moving one vertex changes at most $n^2/2$ edges of the corresponding bipartite $3$-graph. Hence
    \begin{align*}
        |\mathcal H\triangle\mathbb B[U_1,U_2]|
        \le \left(\frac{\varepsilon^2}{128}+2\left(\delta_{\ref{THM:K53-count-edge-stability}}+\frac{\varepsilon^2}{128}\right)^{1/2}+\frac1{2n}\right)n^3
        \le \varepsilon n^3,
    \end{align*}
    by the choice of $\delta_{\ref{THM:K53-count-edge-stability}}$ and $N_{\ref{THM:K53-count-edge-stability}}$. This proves that $\mathcal H$ is $\varepsilon$-close to $\mathbb B_n$.
\end{proof}

\subsection{Local degree control near \texorpdfstring{$\mathbb B_n$}{B n}}\label{SUBSEC:K53-count-local-degree-control}

We next prove the local control needed for the clique-counting replacement argument. The first input is the corresponding near-regularity statement for the number of copies of $K_4^3$ containing a fixed vertex.

\begin{lemma}\label{LEMMA:K53-count-kappa-regularity}
    Suppose that $\mathcal H$ is a $K_5^3$-free $n$-vertex $3$-graph such that
    \begin{align*}
        \Psi(\mathcal H)=\max\{\Psi(\mathcal G):\mathcal G \text{ is a } K_5^3\text{-free }3\text{-graph on } n \text{ vertices}\}.
    \end{align*}
    For each $v\in V(\mathcal H)$, let $\kappa_{\mathcal H}(v)$ denote the number of copies of $K_4^3$ in $\mathcal H$ containing $v$. Then for every vertex $v\in V(\mathcal H)$,
    \begin{align*}
        \left|\kappa_{\mathcal H}(v)-\frac{4\mathrm N(K_4^3,\mathcal H)}{n}\right|\le 2n^2.
    \end{align*}
\end{lemma}

\begin{proof}[Proof of Lemma~\ref{LEMMA:K53-count-kappa-regularity}]
    Fix distinct vertices $u,v\in V(\mathcal H)$. Let $\mathcal H_{u\to v}$ be the $3$-graph obtained from $\mathcal H$ by deleting all edges containing $u$, and then adding $e\cup\{u\}$ for every pair $e\in\binom{V(\mathcal H)\setminus\{u,v\}}2$ such that $e\cup\{v\}\in\mathcal H$. This operation preserves $K_5^3$-freeness. Indeed, no edge of $\mathcal H_{u\to v}$ contains both $u$ and $v$, and any copy of $K_5^3$ using $u$ would give a copy of $K_5^3$ in $\mathcal H$ after replacing $u$ with $v$.

    All edges and copies of $K_4^3$ avoiding $u$ remain present. Moreover, every edge and every copy of $K_4^3$ counted at $v$ whose vertex set avoids $u$ gives, after replacing $v$ with $u$, the corresponding edge or copy of $K_4^3$ in $\mathcal H_{u\to v}$. The number of edges counted by $d_{\mathcal H}(v)$ that also contain $u$ is at most $n$, and the number of copies of $K_4^3$ counted by $\kappa_{\mathcal H}(v)$ that also contain $u$ is at most $\binom n2$. Hence
    \begin{align*}
        \Psi(\mathcal H_{u\to v})
        \ge \Psi(\mathcal H)-d_{\mathcal H}(u)-\kappa_{\mathcal H}(u)+d_{\mathcal H}(v)+\kappa_{\mathcal H}(v)-n^2.
    \end{align*}
    By the extremality of $\mathcal H$, we have $\Psi(\mathcal H)\ge \Psi(\mathcal H_{u\to v})$. Since $0\le d_{\mathcal H}(x)\le\binom{n-1}{2}$ for every vertex $x\in V(\mathcal H)$, it follows that
    \begin{align*}
        \kappa_{\mathcal H}(u)\ge \kappa_{\mathcal H}(v)-2n^2.
    \end{align*}
    Interchanging $u$ and $v$ gives $|\kappa_{\mathcal H}(u)-\kappa_{\mathcal H}(v)|\le 2n^2$. Finally, $\sum_{x\in V(\mathcal H)}\kappa_{\mathcal H}(x)=4\mathrm N(K_4^3,\mathcal H)$, and the lemma follows.
\end{proof}

The next lemma is the clique-counting analogue of Lemma~\ref{LEMMA:K53-L2-local-degree-control}. It shows that, once an extremal graph is known to be close to $\mathbb B_n$, no vertex is incident to many bad edges or many missing edges with respect to a maximum partition.

\begin{lemma}\label{LEMMA:K53-count-local-degree-control}
    Let $\eta>0$ be sufficiently small. There exist $\delta_{\ref{LEMMA:K53-count-local-degree-control}}\coloneqq\delta_{\ref{LEMMA:K53-count-local-degree-control}}(\eta)>0$ and $N_{\ref{LEMMA:K53-count-local-degree-control}}\coloneqq N_{\ref{LEMMA:K53-count-local-degree-control}}(\eta)$ such that the following holds for all $n\ge N_{\ref{LEMMA:K53-count-local-degree-control}}$.
    Suppose that $\mathcal H$ is an $n$-vertex $K_5^3$-free $3$-graph with
    \begin{align*}
        \Psi(\mathcal H)=\max\{\Psi(\mathcal G):\mathcal G \text{ is a } K_5^3\text{-free }3\text{-graph on } n \text{ vertices}\}.
    \end{align*}
    Suppose further that $\mathcal H$ is $\delta$-close to $\mathbb B_n$ for some $0<\delta\le\delta_{\ref{LEMMA:K53-count-local-degree-control}}$.
    Let $V_1\cup V_2=V(\mathcal H)$ be a maximum partition of $\mathcal H$, and let $\mathcal B$ and $\mathcal M$ be the bad and missing edges with respect to this partition. Then
    \begin{align*}
        \max\{\Delta(\mathcal B),\Delta(\mathcal M)\}\le 2\eta n^2.
    \end{align*}
\end{lemma}

\begin{proof}[Proof of Lemma~\ref{LEMMA:K53-count-local-degree-control}]
    Choose $\alpha>0$ sufficiently small so that $\alpha\le \frac{\eta}{100}$ and $\alpha\le \frac{\varepsilon_{\ref{PROP:2-colored-K4-triangle-stability-sym}}(\eta/8)}{100}$.
    Choose $0<\delta_{\ref{LEMMA:K53-count-local-degree-control}}<1$ sufficiently small so that Lemma~\ref{LEMMA:common-near-bipartite-setup} is applicable with parameter $\delta_{\ref{LEMMA:K53-count-local-degree-control}}$, and so that
    \begin{align*}
        \delta_{\ref{LEMMA:K53-count-local-degree-control}}\le \delta_{\ref{LEMMA:common-balanced-k4-missing}}(\alpha),\qquad \delta_{\ref{LEMMA:K53-count-local-degree-control}}\le \frac{\alpha}{2},\qquad 50\delta_{\ref{LEMMA:K53-count-local-degree-control}}^{1/2}+5\alpha\le \min\left\{\frac{\eta}{4},\frac{\varepsilon_{\ref{PROP:2-colored-K4-triangle-stability-sym}}(\eta/8)}{16}\right\}.
    \end{align*}
    Choose $N_{\ref{LEMMA:K53-count-local-degree-control}}$ sufficiently large so that all results used below apply.

    Let $n\ge N_{\ref{LEMMA:K53-count-local-degree-control}}$, and let $\mathcal H$ satisfy the assumptions. Since $\mathbb B_n$ is $K_5^3$-free, the extremality of $\mathcal H$ gives $\Psi(\mathcal H)\ge \Psi(\mathbb B_n)$. Hence
    \begin{align*}
        \mathrm N(K_4^3,\mathcal H)
        =\Psi(\mathcal H)-|\mathcal H|
        \ge \Psi(\mathbb B_n)-|\mathcal H|
        \ge \mathrm N(K_4^3,\mathbb B_n)+|\mathbb B_n|-\binom n3.
    \end{align*}

    Since $\mathcal H$ is $\delta$-close to $\mathbb B_n$ and $\delta\le\delta_{\ref{LEMMA:K53-count-local-degree-control}}$, it is also $\delta_{\ref{LEMMA:K53-count-local-degree-control}}$-close to $\mathbb B_n$. Applying Lemma~\ref{LEMMA:common-near-bipartite-setup} to the maximum partition $V_1\cup V_2$ gives
    \begin{align}
        \max\{|\mathcal B|,|\mathcal M|\}\le 2\delta_{\ref{LEMMA:K53-count-local-degree-control}} n^3,\qquad \left||V_i|-n/2\right|\le 10\delta_{\ref{LEMMA:K53-count-local-degree-control}}^{1/2}n\text{ for each }i\in[2]. \label{equ:K53-count-local-degree-setup}
    \end{align}
    The lower bound above and Lemma~\ref{LEMMA:K53-count-kappa-regularity} show that, for every vertex $v\in V(\mathcal H)$,
    \begin{align}
        \kappa_{\mathcal H}(v)
        \ge \frac{4\mathrm N(K_4^3,\mathcal H)}{n}-2n^2 
        \ge \frac{4}{n}\left(\mathrm N(K_4^3,\mathbb B_n)+|\mathbb B_n|-\binom n3\right)-2n^2 
        \ge \left(\frac{1}{16}-\alpha\right)n^3. \label{equ:K53-count-local-degree-kappa-lower}
    \end{align}
    Here the last inequality follows from the choice of $N_{\ref{LEMMA:K53-count-local-degree-control}}$.

    Fix $v\in V(\mathcal H)$. By symmetry, assume that $v\in V_1$, and put $L\coloneqq L_{\mathcal H}(v)$. By Fact~\ref{FACT:common-maximum-partition-link}, we have $|L[V_2]|\ge |L[V_1]|$. 
    Since $\delta\le\delta_{\ref{LEMMA:K53-count-local-degree-control}}\le \delta_{\ref{LEMMA:common-balanced-k4-missing}}(\alpha)$, Lemma~\ref{LEMMA:common-balanced-k4-missing} gives a set of at most $\alpha n^2$ edges whose deletion makes $L$ balanced $K_4$-free with respect to $V_1\cup V_2$. Choose subsets $V_1'\subseteq V_1\setminus\{v\}$ and $V_2'\subseteq V_2$ with common size $m$ by deleting vertices from the larger of $V_1\setminus\{v\}$ and $V_2$. By~\eqref{equ:K53-count-local-degree-setup}, at most $22\delta_{\ref{LEMMA:K53-count-local-degree-control}}^{1/2}n$ vertices are deleted. Let $G$ be the balanced $K_4$-free graph on $V_1'\cup V_2'$ obtained from $L[V_1'\cup V_2']$ by deleting at most $\alpha n^2$ edges.

    Every copy of $K_4^3$ containing $v$ gives a triangle in $L$. If this triangle is not crossing with respect to $V_1\cup V_2$, then the other three vertices form a bad edge of $\mathcal H$. Therefore~\eqref{equ:K53-count-local-degree-setup} and~\eqref{equ:K53-count-local-degree-kappa-lower} imply
    \begin{align*}
        \mathcal N_{\mathrm{cr}}(K_3,L)\ge \kappa_{\mathcal H}(v)-|\mathcal B|\ge \left(\frac{1}{16}-2\alpha\right)n^3.
    \end{align*}
    Passing from $L$ to $G$ deletes at most $22\delta_{\ref{LEMMA:K53-count-local-degree-control}}^{1/2}n$ vertices and at most $\alpha n^2$ edges, so
    \begin{align*}
        \mathcal N_{\mathrm{cr}}(K_3,G)\ge \left(\frac{1}{16}-3\alpha-22\delta_{\ref{LEMMA:K53-count-local-degree-control}}^{1/2}\right)n^3.
    \end{align*}
    Also $m\le n/2$, while $m\ge (\frac{1}{2}-11\delta_{\ref{LEMMA:K53-count-local-degree-control}}^{1/2})n$ for $n$ sufficiently large. By the choices of $\alpha$, $\delta_{\ref{LEMMA:K53-count-local-degree-control}}$, and $N_{\ref{LEMMA:K53-count-local-degree-control}}$, this gives
    \begin{align*}
        \mathcal N_{\mathrm{cr}}(K_3,G)\ge m\binom m2-\varepsilon_{\ref{PROP:2-colored-K4-triangle-stability-sym}}(\eta/8)m^3.
    \end{align*}
    Proposition~\ref{PROP:2-colored-K4-triangle-stability-sym}, applied to $G$ with parameter $\eta/8$ and part size $m$, gives that $G$ can be transformed into either $\Lambda[V_1';V_2']$ or $\Lambda[V_2';V_1']$ by changing at most $\eta m^2/8\le \eta n^2/8$ edges.

    The second possibility cannot occur. Indeed, if $G$ could be transformed into $\Lambda[V_2';V_1']$ by changing at most $\eta n^2/8$ edges, then $|G[V_1']|\ge \binom m2-\eta n^2/8$ and $|G[V_2']|\le \eta n^2/8$. Returning from $G$ to $L$ gives
    \begin{align*}
        |L[V_1]|\ge \binom m2-\eta n^2/8
        \qquad\text{and}\qquad
        |L[V_2]|\le \left(\frac{\eta}{8}+\alpha+22\delta_{\ref{LEMMA:K53-count-local-degree-control}}^{1/2}\right)n^2.
    \end{align*}
    Since $m\ge n/3$ for $n$ sufficiently large and $\eta$ is sufficiently small, the two inequalities force $|L[V_1]|>|L[V_2]|$, contradicting the maximality of the partition. Hence $G$ can be transformed into $\Lambda[V_1';V_2']$ by changing at most $\eta n^2/8$ edges.

    The edit distance from $G$ to $\Lambda[V_1';V_2']$ contributes at most $\eta n^2/8$ to the corresponding three-term sum. Passing from $G$ back to $L[V_1'\cup V_2']$ restores at most $\alpha n^2$ deleted edges, and restoring the vertices outside $V_1'\cup V_2'$ contributes at most $22\delta_{\ref{LEMMA:K53-count-local-degree-control}}^{1/2}n^2$. Therefore, by the choice of parameters,
    \begin{align*}
        |L[V_1]|+\bigl((|V_1|-1)|V_2|-|L[V_1,V_2]|\bigr)+\left(\binom{|V_2|}{2}-|L[V_2]|\right)
        \le \left(\frac{\eta}{8}+\alpha+22\delta_{\ref{LEMMA:K53-count-local-degree-control}}^{1/2}\right)n^2
        \le 2\eta n^2.
    \end{align*}
    The first term is $d_{\mathcal B}(v)$, while the last two terms sum to $d_{\mathcal M}(v)$. Thus $d_{\mathcal B}(v)\le 2\eta n^2$ and $d_{\mathcal M}(v)\le 2\eta n^2$. The same argument applies when $v\in V_2$, and therefore $\max\{\Delta(\mathcal B),\Delta(\mathcal M)\}\le 2\eta n^2$.
\end{proof}

\subsection{The local replacement step}\label{SUBSEC:K53-count-local-replacement}

We first record the bipartite calculation needed at the end of the exact argument. It says that among all $n$-vertex bipartite $3$-graphs, the complete balanced bipartite construction has maximum $\Psi$-value.

\begin{lemma}\label{LEMMA:K53-count-bipartite-max}
    Let $n\ge4$, and let $\mathcal G$ be an $n$-vertex bipartite $3$-graph. Then
    \begin{align*}
        \Psi(\mathcal G)\le \Psi(\mathbb B_n),
    \end{align*}
    and equality holds if and only if $\mathcal G\cong\mathbb B_n$.
\end{lemma}

\begin{proof}[Proof of Lemma~\ref{LEMMA:K53-count-bipartite-max}]
    Suppose that $\mathcal G$ is an $n$-vertex bipartite $3$-graph with bipartition $V_1\cup V_2=V(\mathcal G)$. Since adding any missing edge increases $|\mathcal G|$ and cannot decrease the number of copies of $K_4^3$, we have
    \begin{align*}
        \Psi(\mathcal G)\le \Psi(\mathbb B[V_1,V_2]),
    \end{align*}
    and equality in this inequality is possible only when $\mathcal G=\mathbb B[V_1,V_2]$.

    Put $m\coloneqq |V_1||V_2|$. In $\mathbb B[V_1,V_2]$, the edges are precisely the triples meeting both parts, and copies of $K_4^3$ are precisely the four-sets with two vertices in each part. Hence
    \begin{align*}
        \Psi(\mathbb B[V_1,V_2])=\frac{m(n-2)}2+\binom{|V_1|}{2}\binom{|V_2|}{2}=\frac{m(m+n-3)}4.
    \end{align*}
    For $n\ge4$, the expression on the right is strictly increasing in $m$. Since $m\le \lfloor n^2/4\rfloor$, with equality if and only if the bipartition is balanced, we have $\Psi(\mathbb B[V_1,V_2])\le\Psi(\mathbb B_n)$. Equality holds if and only if $\mathcal G=\mathbb B[V_1,V_2]$ and the bipartition is balanced. This proves Lemma~\ref{LEMMA:K53-count-bipartite-max}.
\end{proof}

We also need the local replacement estimate. The next lemma is the clique-counting analogue of Lemma~\ref{LEMMA:K53-L2-switching}.

\begin{lemma}\label{LEMMA:K53-count-switching}
    There exists $\xi_{\ref{LEMMA:K53-count-switching}}>0$ such that for every $0<\xi\le \xi_{\ref{LEMMA:K53-count-switching}}$ there exists $N_{\ref{LEMMA:K53-count-switching}}\coloneqq N_{\ref{LEMMA:K53-count-switching}}(\xi)$ such that the following holds for all $n\ge N_{\ref{LEMMA:K53-count-switching}}$.
    Let $\mathcal G$ be an $n$-vertex $3$-graph with a bipartition $V_1\cup V_2=V(\mathcal G)$, and let $\mathcal B$ and $\mathcal M$ be the bad and missing edges with respect to this partition.
    Let $e_\ast\in\partial\mathcal G$ be a pair in the shadow such that $e_\ast\in\overline K[V_1,V_2]$. Suppose that $e_\ast$ satisfies
    \begin{enumerate}[label=(\roman*), ref=(\roman*)]
        \item\label{LEMMA:K53-count-switching-balance} $\bigl||V_i|-n/2\bigr|\le \xi n$ for $i\in[2]$,
        \item\label{LEMMA:K53-count-switching-degree} $\max\{\Delta(\mathcal B),\Delta(\mathcal M)\}\le \xi n^2$,
        \item\label{LEMMA:K53-count-switching-heavy} $d_{\mathcal M}(e_\ast)\ge 47\xi^{1/2}n$,
        \item\label{LEMMA:K53-count-switching-comparison} $d_{\mathcal M}(e_\ast)\ge d_{\mathcal B}(e_\ast)-\xi n$.
    \end{enumerate}
    Let $\mathcal G^\ast\coloneqq(\mathcal G\setminus\mathcal B(e_\ast))\cup\mathcal M(e_\ast)$. Then
    \begin{align*}
        \Psi(\mathcal G^\ast)>\Psi(\mathcal G).
    \end{align*}
\end{lemma}

\begin{proof}[Proof of Lemma~\ref{LEMMA:K53-count-switching}]
    Choose $\xi_{\ref{LEMMA:K53-count-switching}}>0$ sufficiently small. Fix $0<\xi\le \xi_{\ref{LEMMA:K53-count-switching}}$, and choose $N_{\ref{LEMMA:K53-count-switching}}$ sufficiently large.
    By symmetry, assume that $e_\ast=\{u_1,u_2\}\subseteq V_1$.
    Put
    \begin{align*}
        X\coloneqq N_{\mathcal M}(e_\ast),\qquad Y\coloneqq N_{\mathcal B}(e_\ast).
    \end{align*}
    Then $X\subseteq V_2$, $Y\subseteq V_1$, and assumptions~\ref{LEMMA:K53-count-switching-heavy} and~\ref{LEMMA:K53-count-switching-comparison} give
    \begin{align}
        |X|\ge 47\xi^{1/2}n\qquad\text{and}\qquad |Y|\le |X|+\xi n. \label{equ:K53-count-switching-mb}
    \end{align}

    Let $\mathcal T^+$ be the family of copies of $K_4^3$ which are present in $\mathcal G^\ast$ but not in $\mathcal G$, and let $\mathcal T^-$ be the family of copies of $K_4^3$ which are present in $\mathcal G$ but not in $\mathcal G^\ast$.
    Since only edges containing $e_\ast$ are changed, every copy in $\mathcal T^+\cup\mathcal T^-$ contains $e_\ast$.

    Define a graph $G_\ast$ on $V_2$ by declaring $xx'\in G_\ast$ if and only if $u_1xx',u_2xx'\in\mathcal G$.
    The graph $G_\ast$ is unchanged when passing from $\mathcal G$ to $\mathcal G^\ast$.
    Every copy in $\mathcal T^+$ contains an added edge $u_1u_2x$ with $x\in X$, and its fourth vertex must lie in $V_2$, since every triple of the form $u_1u_2y$ with $y\in V_1$ is absent from $\mathcal G^\ast$.
    Thus every copy in $\mathcal T^+$ is of the form $\{u_1,u_2,x,x'\}$, where $xx'\in G_\ast$ and $\{x,x'\}\cap X\ne\emptyset$.
    Hence
    \begin{align*}
        |\mathcal T^+|=\sum_{x\in X}d_{G_\ast}(x)-|G_\ast[X]|.
    \end{align*}
    For each $x\in X$, a vertex $x'\in V_2\setminus\{x\}$ fails to be adjacent to $x$ in $G_\ast$ only if one of $u_1xx'$ and $u_2xx'$ lies in $\mathcal M$. Each missing edge counted in the following sum contains $u_1$ or $u_2$, and is counted at most twice. Therefore
    \begin{align*}
        \sum_{x\in X}d_{G_\ast}(x)
        &\ge |X|(|V_2|-1)-\sum_{x\in X}\bigl(d_{\mathcal M}(u_1x)+d_{\mathcal M}(u_2x)\bigr)\\
        &\ge |X|(|V_2|-1)-2d_{\mathcal M}(u_1)-2d_{\mathcal M}(u_2)
        \ge |X|(|V_2|-1)-4\xi n^2,
    \end{align*}
    where the last inequality uses assumption~\ref{LEMMA:K53-count-switching-degree}.
    Since $|G_\ast[X]|\le\binom{|X|}{2}$, we obtain
    \begin{align}
        |\mathcal T^+|
        \ge |X|(|V_2|-1)-4\xi n^2 - \binom{|X|}{2}
        = |X||V_2|-\frac{|X|(|X|+1)}2-4\xi n^2. \label{equ:K53-count-switching-created}
    \end{align}

    We now bound the copies destroyed by the replacement.
    Every copy in $\mathcal T^-$ contains a deleted edge $u_1u_2y$ with $y\in Y$, and so it either has three vertices in $V_1$ and one in $V_2$, or it is contained in $V_1$.
    If a copy in $\mathcal T^-$ has three vertices in $V_1$ and one in $V_2$, then it is of the form $\{u_1,u_2,y,z\}$ with $y\in Y$ and $z\in V_2\setminus X$.
    Thus the number of such copies is at most $|Y|(|V_2|-|X|)$.
    If a copy in $\mathcal T^-$ is contained in $V_1$, then it has vertex set $\{u_1,u_2,y,y'\}$ with $y,y'\in Y$ and $u_1yy'\in\mathcal B$. The edge $u_1yy'$ determines the copy, and hence assumption~\ref{LEMMA:K53-count-switching-degree} gives at most $\xi n^2$ such copies.
    Consequently,
    \begin{align}
        |\mathcal T^-|\le |Y|(|V_2|-|X|)+\xi n^2. \label{equ:K53-count-switching-destroyed}
    \end{align}

    Combining~\eqref{equ:K53-count-switching-mb}, \eqref{equ:K53-count-switching-created}, and~\eqref{equ:K53-count-switching-destroyed}, and using $|V_2|\le n$, we get
    \begin{align*}
        |\mathcal T^+|-|\mathcal T^-|
        &\ge |X||V_2|-\frac{|X|(|X|+1)}2-4\xi n^2-(|X|+\xi n)(|V_2|-|X|)-\xi n^2\\
        &=\frac{|X|^2-|X|}{2}-\xi n(|V_2|-|X|)-5\xi n^2\\
        &\ge \frac{|X|^2-|X|}{2}-6\xi n^2
        \ge \xi n^2,
    \end{align*}
    where the last two inequalities follow from $|X|\ge 47\xi^{1/2}n$, $|X|\le n$, and the choice of $N_{\ref{LEMMA:K53-count-switching}}$.
    Therefore
    \begin{align*}
        \mathrm N(K_4^3,\mathcal G^\ast)-\mathrm N(K_4^3,\mathcal G)=|\mathcal T^+|-|\mathcal T^-|\ge \xi n^2.
    \end{align*}
    Also, by~\eqref{equ:K53-count-switching-mb},
    \begin{align*}
        |\mathcal G^\ast|-|\mathcal G|=|X|-|Y|\ge -\xi n.
    \end{align*}
    By the choice of $N_{\ref{LEMMA:K53-count-switching}}$, we get
    \begin{align*}
        \Psi(\mathcal G^\ast)-\Psi(\mathcal G)
        =|\mathcal G^\ast|-|\mathcal G|+\mathrm N(K_4^3,\mathcal G^\ast)-\mathrm N(K_4^3,\mathcal G)
        \ge \xi n^2-\xi n>0.
    \end{align*}
    This completes the proof of Lemma~\ref{LEMMA:K53-count-switching}.
\end{proof}

\subsection{Proof of Theorem~\ref{THM:cliques-exact-K53}}\label{SUBSEC:Proof-cliques-exact-K53}

We now give the proof of Theorem~\ref{THM:cliques-exact-K53}.

\begin{proof}[Proof of Theorem~\ref{THM:cliques-exact-K53}]
    Let $\xi_{\ref{LEMMA:K53-count-switching}}$ be the constant from Lemma~\ref{LEMMA:K53-count-switching}. Choose $\eta>0$ sufficiently small so that Lemma~\ref{LEMMA:K53-count-local-degree-control} and Lemma~\ref{LEMMA:common-local-bad-missing-control} apply with parameter $\eta$. Set $\xi\coloneqq7\eta^{1/2}$, and assume, by decreasing $\eta$ if necessary, that
    \begin{align*}
        \xi<\xi_{\ref{LEMMA:K53-count-switching}},\qquad 47\xi^{1/2}\le \frac1{10},\qquad 2\eta\le \xi.
    \end{align*}
    Choose $0<\delta_0<1$ sufficiently small so that Lemma~\ref{LEMMA:common-near-bipartite-setup} applies with parameter $\delta_0$, and
    \begin{align*}
        \delta_0\le \delta_{\ref{LEMMA:K53-count-local-degree-control}}(\eta),\qquad \delta_0\le \delta_{\ref{LEMMA:common-local-bad-missing-control}}(\eta),\qquad 10\delta_0^{1/2}\le \xi.
    \end{align*}
    Choose $n_0$ sufficiently large so that all results used below apply and, for all $n\ge n_0$,
    \begin{align*}
        \mathrm N(K_4^3,\mathbb B_n)+|\mathbb B_n|-\binom n3
        \ge \left(\frac1{64}-\delta_{\ref{THM:K53-count-edge-stability}}(\delta_0)\right)n^4.
    \end{align*}

    Let $n\ge n_0$, and let $\mathcal H$ be an $n$-vertex $K_5^3$-free $3$-graph maximizing $\Psi$. Since $\mathbb B_n$ is $K_5^3$-free, the extremality of $\mathcal H$ gives $\Psi(\mathcal H)\ge \Psi(\mathbb B_n)$, and hence
    \begin{align*}
        \mathrm N(K_4^3,\mathcal H)
        =\Psi(\mathcal H)-|\mathcal H|
        \ge \Psi(\mathbb B_n)-|\mathcal H|
        \ge \mathrm N(K_4^3,\mathbb B_n)+|\mathbb B_n|-\binom n3
        \ge \left(\frac1{64}-\delta_{\ref{THM:K53-count-edge-stability}}(\delta_0)\right)n^4.
    \end{align*}
    By Theorem~\ref{THM:K53-count-edge-stability}, the graph $\mathcal H$ is $\delta_0$-close to $\mathbb B_n$.

    Let $V_1\cup V_2=V(\mathcal H)$ be a maximum partition of $\mathcal H$, and let $\mathcal B$ and $\mathcal M$ be the corresponding sets of bad and missing edges. Lemma~\ref{LEMMA:common-near-bipartite-setup}, together with the choice of $\delta_0$, gives
    \begin{align}
        \max_{i\in[2]}\left||V_i|-\frac n2\right|\le 10\delta_0^{1/2}n\le \xi n. \label{equ:K53-count-proof-balance}
    \end{align}
    Lemma~\ref{LEMMA:K53-count-local-degree-control} and the choice of $\eta$ give
    \begin{align}
        \max\{\Delta(\mathcal B),\Delta(\mathcal M)\}\le 2\eta n^2\le \xi n^2. \label{equ:K53-count-proof-degree}
    \end{align}
    By Lemma~\ref{LEMMA:common-local-bad-missing-control}, every bad edge $E\in\mathcal B$ contains a pair $e\subseteq E$ with $d_{\mathcal M}(e)\ge n/10$. Moreover, for every $e\in\overline K[V_1,V_2]$,
    \begin{align}
        d_{\mathcal M}(e)\ge d_{\mathcal B}(e)-7\eta^{1/2}n=d_{\mathcal B}(e)-\xi n. \label{equ:K53-count-proof-pair-comparison}
    \end{align}

    Suppose, for a contradiction, that $\mathcal B\ne\emptyset$. Define
    \begin{align*}
        \mathcal I\coloneqq\left\{e\in\overline K[V_1,V_2]:d_{\mathcal M}(e)\ge \frac n{10}\text{ and }\mathcal B(e)\ne\emptyset\right\}=\{e_1,\ldots,e_k\}.
    \end{align*}
    Since every bad edge lies inside one part, the preceding paragraph shows that every bad edge contains some pair from $\mathcal I$.

    Let $\mathcal H_0\coloneqq\mathcal H$, $\mathcal B_0\coloneqq\mathcal B$, and $\mathcal M_0\coloneqq\mathcal M$. For $j\in[k]$, once $\mathcal H_{j-1}$, $\mathcal B_{j-1}$, and $\mathcal M_{j-1}$ have been defined, we proceed as follows. If $\mathcal B_{j-1}(e_j)=\emptyset$, then set $\mathcal H_j\coloneqq\mathcal H_{j-1}$, $\mathcal B_j\coloneqq\mathcal B_{j-1}$, and $\mathcal M_j\coloneqq\mathcal M_{j-1}$. Otherwise set
    \begin{align*}
        \mathcal H_j\coloneqq(\mathcal H_{j-1}\setminus\mathcal B_{j-1}(e_j))\cup\mathcal M_{j-1}(e_j),\qquad \mathcal B_j\coloneqq\mathcal H_j\setminus\mathbb B[V_1,V_2],\qquad \mathcal M_j\coloneqq\mathbb B[V_1,V_2]\setminus\mathcal H_j.
    \end{align*}
    This procedure only deletes bad edges and only adds edges of $\mathbb B[V_1,V_2]$, so
    \begin{align*}
        \mathcal B=\mathcal B_0\supseteq\mathcal B_1\supseteq\cdots\supseteq\mathcal B_k\qquad\text{and}\qquad \mathcal M=\mathcal M_0\supseteq\mathcal M_1\supseteq\cdots\supseteq\mathcal M_k.
    \end{align*}
    Every edge of $\mathbb B[V_1,V_2]$ contains exactly one same-side pair. Hence, before step $j$, the additions made at pairs different from $e_j$ do not affect the missing edges containing $e_j$. Consider a step $j$ with $\mathcal B_{j-1}(e_j)\ne\emptyset$. Then $e_j\in\partial\mathcal H_{j-1}$, and
    \begin{align*}
        d_{\mathcal M_{j-1}}(e_j)=d_{\mathcal M}(e_j)\ge \frac n{10}\ge47\xi^{1/2}n,\qquad d_{\mathcal M_{j-1}}(e_j)\ge d_{\mathcal B_{j-1}}(e_j)-\xi n.
    \end{align*}
    The second inequality follows from the equality above,~\eqref{equ:K53-count-proof-pair-comparison}, and the inclusion $\mathcal B_{j-1}\subseteq\mathcal B$. Together with~\eqref{equ:K53-count-proof-balance} and~\eqref{equ:K53-count-proof-degree}, the inclusions $\mathcal B_{j-1}\subseteq\mathcal B$ and $\mathcal M_{j-1}\subseteq\mathcal M$ show that Lemma~\ref{LEMMA:K53-count-switching} applies to $\mathcal H_{j-1}$ with the fixed partition $V_1\cup V_2$ and the pair $e_j$. Hence $\Psi(\mathcal H_j)>\Psi(\mathcal H_{j-1})$ for every step with $\mathcal B_{j-1}(e_j)\ne\emptyset$, while the other steps leave the graph unchanged.
    
    At least one step is nontrivial. Indeed, choose an original bad edge $E\in\mathcal B$ and a pair $e_j\subseteq E$ with $e_j\in\mathcal I$. If $E$ is still present before step $j$, then $\mathcal B_{j-1}(e_j)\ne\emptyset$. Otherwise, $E$ was removed in an earlier nontrivial step. Therefore
    \begin{align*}
        \Psi(\mathcal H_k)>\Psi(\mathcal H).
    \end{align*}
    Since each original bad edge contains a pair from $\mathcal I$, it is removed when one such pair is processed, unless it was removed earlier. Therefore $\mathcal H_k$ has no bad edge, and so $\mathcal H_k\subseteq\mathbb B[V_1,V_2]$. Completing $\mathcal H_k$ to $\mathbb B[V_1,V_2]$ cannot decrease $\Psi$, and Lemma~\ref{LEMMA:K53-count-bipartite-max} gives
    \begin{align*}
        \Psi(\mathcal H)<\Psi(\mathcal H_k)
        \le \Psi(\mathbb B[V_1,V_2])       
        \le \Psi(\mathbb B_n)\le\Psi(\mathcal H),
    \end{align*}
    a contradiction. Hence $\mathcal B=\emptyset$.

    If $\mathcal M\ne\emptyset$, then adding any edge of $\mathcal M$ keeps the graph bipartite and hence $K_5^3$-free. It also increases $|\mathcal H|$ by one and cannot decrease the number of copies of $K_4^3$. This contradicts the extremality of $\mathcal H$. Therefore $\mathcal M=\emptyset$, and so $\mathcal H=\mathbb B[V_1,V_2]$. Finally, Lemma~\ref{LEMMA:K53-count-bipartite-max} gives $\Psi(\mathcal H)\le\Psi(\mathbb B_n)$, while extremality gives the reverse inequality. Hence equality holds in Lemma~\ref{LEMMA:K53-count-bipartite-max}, and so $\mathcal H\cong\mathbb B_n$.

    We have shown that every $n$-vertex $K_5^3$-free $3$-graph maximizing $\Psi$ is isomorphic to $\mathbb B_n$. Therefore every $n$-vertex $K_5^3$-free $3$-graph $\mathcal G$ satisfies $\Psi(\mathcal G)\le\Psi(\mathbb B_n)$, with equality only when $\mathcal G\cong\mathbb B_n$. By~\eqref{equ:K53-count-clique-decomposition}, the same statement holds for $k(\mathcal G)$. This completes the proof of Theorem~\ref{THM:cliques-exact-K53}.
\end{proof}

\section{Concluding remarks}\label{SEC:Concluding-remarks}

We give an application of Theorem~\ref{THM:cliques-exact-K53} to Tur\'{a}n problems for simplicial complexes. 
A simplicial complex $\mathcal{H} = (V,\mathcal{E})$ consists of a vertex set $V$ and a family $\mathcal{E} \subseteq 2^V$ that contains the empty set and all singletons, and is closed under taking subsets. 
Given a simplicial complex $\mathcal{F}$, let $\mathrm{ex}^{\mathrm{simp}}(n,\mathcal{F})$ denote the maximum number of faces in an $n$-vertex $\mathcal{F}$-free simplicial complex, with the empty set counted as a face. 
For integers $t > r \ge 2$, denote by $\mathcal K(t,r)$ the simplicial complex given by the downward closure of $K_t^r$.

By the results in~\cite{AGLP25,CPS23}, it follows that
\begin{align*}
    \mathrm{ex}^{\mathrm{simp}}(n,\mathcal K(t,r)) =
    \max\bigl\{k(\mathcal G):
    \mathcal G\text{ is an }n\text{-vertex }K_t^r\text{-free }r\text{-graph}\bigr\},
\end{align*}
where $k(\mathcal{G})$ counts all cliques in $\mathcal{G}$, including all vertex subsets of size less than $r$.
Therefore, Theorem~\ref{THM:cliques-exact-K53} gives the following corollary.

\begin{corollary}\label{CORo:K53-simplicial-Turan}
    There exists an integer $n_0$ such that for all $n\ge n_0$,
    \begin{align*}
        \mathrm{ex}^{\mathrm{simp}}(n,\mathcal K(5,3))
        = k(\mathbb B_n).
    \end{align*}
    Moreover, up to isomorphism, the unique extremal simplicial complex consists of all cliques in $\mathbb B_n$.
\end{corollary}
\section*{Acknowledgments}\label{SEC:Acknowledgments}
\begin{sloppypar}
We are grateful to Dingyuan Liu for pointing out the application of Theorem~\ref{THM:cliques-exact-K53} to Tur\'{a}n problems for simplicial complexes. 
J.H. was supported by the National Key R\&D Program of China (No.~2023YFA1010202). 
X.L. was supported by the Excellent Young Talents Program (Overseas) of the National Natural Science Foundation of China. 
Y.Z. was supported by the Doctoral Student Program of the Young S\&T Talents Cultivation Project, CAST. 
\end{sloppypar}

\bibliographystyle{plain}
\bibliography{ColoredTuran}
\end{document}